\documentclass[12pt,final]{amsart}
\usepackage{xspace,amssymb,amsfonts,epsfig,marvosym,euscript,eufrak,xypic,enumerate,amsmath,nicefrac}
\usepackage{tikz,etex,bbm,xspace}
\usepackage[colorlinks=true, pdfstartview=FitV, linkcolor=blue, citecolor=blue, urlcolor=blue]{hyperref}%
 \usepackage{epstopdf,ifpdf}
\usepackage{pxfonts,multicol,todonotes}
\usepackage[margin=3cm,footskip=25pt,headheight=20pt]{geometry}
\usepackage{fancyhdr,mathrsfs}
\usepackage[polutonikogreek,english]{babel}
\begin{document}\newtheoremstyle{all}
  {11pt}
  {11pt}
  {\slshape}
  {}
  {\bfseries}
  {}
  {.5em}
  {}
\newtheoremstyle{rem}
  {11pt}
  {11pt}
  {}
  {}
  {\bfseries}
  {}
  {.5em}
  {}

\theoremstyle{all}
\newtheorem{thm}{Theorem}[section]
\newtheorem{prop}[thm]{Proposition}
\newtheorem{cor}[thm]{Corollary}
\newtheorem{lemma}[thm]{Lemma}
\newtheorem{defn}[thm]{Definition}
\newtheorem*{defn*}{Definition}
\newtheorem{ques}[thm]{Question}
\newtheorem{conj}[thm]{Conjecture}
\theoremstyle{rem}
\newtheorem{rem}[thm]{Remark}
\newtheorem{example}[thm]{Example}

 
\newcommand{\nc}{\newcommand}
\newcommand{\renc}{\renewcommand}
  \nc{\kac}{\kappa^C}
\nc{\bcom}{\todo[inline,color=red!40]}
\nc{\alg}{T}
\nc{\Lco}{L_{\la}}
\nc{\qD}{q^{\nicefrac 1D}}
\nc{\ocL}{M_{\la}}
\nc{\excise}[1]{}
\nc{\Dbe}{D^{\uparrow}}
\nc{\tr}{\operatorname{tr}}
\newcommand{\Mirkovic}{Mirkovi\'c\xspace}
\nc{\tla}{\mathsf{t}_\la}
\nc{\llrr}{\langle\la,\rho\rangle}
\nc{\lllr}{\langle\la,\la\rangle}
\nc{\K}{\mathbbm{k}}
\nc{\Stosic}{Sto{\v{s}}i{\'c}\xspace}
\nc{\cd}{\mathcal{D}}
\nc{\vd}{\mathbb{D}}
\nc{\sse}{{\mathsf{e}}}
\nc{\ssf}{{\mathsf{f}}}
\nc{\tsse}{\tilde{\mathsf{e}}}
\nc{\tssf}{\tilde{\mathsf{f}}}
\nc{\R}{\mathbb{R}}
\renc{\wr}{\operatorname{wr}}
  \nc{\Lam}[3]{\La^{#1}_{#2,#3}}
  \nc{\Lab}[2]{\La^{#1}_{#2}}
  \nc{\Lamvwy}{\Lam\Bv\Bw\By}
  \nc{\Labwv}{\Lab\Bw\Bv}
  \nc{\nak}[3]{\mathcal{N}(#1,#2,#3)}
  \nc{\hw}{highest weight\xspace}
  \nc{\al}{\alpha}
\nc{\thetitle}{\Mirkovic -Vilonen polytopes and
  Khovanov-Lauda-Rouquier algebras}
  \nc{\be}{\beta}
  \nc{\bM}{\mathbf{m}}
  \nc{\bkh}{\backslash}
  \nc{\Bi}{\mathbf{i}}
  \nc{\Bj}{\mathbf{j}}
   \nc{\Bk}{\mathbf{k}}
\nc{\RAA}{R^\A_A}
  \nc{\Bv}{\mathbf{v}}
  \nc{\Bw}{\mathbf{w}}
\nc{\Id}{\operatorname{Id}}
  \nc{\By}{\mathbf{y}}
\nc{\eE}{\EuScript{E}}
  \nc{\Bz}{\mathbf{z}}
  \nc{\coker}{\mathrm{coker}\,}
  \nc{\C}{\mathbb{C}}
  \nc{\ch}{\mathrm{ch}}
  \nc{\de}{\delta}
  \nc{\ep}{\epsilon}
  \nc{\Rep}[2]{\mathsf{Rep}_{#1}^{#2}}
  \nc{\Ev}[2]{E_{#1}^{#2}}
  \nc{\fr}[1]{\mathfrak{#1}}
  \nc{\fp}{\fr p}
  \nc{\fq}{\fr q}
  \nc{\fl}{\fr l}
  \nc{\fgl}{\fr{gl}}
 \nc{\fh}{\fr{h}}

\nc{\coa}{\mathbbm{a}}
\nc{\rad}{\operatorname{rad}}
\nc{\ind}{\operatorname{ind}}
  \nc{\GL}{\mathrm{GL}}
  \nc{\Hom}{\mathrm{Hom}}
  \nc{\im}{\mathrm{im}\,}
  \nc{\La}{\Lambda}
  \nc{\la}{\lambda}
  \nc{\mult}{b^{\mu}_{\la_0}\!}
  \nc{\mc}[1]{\mathcal{#1}}
  \nc{\om}{\omega}
\nc{\gl}{\mathfrak{gl}}
  \nc{\cF}{\mathcal{F}}
 \nc{\cC}{\mathcal{C}}
  \nc{\Vect}{\mathsf{Vect}}
 \nc{\modu}{\mathsf{mod}}
  \nc{\qvw}[1]{\La(#1 \Bv,\Bw)}
  \nc{\van}[1]{\nu_{#1}}
  \nc{\Rperp}{R^\vee(X_0)^{\perp}}
  \nc{\si}{\sigma}
  \nc{\croot}[1]{\al^\vee_{#1}}
\nc{\di}{\mathbf{d}}
  \nc{\SL}[1]{\mathrm{SL}_{#1}}
  \nc{\Th}{\theta}
  \nc{\vp}{\varphi}
  \nc{\wt}{\mathrm{wt}}
  \nc{\Z}{\mathbb{Z}}
  \nc{\Znn}{\Z_{\geq 0}}
  \nc{\ver}{\EuScript{V}}
  \nc{\Res}[2]{\operatorname{Res}^{#1}_{#2}}
  \nc{\edge}{\EuScript{E}}
  \nc{\Spec}{\mathrm{Spec}}
  \nc{\tie}{\EuScript{T}}
  \nc{\ml}[1]{\mathbb{D}^{#1}}
  \nc{\fQ}{\mathfrak{Q}}
        \nc{\fg}{\mathfrak{g}}
 \nc{\fn}{\mathfrak{n}}
  \nc{\Uq}{U_q(\fg)}
        \nc{\bom}{\boldsymbol{\omega}}
\nc{\bla}{{\underline{\boldsymbol{\la}}}}
\nc{\bmu}{{\underline{\boldsymbol{\mu}}}}
\nc{\bal}{{\boldsymbol{\al}}}
\nc{\bet}{{\boldsymbol{\eta}}}
\nc{\rola}{X}
\nc{\Sai}{\sigma}
\nc{\wela}{Y}
\nc{\fM}{\mathfrak{M}}
\nc{\fX}{\mathfrak{X}}
\nc{\fH}{\mathfrak{H}}
\nc{\fE}{\mathfrak{E}}
\nc{\fF}{\mathfrak{F}}
\nc{\fI}{\mathfrak{I}}
\nc{\qui}[2]{\fM_{#1}^{#2}}
\renc{\cL}{\mathcal{L}}
\nc{\ca}[2]{\fQ_{#1}^{#2}}
\nc{\Q}{\mathbb{Q}}
 \nc{\cat}{\mathcal{V}}
\nc{\cata}{\mathfrak{V}}
\nc{\pil}{{\boldsymbol{\pi}}^L}
\nc{\pir}{{\boldsymbol{\pi}}^R}
\nc{\cO}{\mathcal{O}}
\nc{\Ko}{\text{\Denarius}}
\nc{\Ei}{\fE_i}
\nc{\Fi}{\fF_i}
\nc{\fil}{\mathcal{H}}
\nc{\brr}[2]{\beta^R_{#1,#2}}
\nc{\brl}[2]{\beta^L_{#1,#2}}
\nc{\so}[2]{\EuScript{Q}^{#1}_{#2}}
\nc{\EW}{\mathbf{W}}
\nc{\rma}[2]{\mathbf{R}_{#1,#2}}
\nc{\Dif}{\EuScript{D}}
\nc{\MDif}{\EuScript{E}}
\renc{\mod}{\mathsf{mod}}
\nc{\modg}{\mathsf{mod}^g}
\nc{\nmod}{\mathsf{mod}^{fd}}
\nc{\id}{\operatorname{id}}
\nc{\DR}{\mathbf{DR}}
\nc{\End}{\operatorname{End}}
\nc{\Fun}{\operatorname{Fun}}
\nc{\Ext}{\operatorname{Ext}}
\nc{\tw}{\tau}
\nc{\A}{\EuScript{A}}
\nc{\Loc}{\mathsf{Loc}}
\nc{\eF}{\EuScript{F}}
\nc{\LAA}{\Loc^{\A}_{A}}
\nc{\perv}{\mathsf{Perv}}
\nc{\gfq}[2]{_{#1}^{#2}}
\nc{\qgf}[1]{A_{#1}}
\nc{\qgr}{\qgf\rho}
\nc{\tqgf}{\tilde A}
\nc{\Tr}{\operatorname{Tr}}
\nc{\Tor}{\operatorname{Tor}}
\nc{\cQ}{\mathcal{Q}}
\nc{\val}{\operatorname{val}}
\nc{\st}[1]{\Delta(#1)}
\nc{\cst}[1]{\nabla(#1)}
\nc{\ei}{\mathbf{e}_i}
\nc{\Be}{\mathbf{e}}
\nc{\Hck}{\mathfrak{H}}
\nc{\e}{{\tilde{e}}}
\nc{\f}{{\tilde{f}}}
\renc{\P}{\mathbb{P}}
\nc{\cI}{\mathcal{I}}
\nc{\coe}{\mathfrak{K}}
\nc{\pr}{\operatorname{pr}}
\nc{\bra}{\mathfrak{B}}
\nc{\rcl}{\rho^\vee(\la)}
\nc{\tU}{\mathcal{U}}
\nc{\RHom}{\mathrm{RHom}}
\nc{\tcO}{\tilde{\cO}}
\renc{\top}{\mathsf{top}}
\nc{\MV}{\mathcal{M\hspace{-.85mm}V}}
\nc{\QH}{\mathcal{K\!\!L\!R}}
\nc{\KLR}{\mathcal{K\!\!L\!R}}
\nc{\sfP}{\mathsf{P}}
\nc{\rk}{\operatorname{rk}}
\nc{\bpi}{\boldsymbol{\pi}}
\nc{\brho}{\boldsymbol{\rho}}
\nc{\bxi}{\boldsymbol{\xi}}

\nc{\cosoc}{\operatorname{cosoc}}
\nc{\soc}{\operatorname{soc}}
\nc{\asl}{\widehat{\mathfrak{sl}}}
\nc{\eps}{\varepsilon}
\nc{\omu}{\overline{\mu}}
\nc{\oa}{\overline{a}}
\nc{\olambda}{\overline{\lambda}}
\nc{\te}{\tilde{e}}
\nc{\tf}{\tilde{f}}
\newcommand{\bentodo}{\todo[inline,color=blue!20]}
\newcommand{\petertodo}{\todo[inline,color=green!20]}

\nc{\fin}{\text{fin}}
\nc{\scrL}{\mathscr{L}}
\numberwithin{equation}{section}
\renc{\theequation}{\arabic{section}.\arabic{equation}}

\renewcommand{\theenumi}{\roman{enumi}}
\renewcommand{\labelenumi}{(\theenumi)}


\newcommand{\arxiv}[1]{\href{http://arxiv.org/abs/#1}{\tt arXiv:\nolinkurl{#1}}}

\baselineskip=1.1\baselineskip

 \usetikzlibrary{decorations.pathreplacing,backgrounds,decorations.markings}
\tikzset{wei/.style={draw=red,double=red!40!white,double distance=1.5pt,thin}}
\tikzset{bdot/.style={fill,circle,color=blue,inner sep=3pt,outer sep=0}}
\begin{center}
\noindent {\large  \bf \Mirkovic -Vilonen polytopes and\\
  Khovanov-Lauda-Rouquier algebras}
\bigskip

  \begin{tabular}{c@{\hspace{15mm}}c}
     {\sc\large Peter Tingley}&{\sc\large Ben Webster}\\
 \it Department of Mathematics and Statistics,&   \it Department of Mathematics,\\ 
   \it Loyola University Chicago&  \it 
   University of Virginia\\
{\tt ptingley\,@\,luc.edu}&{\tt bwebster\,@\,virginia.edu}
 \end{tabular}
\vspace{3mm}
\end{center}
\bigskip
\begin{center}
  {\it Dedicated to the memory of Andrei Zelevinsky (1953-2013).}
\end{center}
\bigskip
{\small
\begin{quote}
\noindent {\em Abstract.}
We describe how Mirkovi\'c-Vilonen polytopes arise naturally from the categorification
of Lie algebras using Khovanov-Lauda-Rouquier algebras. 
This gives an explicit description of the unique
crystal isomorphism between simple representations of KLR
algebras and MV polytopes.  

MV polytopes, as defined from the geometry of the affine Grassmannian,
only make sense in finite type. Our construction on the other hand gives a map from the infinity crystal to
polytopes for all symmetrizable Kac-Moody algebras. However, to make
the map injective and have well-defined crystal operators on the
image, we must in general decorate the polytopes with some extra
information. We suggest that the resulting ``KLR polytopes" are the
general-type analogues of MV polytopes.

We give a combinatorial description of the resulting decorated polytopes in all affine cases,
and show that this recovers the affine MV polytopes recently defined
by Baumann, Kamnitzer and the first author in symmetric affine
types. We also briefly discuss the situation beyond affine type.
\end{quote}
}

\renc{\thethm}{\Alph{thm}}

\setcounter{equation}{0}
\tableofcontents

\section*{Introduction}

Let $\fg$ be a complex semi-simple Lie algebra. 
The crystal
$B(-\infty)$ is a combinatorial object associated to the algebra
$U^+(\fg)$.   This crystal has an axiomatic definition, but many explicit
realizations of it have appeared in the literature, and for many purposes it suffices to work with these. Here we consider the relationship between two such realizations:
\begin{enumerate}

\renewcommand{\theenumi}{\arabic{enumi}}
\item the set $B(-\infty)$ is in canonical bijection with the set $\QH$
  of simple gradable modules of Khovanov-Lauda-Rouquier (KLR)
  algebras, and
\item  the set $B(-\infty)$ is in canonical bijection the set $\MV$ of \Mirkovic
  -Vilonen polytopes.
\end{enumerate}

This certainly defines a bijection between $\KLR$ and $\MV$, but does not describe it explicitly. 
One of our main results is a simple description of this
bijection: There is a KLR algebra $R(\nu)$ attached to
each 
positive sum $\nu=\sum a_i\al_i$ of simple roots.  For any two such $\nu_1,\nu_2$, there is a natural inclusion
$R({\nu_1})\otimes R({\nu_2}) \hookrightarrow
R({\nu_1+\nu_2})$. Define the {\bf character polytope} $P_L$ of an $R(\nu)$-module $L$ to be the
convex hull  of the weights $\nu'$ such that $\Res{\nu}{\nu',\nu-\nu'}L\neq 0$.
\begin{thm} \label{th:ft-iso}
  The map $L \rightarrow P_L$
is the unique crystal isomorphism from $\QH$ to $\MV$. 
\end{thm}

We feel Theorem \ref{th:ft-iso} is interesting in its own right, but perhaps more important is the fact that $\QH$ naturally indexes $B(-\infty)$ for {\it any} symmetrizable Kac-Moody algebra.  Thus, one can try
to use the map above to {\it define} \Mirkovic -Vilonen polytopes outside of finite type. 
However, there are pairs of
non-isomorphic simples with the same polytopes; for $\mathfrak{g}=\asl_2$, in the notation of \eqref{eq:fac},
this happens for 
$\cL(2,2)$ and $\cL(2,1,1)$. Thus, the polytopes alone are not enough information to
parametrize $B(-\infty)$.

As suggested by Dunlap \cite{Dunlap10} and developed in \cite{BKT}, this problem can be overcome by decorating the edges of $P_L$ with extra information. In the current setting, 
the most natural data to associate to an edge is a ``semi-cuspidal"
representation of a smaller KLR algebra (see Definition \ref{cuspidal}). 
In complete generality, there are many different semi-cuspidal representations that can decorate a given edge, and we do not know a fully combinatorial description of the resulting object. 

For edges
parallel to real roots it turns out that there is only one possible semi-cuspidal
representation, and so it is safe to leave off the decoration. Thus,  in finite type the decoration is redundant. 

Next consider the case when $\fg$ is affine of rank $r+1$. Then the only non-real roots are multiples of $\delta$, so the only edges of the polytope that must be decorated are those parallel to $\delta$. The semi-cuspidal representations that can be
associated to such an edge are naturally
indexed by an $r$-tuple of partitions (see Corollary \ref{cor:commute}).  
In fact,
we can reduce the amount of information even further: as in
\cite{BKT}, the (possibly degenerate) $r$-faces of $P_L$ parallel to
$\delta$ are naturally indexed by the chamber coweights $\gamma$ of an
underlying finite type root system. 
Denote the face of $P_L$
corresponding to $\gamma$ by $P_L^\gamma$. We in fact decorate $P_L$
with just the data of a partition $\pi^\gamma$ for each chamber
coweight $\gamma$ (see Definition \ref{def:pig}) in such a way that,
for any edge $E$ parallel to $\delta,$
\begin{equation} \label{eq:cc}
E  \quad \text{ is a translate of } \quad 
 \sum_{ \gamma \; : \; E \subset P_L^\gamma} d_{\gamma} |\pi^\gamma| 
 \delta,
\end{equation}
where $d_{\gamma}$ are scalars attached to the facet
defined in Definition \ref{def:d}.
The representation attached to such an edge $E$ is
determined in a natural way by $\{ \pi^\gamma : E \subset P_L^\gamma \}$.

Define an {\bf affine pseudo-Weyl polytope}\footnote{In \cite{BKT},
  the analogous object is called a ``decorated GGMS
  (Gelfand-Goresky-MacPherson-Serganova) polyotpe.'' Since we work
  purely algebraicly without reference to the geometric structures
  studied in \cite{GGMS}, we think it more appropriate to follow the
  usage of \cite{Kamnitzer,BKMV}, and use ``pseudo-Weyl polytope.''} to be a pair
consisting of
\begin{itemize}
\item a polytope
$P$ in the root lattice of $\fg$ with all edges parallel to roots, and 
\item a choice of partition
$\pi^\gamma$ for each chamber coweight $\gamma$ of the underlying finite
type root system which satisfies condition \eqref{eq:cc} for each edge
parallel to $\delta$. 
\end{itemize} 
To each representation $L$ of $R$ we associate its {\bf 
  affine MV polytope} (see Definition \ref{def:aMV}), which is a special decorated affine
pseudo-Weyl polytope.  Let $P^{MV}$ be the set of these decorated polytopes.
We seek a combinatorial
characterization of $P^{MV}$. As in finite type, this can be done in terms of conditions on the 2-faces. 

For every 2-face $F$ of an affine pseudo-Weyl polytope, the roots parallel to $F$ form a rank 2 sub-root
system $\Delta_F$ of either finite or affine type.
If $\Delta_F$ is of affine type, then $F$ generally has two edges parallel to $\delta$, which are of the form
$E_\gamma=F\cap P_\gamma$ and $E_{\gamma'}=F\cap P_{\gamma'}$ for
unique chamber coweights $\gamma,\gamma'$.  One naive guess is that we
would obtain a rank-2 pseudo-Weyl polytope by decorating
these imaginary edges with $\pi^\gamma$ and $\pi^{\gamma'}$, but this fails to satisfy \eqref{eq:cc}, since $E_\gamma$ and
$E_{\gamma'}$ are too long.  Instead, $F$ is the Minkowski sum of the
line segment $ \Big(\sum_{\xi : F \subset
  P^\xi} d_\xi |\pi^\xi|\Big)\delta$ with a
decorated pseudo-Weyl polytope $\tilde F$, obtained by shortening $E_\gamma$ and
$E_{\gamma'}$ and decorating them with $\pi^\gamma$ and
$\pi^{\gamma'}$. We will show that:
      \begin{thm}\label{thmB}
        For $\fg$ an affine Lie algebra, the affine MV polytopes are
        precisely the decorated affine pseudo-Weyl polytopes where every
        2-dimensional face $F$ satisfies
        \begin{itemize}
        \item If $\Delta_F$ is a finite type root system, then $F$ is
          an MV polytope for that root system (i.e. it satisfies the
          tropical Pl\"ucker relations from \cite{Kamnitzer}).
  
        \item If $\Delta_F$ is of affine type, then $\tilde F$ is an
          MV polytope for that rank 2 affine algebra (either $\asl_2$ or
          $A_2^{(2)}$) as defined in \cite{BDKT}.
        \end{itemize}
      \end{thm}

\noindent The description of rank 2 affine MV polytopes in \cite{BDKT}
is combinatorial, so Theorem \ref{thmB} gives a combinatorial
characterization of KLR polytopes in all affine cases.

In \cite{BKT}, analogues of MV polytopes were constructed in all symmetric affine types as decorated Harder-Narasimhan polytopes, and it was shown that these are characterized by their 2-faces. Thus Theorem \ref{thmB} also allows us to understand the relationship between our decorated polytopes and those defined in \cite{BKT}:

\begin{thm} \label{corB}
Assume $\fg$ is of affine type with symmetric Cartan matrix. Fix $b
\in B(-\infty)$ and let $L$ be the corresponding element of $\QH$. The
affine MV
polytope $P_L$ and the decorated Harder-Narasimhan polytope $HN_b$
from \cite{BKT} have identical underlying polytopes. Furthermore, 
for each chamber coweight $\gamma$ in the underlying finite type root
system, the partition $\lambda_\gamma$ decorating $HN_b$ as defined in \cite[Sections 1.5 and 7.6]{BKT} is the transpose of our $\pi^{\gamma}$. 
\end{thm}

It is natural to ask for an intrinsic
characterization of the polytopes $P_L$ in the general Kac-Moody
case. We do not even have a conjecture for a true combinatorial
characterization, since the polytopes are decorated with
various semi-cuspidal representations, which at the moment are not
well-understood. 
Some difficulties that come up outside of affine type are discussed in \S\ref{ss:beyond-affine}.
However, our construction does still satisfy the
most basic properties one would expect, as we now summarize (see
Corollaries \ref{cor:bij1} and \ref{cor:bij2} for precise statements).
\begin{thm}\label{thmC}
  For $\fg$ an arbitrary symmetrizable Kac-Moody algebra, the map from
  $\QH$ to polytopes with edges labeled by semi-cuspidal representations is
  injective. Furthermore, for each convex order on roots, the elements of $\QH$ are parameterized by the possible tuples of
  semi-cuspidal representations of smaller KLR algebras decorating the edges along a corresponding path through the polytope, generalizing the
  parameterization of crystals in finite type by Lusztig data.
\end{thm}

As we were completing this paper, some independent work on
similar problems appeared: McNamara \cite{McN} proved a version of
Theorem \ref{thmC} in finite type (amongst other theorems on the structure
of these representations) and Kleshchev \cite{Kl} gave a
generalization of this to affine type. While there was some overlap
with the present paper, these other works
are focused on a single convex order, rather than giving a
description of how different orders interact as we do in Theorems
\ref{th:ft-iso}, \ref{thmB} and \ref{corB}.

\section*{Acknowledgements}
\label{sec:acknowledgements}

We thank Arun Ram for first suggesting this connection to us, Joel Kamnitzer and Dinakar
Muthiah for many interesting discussions, Monica Vazirani for pointing
out the example of \S\ref{sec:an-example}, Scott Carnahan for
directing us to a useful reference on Borcherds algebras
\cite{MO94595}, and Hugh Thomas for pointing out a minor error in our
discussion of convex orders for infinite root systems.
P.~T. was supported by NSF grants DMS-0902649, DMS-1162385 and DMS-1265555; B.~W. was supported by the NSF under Grant DMS-1151473
and the NSA under Grant H98230-10-1-0199.

\renc{\thethm}{\arabic{section}.\arabic{thm}}

\section{Background}

\subsection{Crystals} \label{ss:crystal}

Fix a symmetrizable Kac-Moody algebra $\fg$. Let $\Gamma = (I, E)$
be its Dynkin diagram and $U(\fg)$ its quantized universal
enveloping algebra. Let $\{ E_i, F_i : i \in I\}$ be the Chevalley generators, 
and $U^+(\fg)$
be the part of this algebra generated by the $E_i$. Let $P$ be the weight lattice, $\{\alpha_i \}$ the simple roots, $\{\alpha_i^\vee \}$ the simple co-roots, and $\langle \cdot, \cdot \rangle$ the pairing between weight space and coweight space. 
   
We are interested in the crystal $B(-\infty)$ associated with $U^+(\fg)$. This is a combinatorial object arising from the theory of crystal bases for the corresponding quantum group 
This section contains a brief explanation of the results we need, roughly following \cite{Kashiwara:1995} and \cite{Hong&Kang:2000}, to which we refer the reader for details. We start with a combinatorial notion of crystal that includes many examples which do not arise from representations, but which is easy to characterize. 

\begin{defn}\label{def:crystal} (see \cite[Section 7.2]{Kashiwara:1995}) A {\bf combinatorial crystal} is a set $B$ along with functions $\wt \colon B \to P$ (where $P$ is the weight
  lattice), and, for each $i \in I$, $\varepsilon_i, \varphi_i \colon B \to {\mathbb Z} \cup \{-\infty\}$ and $\e_i, \f_i: B \rightarrow B \sqcup \{ \emptyset \}$, such that
  \begin{enumerate}
  \item $\varphi_i(b) = \varepsilon_i(b) + \langle \wt(b), \alpha_i^\vee \rangle$.
  \item $\e_i$ increases $\varphi_i$ by 1, decreases $\eps_i$ by 1 and increases $\wt$ by $\alpha_i$.
  \item $\f_i b = b'$ if and only if $\e_i b' = b$.
  \item If $\varphi_i(b) = -\infty$, then $\e_i b = \f_i b = \emptyset$.
  \end{enumerate}
We often denote a combinatorial crystal simply by $B$, suppressing the other data.
\end{defn}

\begin{defn}
\label{def:lwcrystal}
A {\bf lowest weight} combinatorial crystal is a combinatorial crystal 
which has a distinguished element $b_-$ (the lowest weight element) such that 
\begin{enumerate}
\item The lowest weight element $b_-$ can be reached from any $b \in B$ by applying a sequence of $\f_i$  for various $i \in I$. 
\item For all $b \in B$ and all $i \in I$, $\varphi_i(b) = \max \{ n : \f_i^n(b) \neq \emptyset \}$.
\end{enumerate}
  \end{defn}
Notice that, for a lowest weight combinatorial crystal, the
functions $\varphi_i, \varepsilon_i$ and $ \wt$ are determined by the
$\f_i$ and the weight $\wt(b_-)$ of just the lowest weight
element. 

The following notion is not common in the literature, but will be very convenient. 
\begin{defn}
  A {\bf bicrystal} is a set $B$ with two different crystal structures
  whose weight functions agree.  We will always use the convention of
  placing a star superscript on all data for the second crystal
  structure, so $\e_i^*,\f_i^*,\varphi^*_i$, etc. 
  We say that an element of a bicrystal is {\bf lowest weight} if it is
killed by both $f_i$ and $f_i^*$ for all $i$.  
\end{defn}

We will consider one very important example of a bicrystal: $B(-\infty)$ along with
the usual crystal operators and Kashiwara's $*$-crystal operators, which are the conjugates $\e_i^*=* \e_i *, \f_i^*=* \f_i *$ of the usual operators by
Kashiwara's involution $*\colon B(-\infty) \to B(-\infty)$ (see \cite[2.1.1]{KasLit}). The involution $*$ is a crystal limit of a corresponding involution of the algebra $U^+(\fg)$, but it also has a simple combinatorial
definition in each of the models  we consider.

The following is a rewording of \cite[Proposition 3.2.3]{KS97} designed to make the roles of the usual crystal operators and the $*$-crystal operators more symmetric:

\begin{prop} \label{cor:comb-characterizaton2}
Fix a bicrystal $B$. Assume $(B, \e_i, \f_i)$ and $(B, \e_i^*, \f_i^*)$ are both lowest weight combinatorial crystals with the same lowest weight element $b_-$, where the other data is determined by setting $\wt(b_-)=0$. Assume further that, for all $i \neq j \in I$ and all $b \in B$,
\begin{enumerate}

\item \label{ccc0} $\e_i(b), \e_i^*(b) \neq 0$. 

\item \label{ccc1} $\e_i^*\e_j(b)= \e_j\e_i^*( b)$,

\item \label{ccc2}For all $b\in B$, $\varphi_i(b)+\varphi_i^*(b)- \langle \wt(b),
  \alpha_i^\vee \rangle\geq0$ 

\item \label{ccc3} If $\varphi_i(b)+\varphi_i^*(b)- \langle \wt(b), \alpha_i^\vee \rangle =0$ then $\e_i(b) = \e_i^*(b)$, 

\item \label{ccc4} If $\varphi_i(b)+\varphi_i^*(b)- \langle \wt(b), \alpha_i^\vee \rangle \geq 1$ then $\varphi_i^*(\e_i(b))= \varphi_i^*(b)$ and  $\varphi_i(e^*_i(b))= \varphi_i(b)$.

\item \label{ccc5} If $\varphi_i(b)+\varphi_i^*(b)- \langle \wt(b), \alpha_i^\vee \rangle \geq 2$ then  $\e_i \e_i^*(b) = \e_i^*\e_i(b)$. 

\end{enumerate}
then $(B, \e_i, \f_i) \simeq (B, \e_i^*, \f_i^*) \simeq B(-\infty)$, and
$\e_i^*= *\e_i *, \f_i^*=*\f_i*$, where $*$ is Kashiwara's
involution. Furthermore, these conditions are always satisfied by $B(-\infty)$ along with its operators $\e_i, \f_i, \e_i^*, \f_i^*$. 
\end{prop}

\begin{proof}
We simply explain how \cite[Proposition 3.2.3]{KS97} implies our statement, referring the reader there for specialized notation. Define the map 
$$
\begin{aligned}
B & \rightarrow B  \otimes   B_i \\
b & \mapsto (\f_i^*)^{\varphi_i^*(b)}(b)  \otimes  \e_i^{\varphi_i^*(b)} b_i.
\end{aligned}
$$
One can check that our conditions imply all the conditions from
\cite[Proposition 3.2.3]{KS97}, so that result implies the crystal
structure on $B$ defined by $\e_i, \f_i$ is isomorphic to $B(-\infty)$. The remaining statements then follow from \cite[Theorem 3.2.2]{KS97}.
\end{proof}

The following is immediate from Proposition \ref{cor:comb-characterizaton2}, but perhaps organizes the information in an easier way:

\begin{cor} \label{cor:KS-diag}
For any $i \in I$, and any $b \in B(-\infty)$, the subset of $B(-\infty)$ generated by the operators $\e_i, \f_i, \e_i^*, \f_i^*$ is of the form:
\vspace{0.15cm}

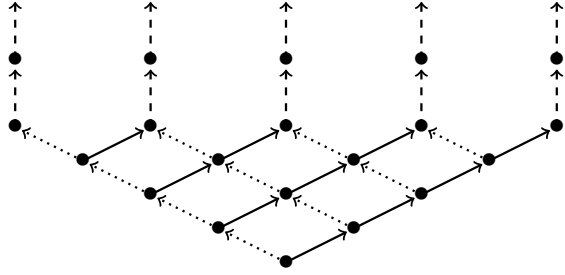
\begin{equation*} \label{ii*-pic}
\setlength{\unitlength}{0.15cm}
\begin{tikzpicture}[xscale=0.45,yscale=-0.45, line width = 0.03cm]

\draw node at (10,5) {$\bullet$};

\draw node at (8,4) {$\bullet$};
\draw node at (12,4) {$\bullet$};
\draw node at (6,3) {$\bullet$};
\draw node at (10,3) {$\bullet$};
\draw node at (14,3) {$\bullet$};

\draw node at (4,2) {$\bullet$};
\draw node at (8,2) {$\bullet$};
\draw node at (12,2) {$\bullet$};
\draw node at (16,2) {$\bullet$};
\draw node at (2,1) {$\bullet$};
\draw node at (6,1) {$\bullet$};
\draw node at (10,1) {$\bullet$};
\draw node at (14,1) {$\bullet$};
\draw node at (18,1) {$\bullet$};

\draw node at (2,-1) {$\bullet$};
\draw node at (6,-1) {$\bullet$};
\draw node at (10,-1) {$\bullet$};
\draw node at (14,-1) {$\bullet$};
\draw node at (18,-1) {$\bullet$};

\draw [->, dotted] (10,5)--(8.2,4.1); 
\draw [->, dotted] (8,4)--(6.2,3.1); 
\draw [->, dotted] (12,4)--(10.2,3.1); 
\draw [->, dotted] (6,3)--(4.2,2.1); 
\draw [->, dotted] (10,3)--(8.2,2.1); 
\draw [->, dotted] (14,3)--(12.2,2.1); 

\draw [->, dotted] (4,2)--(2.2,1.1); 
\draw [->, dotted] (8,2)--(6.2,1.1); 
\draw [->, dotted] (12,2)--(10.2,1.1); 
\draw [->, dotted] (16,2)--(14.2,1.1); 
\draw [->] (10,5)--(11.8,4.1); 

\draw [->] (8,4)--(9.8,3.1); 
\draw [->] (12,4)--(13.8,3.1); 
\draw [->] (6,3)--(7.8,2.1); 
\draw [->] (10,3)--(11.8,2.1); 
\draw [->] (14,3)--(15.8,2.1); 
\draw [->] (4,2)--(5.8,1.1); 
\draw [->] (8,2)--(9.8,1.1); 
\draw [->] (12,2)--(13.8,1.1); 
\draw [->] (16,2)--(17.8,1.1); 

\draw [->, dashed] (2,1) --(2,-0.7);
\draw [->, dashed] (6,1) --(6,-0.7);
\draw [->, dashed] (10,1) --(10,-0.7);
\draw [->, dashed] (14,1) --(14,-0.7);
\draw [->, dashed] (18,1) --(18,-0.7);
\draw [->, dashed] (2,-1) --(2,-2.7);
\draw [->, dashed] (6,-1) --(6,-2.7); 
\draw [->, dashed] (10,-1) --(10,-2.7);
\draw [->, dashed] (14,-1) --(14,-2.7);
\draw [->, dashed] (18,-1) --(18,-2.7);

\end{tikzpicture}
\end{equation*}

\noindent where the solid and dashed arrows show the action of
$\e_i$, and the dotted or dashed arrows denote the action of $\e_i^*$.  Here the width of the diagram at the top is $-\langle \wt(b_v), \alpha_i^\vee \rangle$, where $b_v$ is the bottom vertex (in the example above the width is 4). \qed
\end{cor}

We will also make use of Saito's crystal reflections from \cite{SaitoPBW}.

\begin{defn} \label{def:ref} Fix $b \in B(-\infty)$ with $\varphi_i^*(b)=0$. The {\bf Saito reflection} of $b$ is  $\Sai_i b=
  (\e_i^*)^{\epsilon_i(b)} \f_i^{\varphi_i(b)} b$.  There is also a dual notion of Saito reflection defined by 
  $\Sai_i^*(b) := *(\Sai_i (*b))$, or equivalently $\Sai^*(b) =  (\tilde e_i)^{\epsilon^*_i(b)}\tilde (f^*_i)^{\varphi^*_i(b)} b,$
  which is defined for those $b$ such that $\varphi_i(b)=0$.
  \end{defn}
\noindent  The operation $\Sai_i$ 
  does in fact reflect the weight of $b$ by $s_i$, as the name suggests (although this fails if the condition  $\varphi^*_i(b)=0$ does not hold).

 Finally, we need the notion of string data for an element of $B(-\infty)$.
 This appeared early on in the literature on
  crystals, implicitly in work of Kashiwara \cite{KasLit} 
  and more explicitly in work of Berenstein and Zelevinsky \cite{BZstring}. It was also studied in in the context of KLR algebras (i.e. the context we use) in \cite[\S
  3.2]{KLI} and \cite[\S 3.3]{CTP}.

 Choose a list ${\bf i}= i_1,i_2,\dots$ of simple roots in which each
  simple root occurs infinitely many times (for instance, one could choose an order on the roots and cycle). 
  
  \begin{defn} \label{def-string}
  For any $b \in B(-\infty)$ the {\bf string data} of $b$ with respect to ${\bf i}$ is the lexicographically maximal list of integers $(a_1,a_2, \dots)$ such that $\ldots \f_{i_2}^{\; a_2} \f_{i_1}^{\; a_1} b \neq \emptyset$.
  \end{defn}
  Clearly all but finitely many of the $a_k$ must be zero in any given string datum. Note also that the element $b$ can easily be recovered from its string datum: $b= \e_{i_1}^{\; a_{1}} \e_{i_2}^{\; a_{2}} \cdots b_-$. 
  
  \subsection{Convex orders and charges} \label{ss:convex}
A convex order on roots is generally defined to be a total order such that, if $\alpha, \beta$ and $\alpha+\beta$ are all roots, then $\alpha+\beta$ is between $\alpha$ and $\beta$. Here we need a more geometric definition, and we need to expand to have a notion of convex pre-orders. In fact, our definition makes sense for collections of vectors which do no necessarily come from root systems, and we will set it up in that generality. 
We will then see that in the case of finite type root systems our definition is equivalent to the usual one. 

For this section, fix a finite dimensional vector space $V$ and a set of vectors $\Gamma$ in $V$. 

\begin{defn} \label{def:convex} 
 A {\bf convex preorder} is a pre-order $\succ$ on $\Gamma$
  such that, 
  \begin{enumerate}
\item \label{cpo2} For any equivalence class $\mathscr{C}$, any $a \in
  \text{span}_{{\Bbb R}_{\geq 0}} \mathscr{C}$ and any non-zero $\,x
  \in \text{span}_{{\Bbb Z}_{\geq 0}} \{ \beta \in \Gamma \mid \beta
  \succ \mathscr{C}\}$, we have that 
$a + x \not \in  \text{span}_{{\Bbb Z}_{\geq 0}} \{ \beta \in \Gamma \mid \beta \preceq \mathscr{C}\}.$
\item \label{cpo3} For any equivalence class $\mathscr{C}$, any $a \in \text{span}_{{\Bbb R}_{\geq 0}} \mathscr{C}$ and any non-zero $\,x \in \text{span}_{{\Bbb Z}_{\geq 0}} \{ \beta \in \Gamma \mid \beta \prec \mathscr{C}\}$, we have that 
$a+ x \not \in  \text{span}_{{\Bbb Z}_{\geq 0}} \{ \beta \in \Gamma \mid \beta \succeq \mathscr{C}\}.$
\end{enumerate}
A {\bf convex order} is a convex pre-order which is a total order. 
\end{defn}

\begin{rem}
In the case of a total order, Definition \ref{def:convex} is equivalent to requiring that, for any $S, S' \subset \Gamma$ such that $\alpha \succ \alpha'$ for all $\alpha \in S, \alpha' \in S'$, $\text{span}_{{\mathbb R}_{\geq 0}} S \cap\text{span}_{{\mathbb R}_{\geq 0}} S' = \{ 0 \}.$ 
\end{rem}

\begin{lemma} \label{lem:convex}
  A pre-order $\succ$ on a countable set of vectors $\Gamma$ in a vector space $V$  
  is convex if and
  only if, for any equivalence class $\mathscr{C}$, there is a
  sequence of 
  cooriented hyperplanes $H_n\subset V$ for $n\in
  \Z_{>0}$ such that $\mathscr{C} \subset H_n$ for all $n$, and each
  $\al\in \Gamma$ lies
  \begin{itemize}
\item on the positive
    side of $H_n$ for $n\gg 0$ if $\al \succ \mathscr{C}$ and 
\item on the
    negative side of $H_n$ for $n\gg 0$ if $\al \prec \mathscr{C}$.
  \end{itemize}
\end{lemma}

\begin{rem}
We need to allow a sequence of hyperplanes because $\Gamma$ may be infinite. 
\end{rem}

\begin{proof}
Fix any finite subset $U$ of $\Gamma \backslash \mathscr{C}$. Let $U_\pm$ denote the subsets $U$ consisting of vectors which are greater/less than $\mathscr{C}$ according to $\succ$. 
Consider the quotient  $\mathfrak{h}/\text{span} (\mathscr{C})$, and 
the cones  $C_1=\text{span}_{{\Bbb R}_{\geq 0}} \{ \bar \al  :
\al \in U_- \}$ and $C_2=\text{span}_{{\Bbb R}_{\geq 0} }\{ \bar \al
 : \al \in U_+ \}$ in this space.
Convexity implies that $\bar \alpha \neq 0$ for all $\al \in U$. 

Any point in $C_1 \cap C_2$ has a
preimage in 
$$\text{span}_{{\Bbb R}_{\geq 0}} [\mathscr{C} \cup U_+] \cap
\text{span}_{{\Bbb R}_{\geq 0}}[\mathscr{C} \cup U_-],$$ 
which by convexity must in fact
lie in $\text{span}_{{\Bbb R}_{\geq 0}}\mathscr{C}$. Hence $C_1 \cap C_2 = \{ 0 \}$. 

Similarly, neither
$C_1$ nor $C_2$ contains a line since if $x+y=0$ for $x,y\in C_1$,
then $x$ and $y$ have preimages in $x',y'\in \text{span}_{{\Bbb R}_{\geq 0}}
(U_+ \cup \mathscr{C})$ which we can choose so that $x'+y'\in \text{span}_{{\Bbb R}_{\geq 0}}
\mathscr{C}$.  Convexity thus implies that $x',y'\in \text{span}_{{\Bbb R}_{\geq 0}}
\mathscr{C}$, so $x=y=0$.  

Thus, $C_1$ and $C_2$ are closed finite polyhedral cones in a finite
dimensional vector space whose intersection consists exactly of the origin, neither of which contains a line.  Two such cones
are always separated by a hyperplane since their duals are full
dimensional and span the whole space, and therefore contain elements
in their interiors that sum to 0.  

The preimage of this hyperplane
in $\mathfrak{h}$ separates the elements of $U$ as desired; thus as we
let $U$ grow, we will obtain the desired sequence of hyperplanes.

It is easy to see that, if such a sequence exists for every equivalence
class $\mathscr{C}$, then the order must be convex. 
\end{proof}

We now define charges, which are our main tool for constructing and studying convex orders.

\begin{defn} \label{def:charge}
A {\bf charge} is a linear function $c\colon V\to\C$
such that the image $c(\Gamma)$ is contained in some open half-plane
defined by a line through the origin.
\end{defn}

Every charge defines a preorder $>_c $ on $\Gamma$ by setting
$\alpha \geq_c \beta$ if and only if $\arg(c(\alpha)) \geq
\arg(c(\beta))$, where $\arg$ is the usual argument function on the complex numbers,
taking a branch cut of $\log$ which does not lie in the positive span
of $c(\Gamma)$. This order is independent of the
position of the branch cut.
This preorder is clearly convex, and for generic $c$, it is a total order. 

\begin{lemma} \label{lem:orderexists}
Assume that $\Gamma$ is countable, that it does not contain any pair
of parallel vectors, and that $\Gamma$ is contained in an open half-space $H_+$. Then there is a convex total order on $\Gamma$.
\end{lemma}

\begin{proof}
  Choose a basis $B=\{b_0,b_1,\dots, b_n\}$ such that $b_i$ for $i\geq
  0$ lies in the hyperplane $H=\partial H_+$, and $b_0$ lies in $H_+$.

We can define a charge $c$ by sending $b_i$ to elements of $\R$ and
$b_0$ to the upper half-plane. Since $\Gamma$ is countable, all the
coefficients of $\Gamma$ in terms of $B$ lie in a countable subfield $K$
of $\R$. Choose $a_0\in \C_+$ and $a_1,\dots, a_n\in \R$ such
that $\{\text{Re}(a_0),\text{Im}(a_0), a_1,\dots a_n\}$ is linearly independent
over $K$, and consider the charge defined by  $c(b_i)=a_i$. This sends no two elements of $\Gamma$ to points with the same argument,
since otherwise, writing the two vectors as $v= \sum v_i a_i, v'=\sum v'_i a_i$, we would have 
\[ \sum v_ia_i=\sum v_ic(b_i) =c(v)=p c(v_i')=\sum v'_ic(b_i) p =\sum pv_i'a_i\]
for some $p \in {\Bbb R}$. 
Comparing imaginary parts this is only possible if $v_0=pv_0'$, so
$p\in K$.  This implies that $v_i=pv_i'$ for all $i$ by the linear independence of
$\{\text{Re}(a_0), a_1,\dots a_n\}$, so $v$ and $v'$ are parallel, and we assumed $\Gamma$ does not contain parallel vectors.  Thus, $c$ defines a total order.
\end{proof}

\begin{lemma}\label{refinement}
Assume that $\Gamma$ is countable, that it does not contain any pair
of parallel vectors, and that $\Gamma$ lies in an open half-plane $H_+$. Then
every convex pre-order on $\Gamma$ can be refined to a convex order. Furthermore, this can be done by choosing any convex order on each equivalence class.
\end{lemma}

\begin{proof}
Fix a convex pre-order $\succ$ on $\Gamma$  and an equivalence class $\mathscr{D}$. 
By Lemma \ref{lem:orderexists} we can choose a convex total order $>_\mathscr{D}$ on $\mathscr{D}$. Let $\succ'$ be the refinement of $\succ$ using the order $>_\mathscr{D}$ on $\mathscr{D}$. 
Definition \ref{def:convex} clearly holds for any class
$\mathscr{C}$ which doesn't lie in $\mathscr{D}$.  Thus, we may reduce
to the case where $\mathscr{C}=\{\be\}$ for some $\be\in \mathscr{D}$.

We need to show that if $x\in \text{span}_{{\Bbb R}_{\geq 0}}
\{\gamma\succ'\be \}$, then $x+\be\notin \text{span}_{{\Bbb R}_{\geq 0}}
\{\gamma\preceq'\be \}$.  We can write $x=x'+x''$ with $x'\in   \text{span}_{{\Bbb R}_{\geq 0}}
\{\gamma\succ\mathscr{D} \}$ and $x''\in   \text{span}_{{\Bbb R}_{\geq 0}}
\{\gamma>_{\mathscr{D}}\be \}$.  If $x'\neq 0$, then convexity implies
that \[x+\be=x'+(x''+\be)\notin    \text{span}_{{\Bbb R}_{\geq 0}}
\{\gamma\preceq\mathscr{D} \}\supset   \text{span}_{{\Bbb R}_{\geq 0}}
\{\gamma\preceq'\be\}.\]

On the other hand, if $x'=0$, then we have reduced to the same
situation using only roots from $\mathscr{D}$, so it follows from the
convexity of $>_{\mathscr{D}}$.
\end{proof}

Now fix a symmetrizable Kac-Moody algebra $\fg$ with root system $\Delta$ and Cartan subalgebra $\fh$. Let $\Delta_+^{\min}$ be the set of positive roots $\alpha$ such that $x \alpha$ is not a root for any $0 < x <1$ (this is all positive roots in finite type). From now on we will only consider convex orders on $\Delta_+^{\min}$. 
In this case the conditions of Lemmata \ref{lem:orderexists} and \ref{refinement} clearly hold. Notice also that, since any root can be expressed as a non-negative linear combination of simple roots, for any convex total order the minimal and maximal elements must be simple. 

We will need the following notion of ``reflection" for convex orders and charges.
\begin{defn}
Fix a convex preorder $\succ$ such that $\alpha_i$ is the unique lowest (reps. greatest) root. Define a new convex order $\succ^{s_i}$ by 
$$\be \succ \gamma
\Leftrightarrow s_i\be \succ^{s_i}s_i\gamma \quad \text{ if } \quad  \be, \gamma\neq
\al_i$$
and $\al_i$ greatest (resp. lowest) for $\succ^{s_i}$.

Similarly, for a charge $c$ such that $\arg \al_i$ is lowest (resp. greatest)
amongst positive roots, define a new charge $c^{s_i}$ by $c^{s_i}(\nu)
= c(s_i(\nu))$. 
\end{defn}

It is straightforward to check that reflections for charges
and convex orders are compatible in the sense that, for all charges
$c$ such that $\alpha_i$ is greatest or lowest, $(>_c)^{s_i}$ and
$>_{c^{s_i}}$ coincide.

The following result is well known with the usual definition of convex order. The fact that it holds for our definition as well shows that the two definitions agree in the case of convex orders on finite type root systems. 
\begin{prop} \label{prop:agrees-in finite-type}
Assume $\fg$ is of finite type. 
There is a bijection between convex orders on $\Delta_+$ and
expressions ${\bf i} = i_1 \cdots i_N$ for the longest word $w_0$,
which is given by sending ${\bf i}$ to the order 
$\alpha_{i_1} \succ s_{i_1} \alpha_{i_2} \succ s_{i_1} s_{i_2} \alpha_{i_3} \succ \cdots \succ s_{i_1} \cdots s_{i_{N-1}} \alpha_{i_N}.$ 
\end{prop}

\begin{proof}
First, fix a reduced expression. It is well known that
$$\{\alpha_{i_1}, s_{i_1} \alpha_{i_2},  s_{i_1} s_{i_2} \alpha_{i_3}, \cdots, s_{i_1} \cdots s_{i_{N-1}} \alpha_{i_N} \}$$ is an enumeration of the positive roots, so we have defined a total ordering on positive roots. 
 For any root $\be=s_{i_1} \cdots s_{i_{r-1}} \alpha_{i_r}$, the
 hyperplane defined by the zeros of  $s_{i_1} \cdots s_{i_{r-1}}( \rho^\vee-\om_{i_r}^\vee)$
 separates those larger than it from those smaller than it, so this
 order is convex by Lemma \ref{lem:convex}.

Now fix a convex order $\succ$. The greatest root must be a simple root $\alpha_{i_1}$.  The convex order $\succ^{s_{i_1}}$ as defined above also has a greatest root $\alpha_{i_2}$. Define $i_3$ in the same way using $\succ^{s_1s_2}$ and continue as many times as there
are positive roots.  The list $\alpha_{i_1}, s_{i_1}
\alpha_{i_2}, s_{i_1} s_{i_2} \alpha_{i_3}, \dots, s_{i_1} \cdots
s_{i_{N-1}} \alpha_{i_N}$ is a complete, irredundant list of positive
roots.  This implies that ${\bf i} $ is a reduced expression for $w_0$.
Furthermore, if we apply the procedure in the statement to create an order on positive roots from this expression, we clearly end up with our original convex order.
\end{proof}

Of course, if $\fg$ is of infinite type, the technique in the proof of Proposition \ref{prop:agrees-in finite-type} will result not in a reduced word for the longest element
(which does not exist), but an infinite reduced word $i_1,i_2,i_3,
\dots$ in $I$ as well as a dual sequence $\dots, i_{-3},i_{-2},i_{-1}$
constructed from looking at lowest elements.  The corresponding lists
of roots 
\[\alpha_{i_1}\succ s_{i_1}
\alpha_{i_2}\succ s_{i_1} s_{i_2} \alpha_{i_3}\succ \cdots \qquad\text{and}\qquad \cdots \succ  s_{i_{-1}} s_{i_{-2}} \alpha_{i_{-3}}\succ  s_{i_{-1}}
\alpha_{i_{-2}}\succ\alpha_{i_{-1}} \] are totally ordered, but don't contain every root. We call the roots that appear in the list $\alpha_{i_1}\succ s_{i_1}
\alpha_{i_2}\succ s_{i_1} s_{i_2} \alpha_{i_3}\succ \cdots$ {\bf accessible from below}, and those in the other list {\bf accessible from above}. The terminology is due to the fact that the roots in the first list will correspond to edges near the bottom of the MV polytope, and those in the second list will correspond to edges near the top. The roots that are accessible from above or below are exactly those that are finitely far from one end of the order $\succ$. 

\begin{rem}
 In the
affine case, for most convex orders, only $\delta$ is neither accessible from above nor accessible from below; this happens exactly for the one-row orders from
\cite{Ito}, which includes all orders induced by charges. In more general types, one typically misses many roots, including many real roots. In many cases,
one even misses simple roots.  
\end{rem}

\begin{defn}\label{ctLusztig}
Fix a convex order $\succ$. For each $b \in B(-\infty)$ and
each real root $\al$ which is accessible from above, define an integer $\coa_{\al}^\succ(b)$ by setting $\coa^\succ_{\al_i}(b)=\varphi_i(b)$ if $\alpha_i$ is minimal for $\succ$,
  and \[\coa^\succ_{\al}(b)=\coa_{s_i\al}^{\succ^{s_i}}(\Sai_i^*(\tf_i^{\varphi_i(b)}b))\]
  for all other accessible from above roots.

  Similarly, define $\coa^{\succ }_{\al}$ for all $\alpha$
  which are accessible from below by 
  $\coa^{\succ}_{\al_i}(b)=\varphi^*_i(b)$ if $\alpha_i$ is maximal for $\succ$,
  and \[\coa^{\succ}_{\al}(b)=\coa_{s_i\al}^{\succ^{s_i}}(\Sai_i((\tf^*_i)^{\varphi^*_i(b)}b))\]
  for other accessible from below roots.
  
 We call the collection $\{ \coa^\succ_\alpha(b) \}_{\alpha \in \Delta_+^\text{min}}$ the {\bf crystal-theoretic Lusztig data} for $b$ with respect to $\succ$.
\end{defn}

In infinite types, no root is accessible both from above and below,
but in finite type all are. Thus, in order to justify our
notation, we must prove that the two definitions we
have given for $\coa^{\succ }_{\al}$ agree. In fact, we now show that both agree with the
exponents in Lusztig's PBW basis element corresponding to $b$ for the
reduced expression of $w_0$ giving the convex order $\succ$. This
connection explains the term ``Lusztig data.''

\begin{prop} \label{prop:cld-mv}
In finite type, for any $b \in B(-\infty)$, the two definitions of $\coa_\al^\succ(b)$ in Definition \ref{ctLusztig} agree. Furthermore:
\begin{enumerate}
\item Let $\Bi$ be the reduced expression for $w_0$ corresponding to
  the convex order $\succ$. Then Lusztig's PBW monomial corresponding
  to $b$ as in \cite[Proposition 8.2]{Lusbraid}
is $F_{\beta_1}^{\coa_{\beta_1}^\succ(b)} F_{\beta_2}^{\coa_{\beta_2}^\succ(b)} \cdots F_{\beta_N}^{\coa_{\beta_N}^\succ(b)}.$

\item For all $\alpha$, the
geometric Lusztig data $a^\succ_\alpha(P)$ of the $MV$ polytope corresponding to $b$ agrees with $\coa_\al^\succ(b)$. 
\end{enumerate}
\end{prop}

\begin{proof}
It follows by applying \cite[Proposition 3.4.7]{SaitoPBW} repeatedly
that the definitions of $\coa_\al^\succ(b)$ both read off the exponent
of $F_\alpha$ in Lusztig's PBW monomial corresponding to $b$ for the order $\succ$ (with one definition one starts reading from the right of
the monomial and with the other one starts reading from the
left). Hence they agree and satisfy (i). It is shown in \cite{Kamnitzer, Kamnitzer2} that $a_\alpha^\succ
(P)$ also satisfies (i), from which it is immediate that $a^\succ_\alpha(P)= \coa_\al^\succ(b)$.
\end{proof}

The following notion of compatibility allows us to study an arbitrary convex order on $\Delta_+^\text{min}$ using charges. 

\begin{defn} \label{def:c-comp} Fix a triple $(\mathscr{C}, \succ,n)$,
  where $\succ$ is a convex pre-order on $\Delta_+^{min}$,
  $\mathscr{C}$ is an equivalence class for $\succ$, and $n>0$. 
A charge $c$ is said to be {\bf $(\boldsymbol{\mathscr{C}},
  \boldsymbol\succ,\mathbf{n})$ compatible} if all roots in $\mathscr{C}$ have the same argument with respect to $>_c$ and, for all $\beta \in
\Delta_+$ of depth $\leq n$, we have $\mathscr{C} \prec \beta$ if and only if $ \mathscr{C} <_c \beta$ and $\mathscr{C} \succ \beta$ if and only if $ \mathscr{C} >_c \beta$.
\end{defn}

\begin{lemma} \label{lem:ect}
For every triple $(\mathscr{C}, \succ,n)$ as in Definition \ref{def:c-comp} there is a
$(\mathscr{C}, \succ,n)$ compatible charge. 
\end{lemma}

\begin{proof}
  Choose a sequence of hyperplanes $H_m$ as in Lemma \ref{lem:convex},
  and choose $m$ large enough that all $\beta$'s of depth $\leq n$ are
  on the correct side of $H_m$. One can choose a charge $c$ such that
  $H_m$ is the inverse image of the imaginary line, and such that some
  root $\alpha \succ \mathscr{C}$ has argument greater then
  $\pi/2$. This must in fact be a $(\mathscr{C}, \succ,n)$ compatible
  charge.
\end{proof}

\begin{lemma} \label{lem:i-order}
Fix a convex pre-order $>$ and a cooriented
hyperplane $H$ such that $\al>\be$ whenever $\al$ is on the
positive side of $H$ and $\be$ is on the negative.  Consider a simple
root $\alpha_i$ on the positive side of $H$. Then there is a convex pre-order $\succ$ such that:
\begin{itemize}
\item $\alpha_i$ is the unique maximal root, and
\item  for all
$\alpha, \beta \in \Delta_+^{min}$ with $\alpha$ on the non-positive side of $H$,
$\alpha \succ \beta$ if and only of $\alpha > \beta,$ and $\alpha
\prec \beta$ if and only if $\alpha < \beta$. That is, the relative
position of any root on the non-positive side of $H$ with any other root remains unchanged.
\end{itemize}
If $>$ is a total order then $\succ$
can be taken to be a total order as well.
\end{lemma}

\begin{proof}
Let $f$ be any function whose vanishing locus is $H$ with the correct
coorientation.  Consider the function $f_t=tf+(1-t) s_i \rho^\vee$ for
$t \in [0,1]$.  At $t=0$, the only root on which $f_t$ has a positive
value is $\al_i$; for every other root $\be$ on the positive side of $H$,
there is a unique $t(\be)\in (0,1)$ such that $f_t(\be)=0$.

Define 
$\alpha \succeq \beta$ if
\begin{enumerate}
\item $\alpha \geq \beta$ and $f(\beta) \leq 0$, or

\item $f(\alpha), f(\beta) >0$ and $t(\alpha) \leq t(\beta)$, or

\item $\alpha=\alpha_i$. 
\end{enumerate}
This is a convex pre-order since all equivalence classes and initial/final
segments are either defined as the vectors lying in or on one side of
a  hyperplane, or as equivalence classes or segments for $>$.  Certainly the
relative position of any root on the negative side of $H$ with any other root agrees with the relative
position for $>. $  

By Lemma \ref{refinement}, if the order $>$ is total, we
can refine $\succ$ to a total
order by using $>$ to order within each equivalence class.
\end{proof}

  \subsection{Pseudo-Weyl polytopes} \label{ss:PW}

\begin{defn}
  A {\bf pseudo-Weyl polytope} is a convex polytope $P$ in $\fh^*$ 
  with all edges parallel to roots.
\end{defn}

\begin{defn}
For a pseudo-Weyl polytope $P$, let $\mu_0(P)$ be the vertex of $P$ such that $\langle \mu_0(P), \rho^\vee \rangle$ is lowest, and $\mu^0(P)$ the vertex where this is highest (these are vertices as for all roots $\langle \alpha, \rho^\vee \rangle \neq 0$). 
\end{defn}

\begin{lemma} \label{lem:ispath}
Fix a pseudo-Weyl polytope $P$ and a convex order $\succ$ on $\Delta_+^{min}$. There is a unique path $P^\succ$ through the 1-skeleton of $P$ from $\mu_0(P)$ to $\mu^0(P)$ which passes through at most one edge parallel to each root, and these appear in decreasing order according to $\succ$ as one travels from $\mu_0(P)$ to $\mu^0(P)$.
\end{lemma}

\begin{proof}
Let $\{ \be_1, \be_2, \ldots, \be_r  \} \in \Delta_+^{min}$ be the minimal roots that are parallel to edges in $P$, ordered by $\be_1 \succ \be_2 \succ \cdots \succ \be_r$. Since $\succ$ is convex, for each $1 \leq k \leq r-1$, one can find $\phi_k \in \fh$ such that $\langle \be_r, \phi_k \rangle > 0$ for $r\leq k$, and $\langle \be_r, \phi_k \rangle  < 0$ for $r > k$. Let
$\phi_0= \rho^\vee$ and $\phi_r= - \rho^\vee$. 
Construct a path $\phi_t$ in coweight space for $t$ ranging from $0$ to $r$ by, for $t= k+q$ for $0 \leq q <1$ letting $\phi_t= (1-q) \phi_k+q\phi_{k+1}$. 
As $t$ varies from $0$ to $r$, the locus in the polytope where $\phi_t$ takes on its lowest value is generically a vertex of $P$, but occasionally defines an edge. The set of edges that come up is the required path. 
\end{proof}

\begin{defn} \label{def:geom-LD}
Fix a pseudo-Weyl polytope $P$ and a convex order $\succ$. For each $\alpha \in \Delta_+^{min}$, define $a^\succ_\alpha(P)$ to be the unique non-negative number such that the edge in $P^\succ$ parallel to $\alpha$ is a translate of $a^\succ_\alpha(P) \alpha$. We call the collection $\{ a^\succ_\alpha(P) \}$ the {\bf geometric Lusztig data} of $P$ with respect to $\succ$. 
\end{defn}

\begin{lemma} \label{lem:skel}
Let $P$ be a pseudo-Weyl polytope and $E$ an edge of $P$. Then there
exists a charge $c$ such that $>_c$ is a total order and $E \subset P^{>_c}$. In particular, a pseudo-Weyl polytope $P$ is uniquely determined by its geometric Lusztig data with respect to all convex orders $>_c$ coming from charges. 
\end{lemma}

\begin{proof} Since
$E$ is an edge of $P$, there is a functional $\phi \in \fh$ such
that $$E = \{ p \in P: \langle p, \phi \rangle\: \text{ is greatest} \}.$$ Since $P$ is a pseudo-Weyl polytope, $E$ is parallel to some root $\beta$, and so $\langle \beta, \phi \rangle=0$. Furthermore, $\phi$ may be chosen so that $\langle \beta', \phi \rangle \neq 0$ for all other $\beta'$ which are parallel to edges of $P$. For any linear function $f: \fh \rightarrow {\mathbb R}$ such that $f(\Delta_+) \subset {\mathbb R}_+,$ define a charge $c_f$ by 
$$c_f(p)= \phi(p) + f(p) i.$$
For generic $f$, the charge $c_f$ satisfies the required conditions. 
\end{proof}
The following should be thought of as a general-type analogue of the fact that, in finite type, any reduced expression for $w_0$ can be obtained from any other by a finite number of braid moves. In fact, this statement can be generalized to include all convex orders, not just those coming from charges, but we only need the simpler version. 

There is a natural height function on any pseudo-Weyl polytope given by $\rho^\vee$. Thus for each face there is a notion of top and bottom vertices. 

\begin{lemma} \label{lem:gen-braid}
Let $P$ be a pseudo-Weyl polytope and $c,c'$ two generic charges. Then there is a sequence of generic charges $c_0, c_1, \ldots c_k$ such that $P^{>_{c_0}}= P^{>_{c}}$, $P^{>_{c_k}}=P^{{>_{c'}}},$ and, for all $k \leq j <k$,  $P^{>_{c_j}}$ and $P^{>_{c_{j+1}}}$ differ by moving from the bottom vertex to the top vertex of a single 2-face of $P$ in the two possible directions. 
\end{lemma}
 
\begin{proof}
Let $\Delta^{res}$ be the set of root directions that appear as edges in $P$. 
For $0 \leq t \leq 1$, let $c_t = (1-t) c+ t c'$. Clearly this is a charge. We can deform $c,c'$ slightly, without changing the order of any of the roots in $\Delta^{res}$, such that 
\begin{itemize}
\item For all but finitely many $t$, $c_t$ induces a total order on $\Delta^{res}.$

\item For those $t$ where $c_t$ does not induce a total order, there is exactly one argument $0 < a_t < \pi$ such that more then one root in $\Delta^{res}$ has argument $a_t$. Furthermore, the span of the roots with argument  $a_t$ is 2 dimensional.
\end{itemize}
Denote the values of $t$ where $c_t$ does not induce a total order by $\vartheta_1, \ldots \vartheta_{k-1}$. Fix $t_1, \ldots, t_k$ with
$$0 = t_0 < \vartheta_1 < t_1 < \vartheta_2 \ldots < t_{k-1} < \vartheta_{k-1} < t_k =1.$$
Then $c_j = c_{t_j}$ is the required sequence. 
\end{proof}

\subsection{Finite type MV polytopes}

Mirkovi\'c-Vilonen (MV) polytopes are polytopes in the weight space of a complex-simple Lie algebra which first arose as moment map images of the MV cycles in the affine Grassmannian, as studied by  
 \Mirkovic  and Vilonen \cite{MV}. 
Anderson \cite{JAnderson} and Kamnitzer
\cite{Kamnitzer, Kamnitzer2} developed a realization of $B(-\infty)$
using these polytopes as the underlying set. 
Here we will not
need details of these constructions, but will only use certain characterization theorems.

The following is discussed implicitly in \cite{BKMV}.

\begin{prop} \label{prop:char-finite} For finite type $\fg$,
there is a unique map $b \rightarrow P_b$ from $B(-\infty)$ to pseudo-Weyl polytopes such that
\begin{enumerate}
\item \label{ccf1} $\wt(b)= \mu^0(P_b)- \mu_0(P_b)$. 

\item  \label{ccf2} If $\succ$ is a convex order with minimal root $\alpha_i$, then, for all $\beta \neq \alpha_i$, $a_\beta^{\succ}(P_{\e_i(b)})=  a_\beta^{\succ}(P_b)$, and 
$a_{\alpha_i}^{\succ}(P_{\e_i(b)})=  a_{\alpha_i}^{\succ}(P_b)+1$.

\item \label{ccf3}  If $\succ$ is a convex order with minimal root $\alpha_i$ and $\varphi_i(P_b)=0$, then, for all $\beta \neq \alpha_i$, $a_\beta^\succ(P_b) = a^{\succ^{s_i}}_{s_i(\beta)}(P_{\Sai_i b})$ and $a^{\succ^{s_i}}_{\alpha_i}(P_{\Sai_i b})=0.$ 

\end{enumerate}
Here $\sigma_i$ is the Saito reflection from Definition \ref{def:ref}.
This map is the unique bicrystal isomorphism between $B(-\infty)$ and  the set of MV polytopes.
\end{prop}

\begin{proof}
The first step is to show that there is at most one map $b \rightarrow P_b$ satisfying the conditions. To see this we proceed by induction. Consider the reverse-lexicographical order on collections of integers ${\bf a}= (a_{k})_{1 \leq k \leq N}$. Assume ${\bf a}$ is minimal such that, for some convex order $$\beta_1 \succ \beta_2 \succ \cdots \cdots \succ \beta_N,$$ and for two maps $b \rightarrow P_b$ and $b \rightarrow P'_b$ satisfying the conditions, 
$a_{\beta_k}^{\succ}(P_b)= a_k $ for all $k$, but $a_{\beta_k}^{\succ}(P'_b)\neq a_k$ for some $k$. 

If $a_N \neq 0$ we can reduce to a smaller such example using condition (ii). Otherwise, as long as some $a_k \neq 0$, we can reduce to a smaller such example using (iii). Clearly the map is unique if all $a_k=0$, so this proves uniqueness. 

It remains to show that $b \rightarrow MV_b$ does satisfy the conditions. But this is immediate from \cite[Proposition 3.4.7]{SaitoPBW} and the fact that the integers $a^\succ_{\beta_k}(MV_b)$ agree with the exponents in the Lusztig's PBW monomial corresponding to $b$, which is shown in \cite[Theorem 7.2]{Kamnitzer}.
\end{proof}

We also need the following standard facts about MV polytopes:
\begin{thm} [\mbox{\cite[Theorem D]{Kamnitzer}}] \label{th:ft-2-face-char}
The MV polytopes are exactly those pseudo-Weyl polytopes such that all 2-faces are MV polytopes for the corresponding rank 2 root system.
\qed
\end{thm}

\begin{thm}  [\mbox{\cite[4.2]{Kamnitzer}}]
An MV polytope is uniquely determined by its geometric Lusztig data with respect to any one convex order on positive roots. \qed
\end{thm}

\subsection{Rank 2 affine MV polytopes}
We briefly review the MV polytopes associated to the affine root
systems $\asl_2$ and $A_2^{(2)}$ in \cite{BDKT}, and recall a
characterization of the resulting polytopes developed in \cite{MT??}.

The $\asl_2$ and $A_2^{(2)}$ root systems correspond to the affine Dynkin diagrams 

\begin{center}
\begin{tikzpicture}
\draw 
(-0.1, 0.2) --(0.7,0.2)
(-0.1, 0.1) --(0.7,0.1)

(-0.15,0.15) --(0,0.3)
(-0.15,0.15) --(0,0)

(0.75,0.15) --(0.6,0.3)
(0.75,0.15) --(0.6,0)
;

\draw node at (-0.4, 0.15) {$\bullet$};
\draw node at (1, 0.15) {$\bullet$};

\draw node at (-0.4, -0.35) {$0$};

\draw node at (1.15, -0.35) {$1 \;\;,$};

\draw node at (-1.5, 0) {$\asl_2:$};

\end{tikzpicture}
\qquad \qquad
\begin{tikzpicture}

\draw (0, 0.3) --(0.7,0.3)
(-0.1, 0.2) --(0.7,0.2)
(-0.1, 0.1) --(0.7,0.1)
(0, 0) --(0.7,0)

(-0.15,0.15) --(0.1,0.4)
(-0.15,0.15) --(0.1,-0.1)
;

\draw node at (-0.4, 0.15) {$\bullet$};
\draw node at (1, 0.15) {$\bullet$};

\draw node at (-0.4, -0.35) {$0$};

\draw node at (1.15, -0.35) {$1 \;\;.$};

\draw node at (-1.5, 0) {$A_2^{(2)}:$};

\end{tikzpicture}
\end{center}

\noindent The corresponding symmetrized Cartan matrices are 
$$\asl_2: \quad N = 
\left(
\begin{array}{rr}
2 & -2 \\
-2 & 2
\end{array}
\right), \qquad
A_2^{(2)}: \quad N = 
\left(
\begin{array}{rr}
2 & -4 \\
-4 & 8
\end{array}
\right).
$$
Denote the simple roots by $\alpha_0, \alpha_1$, where in the case of $A_2^{(2)}$ the short root is $\alpha_0$.  Define $\delta = \alpha_0+ \alpha_1$ for $\asl_2$ and $\delta =  2 \alpha_0+ \alpha_1$ for $A_2^{(2)}$. 

The dual Cartan subalgebra $\mathfrak{h}^*$ of $\fg$ is a three dimensional vector space containing $\alpha_0, \alpha_1$. This has a standard non-degenerate bilinear form $(\cdot, \cdot)$ such that 
$(\alpha_i, \alpha_j) = N_{i,j}.$
Notice that $(\alpha_0, \delta)= (\alpha_1, \delta)=0$. Fix  fundamental coweights $\omega_0, \omega_1$ which satisfy $( \alpha_i, \omega_j) = \delta_{i,j}$, where we are identifying coweight space with weight space using $(\cdot, \cdot)$. 

The set of positive roots for $\asl_2$ is
\begin{equation} \label{eq:roots}
\{\alpha_0, \alpha_0+\delta, \alpha_0+2\delta, \ldots \} \sqcup \{\alpha_1, \alpha_1+\delta, \alpha_1+2\delta, \ldots \} \sqcup \{\delta, 2\delta, 3\delta \ldots \},
\end{equation}
where the first two families consist of real roots and the third family of imaginary roots. 
The set of positive roots for $A_2^{(2)}$ is 
\begin{equation}
\Delta^+_{re} = \{   \alpha_0+ k  \delta,   \alpha_1+2 k  \delta,   \alpha_0+  \alpha_1+k  \delta, 2 \alpha_0+ (2k+1) \delta \mid k \geq 0 \} \quad \text{and} \quad \Delta_{im}^+= \{ k \delta \mid k \geq 1 \},
\end{equation} 
where $\Delta^+_{re}$ consists of real roots and $\Delta^+_{im}$ of imaginary roots.
We draw these as
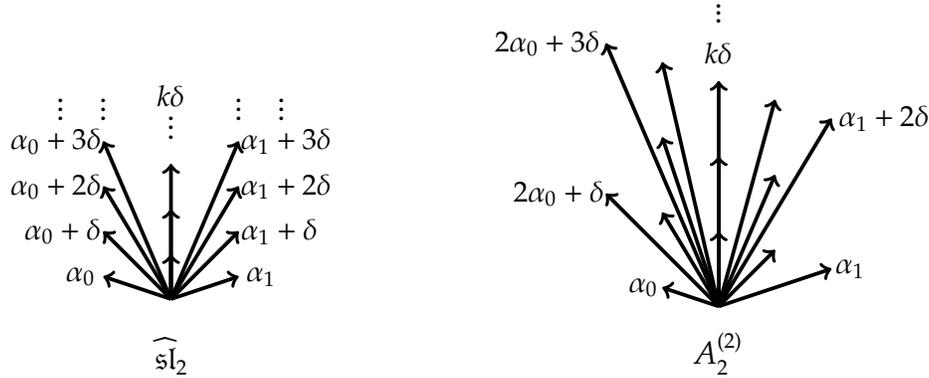
\begin{center}
\begin{tikzpicture}[scale=0.3]

\draw[line width = 0.05cm, ->] (0,0)--(3,1);
\draw[line width = 0.05cm, ->] (0,0)--(-3,1);

\draw[line width = 0.05cm, ->] (0,0)--(3,3);
\draw[line width = 0.05cm, ->] (0,0)--(-3,3);

\draw[line width = 0.05cm, ->] (0,0)--(3,5);
\draw[line width = 0.05cm, ->] (0,0)--(-3,5);

\draw[line width = 0.05cm, ->] (0,0)--(3,7);
\draw[line width = 0.05cm, ->] (0,0)--(-3, 7);

\draw[line width = 0.05cm, ->] (0,0)--(0,2);
\draw[line width = 0.05cm, ->] (0,0)--(0,4);
\draw[line width = 0.05cm, ->] (0,0)--(0,6);

\draw node at (0,8) {.};
\draw node at (0,7.6) {.};
\draw node at (0,7.2) {.};

\draw node at (-3,8.8) {.};
\draw node at (-3,8.4) {.};
\draw node at (-3,8) {.};

\draw node at (3,8.8) {.};
\draw node at (3,8.4) {.};
\draw node at (3,8) {.};

\draw node at (-4.9,8.8) {.};
\draw node at (-4.9,8.4) {.};
\draw node at (-4.9,8) {.};

\draw node at (4.9,8.8) {.};
\draw node at (4.9,8.4) {.};
\draw node at (4.9,8) {.};

\draw node at (-4,1) {{\small $ \alpha_0$}};
\draw node at (-4.8,3) {{\small $\alpha_0+\delta$}};
\draw node at (-5.1,5) {{\small $\alpha_0+2\delta$}};
\draw node at (-5.1,7) {{\small $\alpha_0+3\delta$}};

\draw node at (4,1) {{\small $\alpha_1$}};
\draw node at (4.8,3) {{\small $\alpha_1+\delta$}};
\draw node at (5.1,5) {{\small $\alpha_1+2\delta$}};
\draw node at (5.1,7) {{\small $\alpha_1+3\delta$}};

\draw node at (0,9) {{\small $k \delta$}};

\draw node at (0,-2.5) {$\asl_2$};

\end{tikzpicture}
$\qquad \qquad$
\begin{tikzpicture}[scale=0.25]

\draw[line width = 0.05cm, ->] (0,0)--(-3,1);

\draw[line width = 0.05cm, ->] (0,0)--(6,2);
\draw[line width = 0.05cm, ->] (0,0)--(3,3);
\draw[line width = 0.05cm, ->] (0,0)--(0,4);
\draw[line width = 0.05cm, ->] (0,0)--(-3,5);
\draw[line width = 0.05cm, ->] (0,0)--(-6,6);

\draw[line width = 0.05cm, ->] (0,0)--(3,7);
\draw[line width = 0.05cm, ->] (0,0)--(0,8);
\draw[line width = 0.05cm, ->] (0,0)--(-3,9);

\draw[line width = 0.05cm, ->] (0,0)--(6,10);
\draw[line width = 0.05cm, ->] (0,0)--(3,11);
\draw[line width = 0.05cm, ->] (0,0)--(0,12);
\draw[line width = 0.05cm, ->] (0,0)--(-3,13);
\draw[line width = 0.05cm, ->] (0,0)--(-6,14);

\draw node at (0,16) {.};
\draw node at (0,15.6) {.};
\draw node at (0,15.2) {.};

\draw node at (-4,1) {{\small $ \alpha_0$}};
\draw node at (-8.5,6) {{\small $2 \alpha_0 +  \delta$}};
\draw node at (-9.1,14) {{\small $2 \alpha_0 + 3  \delta$}};

\draw node at (7,2) {{\small $ \alpha_1$}};
\draw node at (8.8,10) {{\small $ \alpha_1 + 2  \delta$}};

\draw node at (0,13.5) {{\small $k  \delta$}};

\draw node at (0,-2.5) {$A_2^{(2)}$};

\end{tikzpicture}
\end{center}

\begin{defn} Label the positive real roots by $r_k, r^k$ for $k \in {\mathbb Z}_{>0}$ by:
\begin{itemize}
\item For $\asl_2$:
$r_k = \alpha_1 + (k-1) \delta$ and $r^k = \alpha_0 + (k-1) \delta$. 

\item For $A_2^{(2)}$: 
 $$
 r_k=\begin{cases}
 \tilde\alpha_1+(k-1)\tilde\delta&\text{if $k$ is odd,}\\
 \tilde\alpha_0+\tilde\alpha_1+\frac{k-2}2\tilde\delta&\text{if $k$ is even,}
 \end{cases}
 \qquad
 r^k=\begin{cases}
 \tilde\alpha_0+\frac{k-1}2\tilde\delta&\text{if $k$ is odd,}\\
 2\tilde\alpha_0+(k-1)\tilde\delta&\text{if $k$ is even.}
 \end{cases}
$$
\end{itemize}
\end{defn}
\noindent There are exactly two convex orders on $\Delta_+^{min}$: the order $\succ_+$
$$ r_1 \succ_+ r_2 \succ_+ \cdots \succ_+ \delta \succ_+ \cdots \succ_+ r^2 \succ_+ r^1,$$
and the reverse of this order, which we denote by $\succ_-$.

\begin{defn} \label{def:DPW-r2a}
A {\bf rank 2 affine decorated pseudo-Weyl polytope} is a pseudo-Weyl polytope along with a choice of two partitions
$a_\delta= (\lambda_1 \geq  \lambda_2 \geq \cdots)$ and $\oa_\delta = (\olambda_1 \geq \olambda_2 \geq \cdots)$ such that $\mu^\infty - \mu_\infty = |a_\delta| \delta$ and  $\omu^\infty - \omu_\infty = |\oa_\delta| \delta$. Here $|a_\delta|= \lambda_1 + \lambda_2+ \cdots$ and $ |\oa_\delta|= \olambda_1 + \olambda_2 + \cdots$ and $\mu_\infty, \omu_\infty, \mu^\infty, \omu^\infty$ are as in Figure \ref{MVpoly}.
\end{defn}

\begin{defn}
The right Lusztig data of a decorated pseudo-Weyl polytope $P$ is the refinement  ${\bf a}= ( a_\alpha )_{\alpha \in \Delta_+^\text{min}}$ of the Lusztig data from \S\ref{ss:PW} with respect to $\succ_+$ (which records the lengths of the edges parallel to each root up one side of $P$), where, for $\alpha \neq \delta$, $a_\alpha = a^{\succ_+}_\alpha(P)$, and $a_\delta$ is the partition from Definition \ref{def:DPW-r2a}. Similarly the left Lusztig data is
 ${\bf \oa}=( \oa_\alpha )_{\alpha \in \Delta_+^\text{min}}$ where, for $\alpha \neq \delta,$ $ \oa_\alpha = a^{\succ_-}_\alpha(P)$, and $\oa_\delta$ is as in Definition \ref{def:DPW-r2a}.
\end{defn}

\begin{figure}[ht]

\begin{center}
\begin{tikzpicture}[yscale=0.2, xscale=0.6]

\draw [line width = 0.01cm, color=gray] (-1.5,2.5) -- (3.5, 7.5);
\draw [line width = 0.01cm, color=gray] (-2,4) -- (4, 10);
\draw [line width = 0.01cm, color=gray] (-2.5,5.5) -- (4, 12);
\draw [line width = 0.01cm, color=gray] (-3.3333,8.6666) -- (4, 16);
\draw [line width = 0.01cm, color=gray] (-3.6666,10.3333) -- (4, 18);
\draw [line width = 0.01cm, color=gray] (-4.25,13.75) -- (4, 22);
\draw [line width = 0.01cm, color=gray] (-4.5,15.55) -- (4, 24);
\draw [line width = 0.01cm, color=gray] (-4.75,17.25) -- (4, 26);
\draw [line width = 0.01cm, color=gray] (-5,19) -- (4, 28);
\draw [line width = 0.01cm, color=gray] (-5,21) -- (4, 30);
\draw [line width = 0.01cm, color=gray] (-5,23) -- (4, 32);
\draw [line width = 0.01cm, color=gray] (-5,25) -- (4, 34);
\draw [line width = 0.01cm, color=gray] (-4.6666,29.3333) -- (3.6666, 37.6666);
\draw [line width = 0.01cm, color=gray] (-4.3333,31.6666) -- (3.33333, 39.3333);

\draw [line width = 0.01cm, color=gray] (1,1) -- (-2,4);
\draw [line width = 0.01cm, color=gray]  (2,2)-- (-3,7);
\draw [line width = 0.01cm, color=gray] (2.5,3.5) -- (-3.5,9.5);
\draw [line width = 0.01cm, color=gray] (3.3333,6.6666) -- (-4.3333,14.3333);
\draw [line width = 0.01cm, color=gray] (3.6666,8.3333) -- (-4.6666,16.6666);
\draw [line width = 0.01cm, color=gray] (4,12) -- (-5,21);
\draw [line width = 0.01cm, color=gray] (4,14) -- (-5,23);
\draw [line width = 0.01cm, color=gray] (4,16) -- (-5,25);
\draw [line width = 0.01cm, color=gray] (4,18) -- (-5,27);

\draw [line width = 0.01cm, color=gray] (4,20) -- (-4.75,28.75);
\draw [line width = 0.01cm, color=gray] (4,22) -- (-4.5,30.5);
\draw [line width = 0.01cm, color=gray] (4,24) -- (-4.25,32.25);
\draw [line width = 0.01cm, color=gray] (4,28) -- (-3.5,35.5);

\draw [line width = 0.01cm, color=gray] (4,32) -- (-2,38);
\draw [line width = 0.01cm, color=gray] (4,34) -- (-1,39);
\draw [line width = 0.01cm, color=gray] (4,36) -- (0,40);
\draw [line width = 0.01cm, color=gray] (3.5,38.5) -- (1,41);

\draw [line width = 0.01cm, color=gray] (-3,7) -- (4, 14);
\draw [line width = 0.01cm, color=gray] (-4,12) -- (4, 20);

\draw [line width = 0.01cm, color=gray] (-3,37) -- (4, 30);
\draw [line width = 0.01cm, color=gray] (-4,34) -- (4, 26);
\draw [line width = 0.01cm, color=gray] (-5,27) -- (4, 18);

\draw (0,0) node {$\bullet$};

\draw (-1,1) node {$\bullet$};
\draw (-3,7) node {$\bullet$};
\draw (-4,12) node  {$\bullet$};
\draw (-5,19) node  {$\bullet$};
\draw (-5,23) node {$\bullet$};
\draw (-5,25) node {$\bullet$};
\draw  (-5,27) node {$\bullet$};
\draw (-4,34) node {$\bullet$};
\draw (-3,37) node {$\bullet$};
\draw  (2,42)  node {$\bullet$};

\draw (2,2)  node {$\bullet$};
\draw (3,5) node {$\bullet$};
\draw (4,10) node {$\bullet$};
\draw (4,28) node {$\bullet$};

\draw (4,28) node {$\bullet$};
\draw (4,32) node {$\bullet$};
\draw (4,34) node {$\bullet$};
\draw (4,36) node {$\bullet$};
\draw (3,41) node {$\bullet$};

\draw [line width = 0.04cm] (0,0)--(-1,1);
\draw  [line width = 0.04cm] (-1,1)--(-3,7);
\draw [line width = 0.04cm] (-3,7)--(-4,12);
\draw [line width = 0.04cm] (-4,12)--(-5,19);
\draw [line width = 0.04cm] (-5,19)--(-5,27);
\draw [line width = 0.04cm] (-5,27)--(-4,34);
\draw [line width = 0.04cm] (-4,34)--(-3,37);
\draw [line width = 0.04cm](-3,37)--(2,42);

\draw [line width = 0.04cm](0,0)--(2,2);
\draw [line width = 0.04cm] (2,2)--(3,5);
\draw [line width = 0.04cm] (3,5)--(4,10);
\draw [line width = 0.04cm] (4,10)--(4,36);
\draw [line width = 0.04cm](4,36)--(3,41);
\draw [line width = 0.04cm] (3,41)--(2,42);

\draw [line width = 0.01cm, color=gray] (-1,1) -- (3,5);
\draw [line width = 0.01cm, color=gray] (3,5) -- (-4,12);
\draw [line width = 0.01cm, color=gray] (4,10) -- (-5,19);
\draw [line width = 0.01cm, color=gray] (-5,27) -- (4,36);
\draw [line width = 0.01cm, color=gray] (-4,34) -- (3,41);

\draw (1.2,-0.35) node {\tiny $\alpha_1$};
\draw (3,2.7) node {\tiny $\alpha_1+\delta$};
\draw (4.4,7) node {\tiny $\alpha_1+2\delta$};
\draw (4.6,23) node {\tiny $\delta$};
\draw (4.4,39) node {\tiny $\alpha_0+2\delta$};
\draw (2.8,42.5) node {\tiny $\alpha_0$};

\draw (-0.5,41) node {\tiny $\alpha_1$};
\draw (-4.15,36.3) node {\tiny $\alpha_1+\delta$};
\draw (-5.5,31) node {\tiny $\alpha_1+3\delta$};
\draw (-5.5,23) node {\tiny $\delta$};
\draw (-5.5,15) node {\tiny $\alpha_0+3\delta$};
\draw (-4.5,9) node {\tiny $\alpha_0+2\delta$};
\draw (-3,3.5) node {\tiny $\alpha_0+\delta$};
\draw (-0.7,-0.7) node {\tiny $\alpha_0$};

\draw(5,10) node {$ \mu_\infty$};
\draw(5,36) node {$\mu^\infty$};
\draw(-6,19) node {$\overline \mu_\infty $};
\draw(-6,27) node {$\overline \mu^\infty $};

\draw[line width = 0.03cm, ->] (0,-4)--(1,-3);
\draw[line width = 0.03cm, ->] (0,-4)--(-1,-3);
\draw (1.5,-2.6) node {$\alpha_1$};
\draw (-1.5,-2.6) node {$\alpha_0$};

\end{tikzpicture}

\end{center}

\caption{\label{MVpoly}An $\asl_2$ MV polytope. 
The partitions labeling the vertical edges are indicated by including extra vertices on the vertical edges, such that the edge is cut into the pieces indicated by the partition. 
The root parallel to each non-vertical edge is indicated. The Lustig data ${\bf a}= {\bf a}^{\succ_+}$ records the path on the right side, and ${\bf \oa}= {\bf a}^{\succ_-}$ records the path on the left. Hence
$$ \hspace{-3cm} a_{\alpha_1}=2, \; a_{\alpha_1+\delta}=1, \; a_{\alpha_1+2\delta} =1,\; a_\delta=(9,2,1,1), \; a_{\alpha_0+2\delta}=1, \; a_{\alpha_0}=1, $$
$$\hspace{-2.5cm} \oa_{\alpha_0} =1,\;  \oa_{\alpha_0+\delta} = 2, \; \oa_{\alpha_0+2\delta}=1, \; \oa_{\alpha_0+3\delta} = 1, \; \oa_\delta=(2,1,1), \; \oa_{\alpha_1+3\delta}=1, \; \oa_{\alpha_1+\delta}=1, \; \oa_{\alpha_1}=5,$$
and all others are $0$. }

\end{figure}

In \cite{BDKT}, the first author and collaborators combinatorially
define a set $\MV$ of decorated pseudo-Weyl polytopes, which they call
{\bf rank 2 affine MV polytopes}. 
We will not need the details of this construction, but will instead use the following result from \cite{MT??}.

Assume $\fg$ is of rank 2 affine type. Define $\ell_0$ and $\ell_1$ by $\delta= \ell_0 \alpha_0+ \ell_1 \alpha_1$ (so $\ell_0=\ell_1=1$ for $\asl_2$, and $\ell_0=2, \ell_1=1$ for $A_2^{(2)}$). 

\begin{thm} \cite[Theorem 3.11]{MT??} \label{th:unique-aff} There is
  a unique map $b \rightarrow P_b$ from $B(-\infty)$ to type $\fg$
  decorated pseudo-Weyl polytopes (considered up to translation) such
  that, for all $b \in B(-\infty)$, the following hold.
    \begin{enumerate}
    \item \label{CC1-aff} $\wt(b)= \mu^0(P_b)- \mu_0(P_b)$.

    \item[(ii.1)] \label{CC2-aff}
        $a_{\alpha_0}(P_{\e_0 b}) = a_{\alpha_0}(P_b)+1$, and for
        all other root directions $a_\alpha(P_{\e_0 b})=
        a_\alpha(P_b)$; 
\item[(ii.2)] $\oa_{\alpha_1}(P_{\e_1 b}) =
        \oa_{\alpha_1}(P_b)+1$, and for all other root directions
        $\oa_\alpha(P_{\e_1 b})= \oa_\alpha(P_b)$;
\item[(ii.3)]
        $a_{\alpha_1}(P_{\e_1^*b}) = a_{\alpha_1}(P_b)+1$, and for all
        other root directions $a_\alpha(P_{\e_1^*b})= a_\alpha(P_b)$;
\item [(ii.4)]
        $\oa_{\alpha_0}(P_{\e_0^*b}) = \oa_{\alpha_0}(P_b)+1$, and for
        all other root directions $\oa_\alpha(P_{\e_0^*b})=
        \oa_\alpha(P_b)$.
    
  \hspace{-0.65in} Let $\Sai_0, \Sai_1$ denote Saito's reflections.

    \item[(iii.1)] \label{CC4-aff}
    If $a_{\alpha_0}(P_b)=0$, then for all $\alpha \neq
        \alpha_0$, $a_\alpha(P_b) =
        \oa_{s_0(\alpha)}(P_{\Sai_0(b)})$ and $\oa_{\alpha_0}(P_{\Sai_0(b)})=0$; 
\item[(iii.2)] if
        $\oa_{\alpha_1}(P_b)=0$, then for all $\alpha \neq \alpha_1$,
        $\oa_\alpha(P_b) = a_{s_1(\alpha)}(P_{\Sai_1(b)})$ and $a_{\alpha_1}(P_{\Sai_1(b)})=0$; 
\item[(iii.3)] if
        $\oa_{\alpha_0}(P_b)=0$, then for all $\alpha \neq \alpha_0$,
        $\oa_\alpha(P_b) = a_{s_0(\alpha)}(P_{\Sai^*_0(b)})$ and $a_{\alpha_0}(P_{\Sai_0^*(b)})=0$; 
\item[(iii.4)] if
        $a_{\alpha_1}(P_b)=0$, then for all $\alpha \neq \alpha_1$,
        $a_\alpha(P_b) = \oa_{s_1(\alpha)}(P_{\Sai^*_1(b)})$ and $\oa_{\alpha_1}(P_{\Sai_1^*(b)})=0$.

    \item[(iv)] \label{CC5-aff} If $a_\beta(P_b)=0$ for all real roots
      $\beta$ and $a_\delta(P_b)= \lambda \neq 0$ then:
      \[\oa_{\alpha_1}(P_b)=\ell_1 \lambda_1;\quad
      \oa_\delta(P_{b})=\lambda \backslash \lambda_1;\quad
      \oa_{\alpha_0}(P_{b})
      =\ell_0\lambda_1;\]
\[ \oa_\beta(P_b)=0\text{ for all other }\beta \in \tilde
      \Delta_+.\] 
    \end{enumerate}
    This map sends $b$ to the corresponding affine MV polytope in the realization of $B(-\infty)$ from \cite{BDKT}. In particular, the image is exactly $\MV$. \qed
  \end{thm}

\begin{rem} \label{rem-lda}
Theorem \ref{th:unique-aff} implies that, for any rank-2 affine MV polytope and any convex order $\succ$, the crystal theoretic Lusztig data $\coa^\succ_\alpha$ agrees with the geometric Lusztig data $a^\succ_\alpha$ for the corresponding MV polytope for all accessible roots $\alpha$. In fact, it follows from Corollary \ref{cor:crystal-polytope} below that this remains true in higher rank affine cases.
\end{rem}

\subsection{Khovanov-Lauda-Rouquier algebras} The construction in this section is due to  \cite{KLI,Rou2KM}
 for Kac-Moody algebras,
and was extended to the case of Borcherds algebras in \cite{KOP}.  

The Khovanov-Lauda-Rouquier (KLR) algebra is 
built out of generic string diagrams, i.e. immersed
1-dimensional sub-manifolds of $\R^2$ whose boundary lies on the lines
$y=0$ and $y=1$, where each string (i.e each immersed copy of the interval) projects
homeomorphically to $[0,1]$ under the projection to the $y$-axis (so in particular there are no closed loops). These
are assumed to be generic in the sense that 
\begin{itemize}

\item no points lie on 3 or more components
\item no components intersect non-transversely.
\end{itemize}
Each string is labeled with a simple root of the corresponding
Kac-Moody algebra, and each string is allowed to carry dots at
any point where it does not intersect another (but with only finitely
many dots in each diagram).  All diagrams are considered
up to isotopy preserving all these conditions.

Define a product on the space of $\K$-linear combinations of these diagrams, where the product $ab$ of two diagrams is
formed by stacking $a$ on top of $b$, shrinking
vertically by a factor of 2, and smoothing kinks; if the labels of the
line $y=0$ for $a$ and $y=1$ for $b$ cannot be isotoped to match, the
product is 0.

This product gives the space of $\K$-linear combinations of these diagrams the structure of an algebra, which has the following generators: For each sequence $\Bi=(i_1,\dots, i_n)$ of nodes of the Dynkin diagram,
\begin{itemize}
\item The idempotent $e_{\Bi}$ which is straight lines labeled with $(i_1,\dots, i_n)$.

\item The element $y_k^\Bi$ which is just straight lines with a dot on
the $k$th strand.

\item The element $\psi_k^\Bi$ which is a crossing of the $i$ and
$i+1$st strand.  
\end{itemize}

 
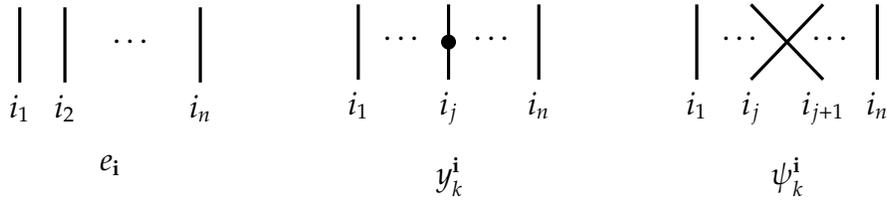
\begin{equation*}
    \tikz{
      \node[label=below:{$e_{\Bi}$}] at (-4.5,0){ 
        \tikz[very thick,xscale=1.2]{
          \draw (-.5,-.5)-- (-.5,.5) node[below,at start]{$i_1$};
          \draw (0,-.5)-- (0,.5) node[below,at start]{$i_2$};
          \draw (1.5,-.5)-- (1.5,.5) node[below,at start]{$i_n$};
          \node at (.75,0){$\cdots$};
        }
      };
      \node[label=below:{$y_k^\Bi$}] at (0,0){ 
        \tikz[very thick,xscale=1.2]{
          \draw (-.5,-.5)-- (-.5,.5) node[below,at start]{$i_1$};
          \draw (.5,-.5)-- (.5,.5) node [midway,fill=black,circle,inner
          sep=2pt]{} node[below,at start]{$i_j$};
          \draw (1.5,-.5)-- (1.5,.5) node[below,at start]{$i_n$};
          \node at (1,0){$\cdots$};
          \node at (0,0){$\cdots$};
        }
      };
      \node[label=below:{$\psi_k^\Bi$}] at (4.5,0){ 
        \tikz[very thick,xscale=1.2]{
          \draw (-.5,-.5)-- (-.5,.5) node[below,at start]{$i_1$};
          \draw (.1,-.5)-- (.9,.5) node[below,at start]{$i_j$};
          \draw (.9,-.5)-- (.1,.5) node[below,at start]{$i_{j+1}$};
          \draw (1.5,-.5)-- (1.5,.5) node[below,at start]{$i_n$};
          \node at (1,0){$\cdots$};
          \node at (0,0){$\cdots$};
        }
      };
    }
  \end{equation*}

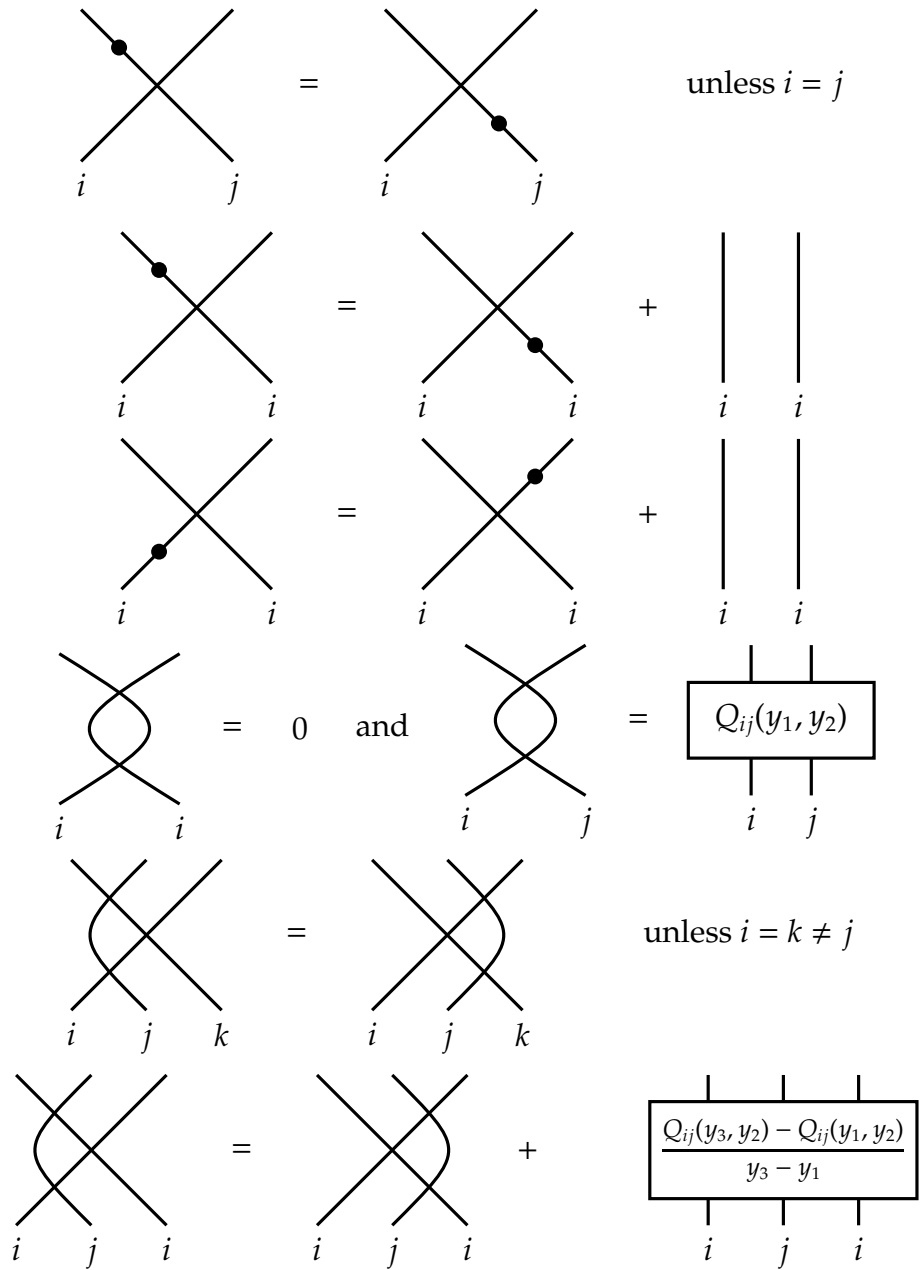
\begin{figure}
\begin{equation*}
    \begin{tikzpicture}[scale=1.01]
      \draw[very thick](-4,0) +(-1,-1) -- +(1,1) node[below,at start]
      {$i$}; \draw[very thick](-4,0) +(1,-1) -- +(-1,1) node[below,at
      start] {$j$}; \fill (-4.5,.5) circle (3pt);
      \node at (-2,0){=}; \draw[very thick](0,0) +(-1,-1) -- +(1,1)
      node[below,at start] {$i$}; \draw[very thick](0,0) +(1,-1) --
      +(-1,1) node[below,at start] {$j$}; \fill (.5,-.5) circle (3pt);
      \node at (4,0){unless $i=j$};
    \end{tikzpicture}
  \end{equation*}
  \begin{equation*}
    \begin{tikzpicture}[scale=1]
      \draw[very thick](-4,0) +(-1,-1) -- +(1,1) node[below,at start]
      {$i$}; \draw[very thick](-4,0) +(1,-1) -- +(-1,1) node[below,at
      start] {$i$}; \fill (-4.5,.5) circle (3pt);
      \node at (-2,0){=}; \draw[very thick](0,0) +(-1,-1) -- +(1,1)
      node[below,at start] {$i$}; \draw[very thick](0,0) +(1,-1) --
      +(-1,1) node[below,at start] {$i$}; \fill (.5,-.5) circle (3pt);
      \node at (2,0){+}; \draw[very thick](4,0) +(-1,-1) -- +(-1,1)
      node[below,at start] {$i$}; \draw[very thick](4,0) +(0,-1) --
      +(0,1) node[below,at start] {$i$};
    \end{tikzpicture}
  \end{equation*}
 \begin{equation*}
    \begin{tikzpicture}[scale=1]
      \draw[very thick](-4,0) +(-1,-1) -- +(1,1) node[below,at start]
      {$i$}; \draw[very thick](-4,0) +(1,-1) -- +(-1,1) node[below,at
      start] {$i$}; \fill (-4.5,-.5) circle (3pt);
            \node at (-2,0){=}; \draw[very thick](0,0) +(-1,-1) -- +(1,1)
      node[below,at start] {$i$}; \draw[very thick](0,0) +(1,-1) --
      +(-1,1) node[below,at start] {$i$}; \fill (.5,.5) circle (3pt);
      \node at (2,0){+}; \draw[very thick](4,0) +(-1,-1) -- +(-1,1)
      node[below,at start] {$i$}; \draw[very thick](4,0) +(0,-1) --
      +(0,1) node[below,at start] {$i$};
    \end{tikzpicture}
  \end{equation*}
  \begin{equation*}
    \begin{tikzpicture}[very thick,scale=1]
      \draw (-2.8,0) +(0,-1) .. controls (-1.2,0) ..  +(0,1)
      node[below,at start]{$i$}; \draw (-1.2,0) +(0,-1) .. controls
      (-2.8,0) ..  +(0,1) node[below,at start]{$i$}; \node at (-.5,0)
      {=}; \node at (0.4,0) {$0$};
\node at (1.5,.05) {and};
    \end{tikzpicture}
\hspace{.4cm}
    \begin{tikzpicture}[very thick,scale=1]

      \draw (-2.8,0) +(0,-1) .. controls (-1.2,0) ..  +(0,1)
      node[below,at start]{$i$}; \draw (-1.2,0) +(0,-1) .. controls
      (-2.8,0) ..  +(0,1) node[below,at start]{$j$}; \node at (-.5,0)
      {=};

\draw (1.8,0) +(0,-1) -- +(0,1) node[below,at start]{$j$};
      \draw (1,0) +(0,-1) -- +(0,1) node[below,at start]{$i$}; 
\node[inner xsep=10pt,fill=white,draw,inner ysep=8pt] at (1.4,0) {$Q_{ij}(y_1,y_2)$};
    \end{tikzpicture}
  \end{equation*}
  \begin{equation*}
    \begin{tikzpicture}[very thick,scale=1]
      \draw (-3,0) +(1,-1) -- +(-1,1) node[below,at start]{$k$}; \draw
      (-3,0) +(-1,-1) -- +(1,1) node[below,at start]{$i$}; \draw
      (-3,0) +(0,-1) .. controls (-4,0) ..  +(0,1) node[below,at
      start]{$j$}; \node at (-1,0) {=}; \draw (1,0) +(1,-1) -- +(-1,1)
      node[below,at start]{$k$}; \draw (1,0) +(-1,-1) -- +(1,1)
      node[below,at start]{$i$}; \draw (1,0) +(0,-1) .. controls
      (2,0) ..  +(0,1) node[below,at start]{$j$}; \node at (5,0)
      {unless $i=k\neq j$};
    \end{tikzpicture}
  \end{equation*}
  \begin{equation*}
    \begin{tikzpicture}[very thick,scale=1]
      \draw (-3,0) +(1,-1) -- +(-1,1) node[below,at start]{$i$}; \draw
      (-3,0) +(-1,-1) -- +(1,1) node[below,at start]{$i$}; \draw
      (-3,0) +(0,-1) .. controls (-4,0) ..  +(0,1) node[below,at
      start]{$j$}; \node at (-1,0) {=}; \draw (1,0) +(1,-1) -- +(-1,1)
      node[below,at start]{$i$}; \draw (1,0) +(-1,-1) -- +(1,1)
      node[below,at start]{$i$}; \draw (1,0) +(0,-1) .. controls
     (2,0) ..  +(0,1) node[below,at start]{$j$}; \node at (2.8,0)
      {$+$};        \draw (6.2,0)
      +(1,-1) -- +(1,1) node[below,at start]{$i$}; \draw (6.2,0)
      +(-1,-1) -- +(-1,1) node[below,at start]{$i$}; \draw (6.2,0)
      +(0,-1) -- +(0,1) node[below,at start]{$j$}; 
\node[inner ysep=8pt,inner xsep=5pt,fill=white,draw,scale=.8] at (6.2,0){$\displaystyle \frac{Q_{ij}(y_3,y_2)-Q_{ij}(y_1,y_2)}{y_3-y_1}$};
    \end{tikzpicture}
  \end{equation*}
\caption{The relations of the KLR algebra.  
\vspace{2.5cm}
$\mbox{}$
}
\label{quiver-hecke}
\end{figure}

In order to arrive at the KLR algebra $R$, we must impose the relations shown in
Figure \ref{quiver-hecke}. All
of these relations are local in nature, that is, if we recognize a
small piece of a diagram which looks like the LHS of a relation, we
can replace it with the RHS, leaving the rest unchanged.  The
relations depend on a choice of a polynomial $Q_{ij}(u,v)\in \K[u,v]$
for each pair $i\neq j$. Let $C=(c_{ij})$ be the Cartan matrix of $\fg$ and $d_i$
be co-prime integers so that $d_jc_{ij}=d_ic_{ji}$. We assume each polynomial is homogeneous
of degree $\langle\al_i,\al_j\rangle= -2d_jc_{ij}=-2d_ic_{ji}$ where
$u$ has degree $2d_i$ and $v$ degree $2d_j$.   We will always
assume that the leading order of $Q_{ij}$ in $u$ is $-c_{ji}$, and
that $Q_{ij}(u,v)=Q_{ji}(v,u)$. 

\begin{rem}
Since we use some results from \cite{LV}, we could constrain ourselves
to the cases they consider, where $Q_{ij}=u^{-c_{ji}}+v^{-c_{ij}}$;
however, nothing about their results depends on this choice, and for
some purposes, it seems to be better to consider a different one. For
example, this is necessary in order to define isomorphisms between KLR algebras and
Hecke algebras as in \cite{BKKL}, to define isomorphisms to
convolution algebras as in \cite{VV}, or to define a relationship to
R-matrices as in \cite{SKK}.

While some things are quite
sensitive to the choice of $\K$ and $Q_{ij}$ (for example, the
dimensions of simple $R$-modules), none of the theorems we prove will
depend on it; the reader is free to imagine that we have chosen their
favorite field and worked with it throughout.
\end{rem} 

Since the diagrams allowed in $R$ never change the sum of the simple
roots labeling the strands, $R$ breaks up as a direct sum of algebras
$R\cong \oplus_{\nu \in Q^+} R(\nu)$,
where $Q^+$ is the positive part of the root lattice, and for $\nu =
\sum a_i \alpha_i$, $R(\nu)$ is the span of the diagrams with exactly $a_i$ strings colored with each simple root $\alpha_i$.  In particular,
for any simple $R$-module $L$, there is a unique $\nu$ such that
$R(\nu)\cdot L=L$.  We call this the {\bf weight} of $L$.  We let
$\mathscr{L}_i$ denote the unique 1-dimensional simple module of $R({\al_i})$.

It is shown in \cite[2.5]{KLI} that, for all $\nu$, 
\begin{equation} \label{eq:basis}
\left\{ \psi_\sigma \left( \prod_{k=1}^n (y_k^\Bi)^{r_k}  \right) e_{\Bi} \mid \wt(\Bi)= \nu, r_1, \ldots, r_n \geq 0, \sigma \in S_n \right\}
\end{equation}
is a basis for $R(\nu)$, where, for each permutation $\sigma$, $\psi_\sigma$ is an arbitrarily chosen diagram
which permutes its strands as the permutation $\sigma$ with no double crossings.

\subsection{Crystal structure on KLR modules} \label{ssec:cs1}

\begin{defn}[\mbox{see \cite[\S2.5]{KLI}}] \label{rem:char} 
 The ``character'' of a KLR modules $M$ is 
\[\ch(M)=\sum_{\Bi}\dim (e_{\Bi}M)\cdot w[\Bi],\] an element of
$\cF$, the abelian group freely generated by words in the nodes
of the Dynkin diagram.  
\end{defn}

In \cite[1.1.4]{LV}, Lauda and Vazirani define an automorphism
$\sigma\colon R \to R$ which up to sign reflects the diagrams through
the vertical axis.   We let $M^\sigma$ denote the twist of an
$R$-module by this automorphism.

For any two positive elements $\mu, \nu$ in the root lattice there is
an inclusion $R(\mu) \otimes R(\nu) \hookrightarrow R(\mu+\nu)$ given
by horizontal juxtaposition; let $e_{\mu, \nu}$ denote the image of
the identity of $R(\mu) \otimes R(\nu)$ under this map. Let
\[\Res{\mu+\nu}{\mu,\nu}(M)=\Res{R(\mu+\nu)}{R(\mu)\otimes R(\nu)}(M)=e_{\mu, \nu}M\quad \text{ and } \quad 
\operatorname{Ind}^{\mu+\nu}_{\mu,\nu}(M)=R(\mu+\nu)\otimes_{R(\mu) \otimes R(\nu)}M\] denote the functors of restriction and
extension of scalars along this map. Note that, since the unit $e_{\mu, \nu}$ of $R(\mu) \otimes R(\nu)$ is not the same as the unit of $R(\mu+\nu)$, the underlying vector space of $\Res{\mu+\nu}{\mu,\nu}(M)$ is not $M$ but rather $e_{\mu,\nu} M$. 

\begin{defn}
Fix representations $L$ of $R(\mu) $ and $L'$ of $R(\nu)$; Define 
 \[L \circ L':=\operatorname{Ind}_{\mu,\nu}^{\mu+\nu} (L \boxtimes L').\] See \cite[\S 2.6]{KLI} for a more extensive discussion of this functor.
\end{defn}

\begin{defn}\label{integrate-out}
  For any $R(\nu'')$ module $M$ and $R(\nu)$ module $N$, let \[M\triangleleft
  N=\Hom_{R({\nu''})}(M,\Res{\nu}{\nu',\nu''} N),\]
  \[ N\triangleright M=\Hom_{ R({\nu''})}(
  M,\Res{\nu}{\nu'',\nu'} N).\]  Note that these are right adjoint to $\circ$
  in the sense that, for any $R(\nu)$ module $K$, 
  \[\Hom(M\circ N, K)\cong \Hom(M,K\triangleright
  N)\cong \Hom(N,M\triangleleft K).\]
\end{defn}

As shown in \cite[2.20]{KLI}, it follows from \eqref{eq:basis} that
\[\ch(M_1\circ M_2)=\ch(M_1) * \ch(M_2)\]  
where the product on the right is the usual shuffle product.

\begin{defn}
Let $R\text{-nmod}$ be the category of finite dimensional $R$-modules
on which all the $y_k^\Bi$'s
act nilpotently.
\end{defn}
The simple modules in $R\text{-nmod}$ coincide with the gradable modules considered in \cite{LV}; when $Q_{ij}$ is
appropriately homogeneous, the algebra $R$ can be graded as in
\cite[(9)]{KLII}, and in this case a simple can be given a compatible grading if
and only if it lies in $R\text{-nmod}$.

\begin{defn}
Let $\KLR$ be the set of isomorphism classes of simple modules in  $R\text{-nmod}$. 
\end{defn}

The following result of Lauda and Vazirani is crucial to us:

\begin{prop} (\cite[Section 5.1]{LV}) \label{prop:cos}
  The set $\KLR$ carries a bicrystal
  structure with operators defined by 
\[\tilde{e}_i L= \operatorname{cosoc} (L\circ \mathscr{L}_i) \qquad
\tilde{f}_i L= \operatorname{soc} (L\triangleright \mathscr{L}_i)),\]
\[\tilde{e}_i^* L= \operatorname{cosoc} (\mathscr{L}_i \circ L) \qquad
\tilde{f}_i^* L= \operatorname{soc}( \mathscr{L}_i \triangleleft L)),\]
and this bicrystal is isomorphic to $B(-\infty)$. The map
$(-)^\sigma\colon \QH\to\QH$ is intertwined with the Kashiwara
involution of $B(-\infty)$.
\end{prop}
\begin{rem}
As in \S\ref{ss:crystal}, since $B(-\infty)$ is a lowest weight combinatorial crystal, the functions $\varphi_i,\varphi^*_i, \varepsilon_i, \varepsilon_i^*$ are all determined by the action of the $\tilde f_i, \tilde f_i^*$. The first two also have intrinsic meaning:
$$\varphi_i(L) = \text{max} \{ n: \Res{R(\wt(L))}{R(\wt(L)-n \alpha_i) \otimes R(n \alpha_i)} (L) \neq 0 \}, \; \varphi^*_i(L) = \text{max} \{ n: \Res{R(\wt(L))}{R(n \alpha_i) \otimes R(\wt(L)-n \alpha_i)} L \neq 0 \}.$$
 \end{rem}
\begin{rem}
  Our conventions are dual to those of \cite{LV},
  since we consider $B(-\infty)$ 
  rather than $B(\infty)$. 
\end{rem}
\begin{rem}\label{rem:def-change}
  The keen eyed reader will note that the operator
  $\tilde{f}_i$ in \cite{LV} was defined slightly differently.  In our notation, it
  was defined to be $\operatorname{soc} (L\triangleright R(\al_i))$ as opposed to $\operatorname{soc} (L\triangleright \mathscr{L}_i)$.
  However, $L\triangleright \mathscr{L}_i$ is a submodule of
  $L\triangleright R(\al_i)$ via the map induced by the surjection
  $R(\al_i)\to \mathscr{L}_i$ and $L\triangleright R(\al_i)$ has a simple socle, so they in fact have the same socle.

  We prefer the definition of $\tilde f_i$ above since it generalizes more
  readily to the face crystals defined in \S\ref{sec:cryst-corr-face}, and because it uses the adjoint functor
  to that in the definition of $\te_i$. The latter imbalance could also be
  corrected by defining $\te_iL$ to be $\operatorname{cosoc} (L\circ
  R(\al_i))$.
\end{rem}

We also need the following simplified version of the the Lauda-Vazirani jump
lemma from \cite[Lemmata 6.5 \& 6.7]{LV}. Converted into our conventions:
\begin{lemma} \label{lem:jump}
Fix $L \in \KLR$. The quantity $\text{jump}_i(L)= \varphi_i(L) + \varphi_i^*(L) - \langle \wt(L), \alpha_i \rangle$ is always non-negative. Furthermore, if $\text{jump}_i(L)=0$, then $$\te_i(L)=\te_i^*(L) = L \circ \mathscr{L}_i= \mathscr{L}_i \circ L.$$
\end{lemma}

\section{Cuspidal decompositions}
\label{sec:c-good-words}

Kleshchev and Ram's work \cite{KlRa} 
uses Lyndon word combinatorics to parameterize the simple gradable KLR modules (in finite type) by a tuple of integers, one for each positive root. That is, they parameterize the simples by data which looks like Lusztig data (and in fact is Lusztig data, with respect to an appropriate reduced expression of $w_0$). Their construction however only sees the Lusztig data for certain reduced words, or equivalently certain convex orders. 
We now extend this to obtain a Lusztig datum for any convex order. We can no longer use the combinatorics on words that they develop, and instead our main tool is the notion of a cuspidal representation with respect to a charge (see Definition \ref{def:charge}). We also develop this for all symmetrizable Kac-Moody algebras, not just finite type. 

\subsection{Cuspidal decompositions for charges}

Let $\Bi=i_1\cdots i_n$ be a word in the nodes of the Dynkin diagram
and let $\al_{\Bi}=\sum_{k=1}^n\al_{i_k}$.  Fix a charge $c$, and consider the preorder $>$ on positive elements of the root lattice
induced by taking arguments with respect to this charge, as in \S\ref{ss:convex}.
\begin{defn}
  The {\bf top} of a word $\Bi$ is the maximal element which appears as
  the sum of a proper left prefix of the word; that
  is \[\top(\Bi)=\max_{1\leq j<n}\al_{i_1\cdots i_j}.\]
  We call a word in the simple roots {\bf $c$-cuspidal} if $\top(\Bi)
  <\al_{\Bi}$ and {\bf $c$-semi-cuspidal} if $\top(\Bi)\leq \al_{\Bi}$
\end{defn}

\begin{rem}
Geometrically, we can visualize a word as a path in the weight
lattice, and then picture its image in the complex plane under $c$.
A word is $c$-cuspidal if this path stays strictly clockwise of the line
from the beginning to the end of the word and $c$-semi-cuspidal if stays
weakly clockwise of this line, as shown in Figure \ref{cuspidal-picture}.
\end{rem}

\begin{figure}
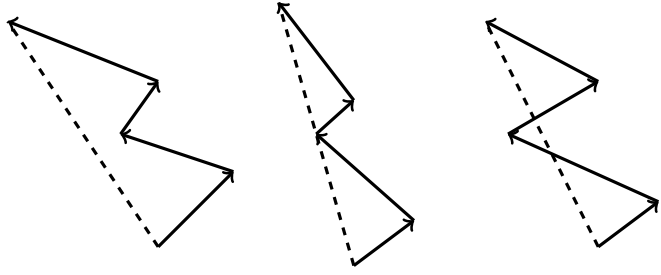

  \centering
  \tikz{
\node at (-3,0) {
\tikz[very thick]{
\draw[dashed] (0,0) -- (-2,3);
\draw[->] (0,0) -- (1,1);
\draw[->] (1,1) -- (-.5,1.5);
\draw[->] (-.5,1.5)--(0,2.2);
\draw[->] (0,2.2)--(-2,3);
}
};
\node at (0,0) {
\tikz[very thick]{
\draw[dashed] (0,0) -- (-1,3.5);
\draw[->] (0,0) -- (.8,.6);
\draw[->] (.8,.6) -- (-.5,1.75);
\draw[->] (-.5,1.75)--(0,2.2);
\draw[->] (0,2.2)--(-1,3.5);
}
};
\node at (3,0) {
\tikz[very thick]{
\draw[dashed] (0,0) -- (-1.5,3);
\draw[->] (0,0) -- (.8,.6);
\draw[->] (.8,.6) -- (-1.2,1.5);
\draw[->] (-1.2,1.5)--(0,2.2);
\draw[->] (0,2.2)--(-1.5,3);
}
};
}
  \caption{Examples of  $c$-cuspidal, $c$-semi-cuspidal, and
    non-$c$-semi-cuspidal paths.}
\label{cuspidal-picture}
\end{figure}

\begin{defn}\label{cuspidal}
  The {\bf top} of a module $L \in R(\nu)\text{-nmod}$ is the
  maximum among the tops of all $\Bi$ such that $e_{\Bi}M\neq 0$.  We
  call a simple module $L$ {\bf cuspidal} if $\top(L)<\nu$,
  and {\bf semi-cuspidal} if $\top(L)\leq\nu$.
\end{defn}

Obviously, a representation is (semi-)cuspidal if and only if
all words which appear in its character are (semi-)cuspidal. 

\begin{thm}\label{th:semi-cuspidal}
  Fix an arbitrary charge $c$. If $L_1,\dots,L_h \in \KLR$ are $c$-semi-cuspidal with $\wt(L_1)>_c
  \cdots >_c
  \wt(L_h)$, then $L_1\circ\cdots \circ L_h$ has a unique
  simple quotient.
  Furthermore, every $L \in \KLR$ appears this way for a unique
  sequence of semi-cuspidal representations. 
\end{thm}

We delay the proof of
Theorem \ref{th:semi-cuspidal} while we introduce a more general compatibility
condition on representations and prove some preliminary results.
\begin{defn} \label{def:unmixing} For $L_1, \ldots, L_h \in R\text{-nmod}$, 
  we call $(L_1,\dots,L_h)$ {\bf unmixing} if
  \[\Res{\nu_1+\cdots+\nu_h}{\nu_1,\dots,\nu_h}(L_1\circ \cdots \circ L_h)=L_1\boxtimes \cdots \boxtimes L_h.\]
\end{defn}
The notion of unmixing is important because of the following fact.

\begin{lemma}\label{lem:unmixing-quotient}
  If $(L_1,\dots,L_h) \subset \KLR^h$ is unmixing, then $L_1\circ
  \cdots \circ L_h$ has a unique simple quotient. We denote this by $A(L_1,\dots, L_h)$.
\end{lemma}
\begin{proof}
  Let $e$ denote the idempotent in $R(\nu)$ projecting to
  $\Res{\nu_1+\cdots+\nu_h}{\nu_1,\dots,\nu_h}(-)$.  Then $L_1\circ
  \cdots \circ L_h$ is generated by any non-zero vector in the image
  of $e$; thus, a submodule $M\subset L_1\circ
  \cdots \circ L_h$ is proper if and only if it is killed by $e$.  It follows that the sum of any
  two proper submodules is still killed by $e$, and thus again
  proper.  There is thus a unique maximal proper submodule of $L_1\circ
  \cdots \circ L_h$, so it has a unique simple quotient.  
\end{proof}

\begin{lemma} \label{lem:um2}
If $(L_1 \circ \cdots \circ L_{k-1}, L_k \circ \cdots \circ L_h)$ is unmixing for all $2 \leq k \leq h$
then $(L_1,\dots,L_h)$ is unmixing. 
\end{lemma}

\begin{proof}
Assume $(L_1,\dots,L_h)$ is not unmixing, so
$$\Res{\nu_1+\cdots+\nu_h}{\nu_1,\dots,\nu_h}(L_1\circ \cdots \circ L_h) \neq L_1\boxtimes \cdots \boxtimes L_h.$$
Then there is some shuffle that non-trivially mixes the factors and survives in the restriction. This involves shuffling at least one strand in some factor $L_j$ to the right, and it survives in the restriction
$$\Res{\nu_1+\cdots+\nu_h}{\nu_1+ \cdots + \nu_j,\nu_{j+1} +\dots+\nu_h}(L_1\circ \cdots \circ L_h).$$
Thus the pair $(L_1 \circ \cdots \circ L_{k-1}, L_k \circ \cdots \circ L_h)$ is not unmixing. 
\end{proof}

\begin{lemma} \label{lem:gum}
A pair $(L_1, L_2)$ of representations in $R\text{-nmod}$ is unmixing
if and only if there are no words $\Bi', \Bi'', \Bj', \Bj''$ with
$e_{\Bi'\Bi''} L_1 \neq 0, e_{\Bj'\Bj''} L_2\neq 0$ and $\alpha_{\Bi''}=\alpha_{\Bj'}$. 
\end{lemma}

\begin{proof}
  By \eqref{eq:basis} the multiplicity of a word $\Bk$
  in the character of the induction $\Res{\nu_1+\nu_2}{\nu_1,\nu_2}
  (L_1 \circ L_2)$ is the number of ways of writing $\Bk$ as a shuffle
  of a word in the character of $L_1$ with a word in the character of
  $L_2$. Now $\Bk$ must be of the form $\Bk'\Bk''$ with
  $\alpha_{\Bk'}=\nu_1$ and $\alpha_{\Bk''}= \nu_2$, so $\Bk$ is of
  the form $\text{Sh}(\Bi' , \Bj') \text{Sh}'(\Bi'', \Bj'')$, where
  $\Bi= \Bi'\Bi''$, $\Bj=\Bj'\Bj''$, $\text{Sh}, \text{Sh}'$ are
  shuffles, and $\alpha_{\Bj'}=\alpha_{\Bi''}$. The condition in the
  lemma exactly forces both $\Bj'$ and $\Bi''$ to be the trivial word,
  and hence $\Bk$ itself is a trivial shuffle of words in the
  character of $L_1$ and $L_2$. But being unmixing exactly means that
  all words in the character of of $\Res{\nu_1+\nu_2}{\nu_1,\nu_2}
  (L_1 \circ L_2)$ are in fact in $L_1 \boxtimes L_2$, so the lemma
  follows .
\end{proof}

\begin{lemma}\label{lem:semi-cuspidal-unmixing}
  If $L_1,\dots,L_h \in \KLR$ are semi-cuspidal with $\wt(L_1)>
  \cdots >
  \wt(L_h)$, then the $h$-tuple $(L_1,\dots,L_h)$ is unmixing.
\end{lemma}
\begin{proof}
It is immediate from the definition of cuspidal representation that, for all $k$, the pair $(L_1 \circ \cdots \circ L_{k-1}, L_k \circ\cdots \circ L_h)$ satisfies the conditions of Lemma \ref{lem:gum} and is thus unmixing. The lemma is then immediate by Lemma \ref{lem:um2}
\end{proof}

\begin{lemma} \label{lem:umis}
Assume $L$ is simple and every composition factor of $
\Res{\nu_1+\nu_2}{\nu_1,\nu_2} L$ is of the form $L' \boxtimes L''$ for an unmixing pair $L' \boxtimes L''$. Then $ \Res{\nu_1+\nu_2}{\nu_1,\nu_2} L$ is in fact simple.
\end{lemma}

\begin{proof}
Choose a simple quotient $L' \boxtimes L''$ of $\Res{\nu_1+\nu_2}{\nu_1,\nu_2} L$. Then there is a non-zero map $\phi: L' \circ L'' \rightarrow L$, and since $L$ is simple the image is all of $L$. Since $(L' , L'')$ is unmixing, 
\begin{equation*}
\Res{\nu_1+\nu_2}{\nu_1,\nu_2}  L=\Res{\nu_1+\nu_2}{\nu_1,\nu_2} \im \phi \simeq  L' \boxtimes L''. \qedhere
\end{equation*} \end{proof}

\begin{proof}[Proof of Theorem \ref{th:semi-cuspidal}]
By Lemmata \ref{lem:unmixing-quotient} and
\ref{lem:semi-cuspidal-unmixing},  the induction $L_1\circ\cdots \circ L_h$ has a unique
  simple quotient. 
It remains to show that every simple appears in this way for a unique
sequence of semi-cuspidals.

Fix a simple $L$. 
Consider the maximum argument $\text{arg}_{\operatorname{max}}$ of any prefix of any word in the character of $L$. 
Let $\nu_1$ be the element of the root lattice of
greatest height such that $\text{arg}_{\operatorname{max}}$ is achieved by a
prefix of weight $\nu_1$.  We proceed by induction on the
height of $\nu-\nu_1$.  If $\nu=\nu_1$, then $L$ is semi-cuspidal, and
we are done.

By assumption, $\Res{\nu}{\nu_1,\nu-\nu_1}L\neq 0$. Every
composition factor $L' \boxtimes L''$ must have the property that
no word in the character of $L''$ has a prefix with argument $\geq \arg \nu_1$, as otherwise we
could find $\nu'_1 > \nu_1$ with at least as big an argument.
Also, no word in the character of $L'$ can have a prefix of argument $> \nu_1$, which implies that no word in this character can have a proper suffix of argument $\leq \nu_1$. It
follows by Lemma \ref{lem:gum} that $(L', L'')$ is unmixing, and so by
Lemma \ref{lem:umis}, the module $\Res{\nu}{\nu_1,\nu-\nu_1}L$ is in fact a single simple $L' \boxtimes L''$. Then $L' \circ L''$ has a unique simple quotient by Lemma \ref{lem:unmixing-quotient}, and this admits a non-trivial map to $L$, so it must be $L$. 

By the inductive assumption, $L''=A(L_2,\dots, L_h)$ for some semi-cuspidals satisfying the conditions, and $\wt(L_2)$ must have argument less than $\wt(L')$ by the maximality of $\text{arg}_{\text{max}}$ and $\nu_1$. Thus $L=A(L',L_2,\dots,
L_h)$, so every simple has the desired form.  

It remains to show uniqueness. If $L=A(L_1',\dots, L_p')$ for some other cuspidal
simples with $\wt(L_1')> \cdots > \wt(L_p')$, then, by the maximality of the argument of
$\nu_1$, either $\wt(L'_1)$ has argument less then $\nu_1$, or $\wt(L'_1) =r \nu_1$ for $r\leq1$. There is a word in $L$ with weight $\nu_1$ so, unless we are in the case where $\wt(L'_1) =\nu_1$, there must be a prefix of some $L'_i$, $i \geq 2$, with argument $\geq \nu_1 \geq \wt(L'_1)$, contradicting the fact that $L'_i$ is $c$-semi-cuspidal with $\wt(L'_1) >_c \wt(L'_i)$. But then the argument above shows that $L_1'=L'$, and by induction the two lists of simples in fact   agree. 
\end{proof}

\begin{defn}
  For a fixed charge $c$ and simple $L$, we call the tuple of simples $(L_1,\dots, L_h)$ associated to $L$ by Theorem \ref{th:semi-cuspidal} the {\bf c-semi-cuspidal decomposition} of $L$.
\end{defn}

\begin{cor}\label{cor:semi-cuspidal-count}
Fix a charge $c$.
  The number of $c$-semi-cuspidals of weight $\nu$ in $\KLR$ is 
  \[\sum_{\substack{\nu=\be_1+\cdots+\be_n\\\arg c(\be_i)=\arg
      c(\nu)}} \prod_{i=1}^n m_{\be_i},\]
       the sum over the distinct ways
 of writing $\nu$ as a sum of positive roots
  $\be_*$ which all satisfy $\arg c(\be_i)=\arg c(\nu)$ of the
  product of the root multiplicities.
\end{cor}

\begin{proof}
  We proceed by induction on
  $\rho^\vee(\nu)$. If $\nu$ is a simple root, then the
  statement is obvious, providing the base case. 
 
 In general
  \begin{equation*} \label{eq:cp}
  \dim U(\fn)_\nu =\sum_{\substack{\nu=\be_1+\cdots+\be_n}} \prod_{i=1}^n m_{\be_i},
  \end{equation*}
so this is the number of isomorphism classes of simples in  $R(\nu)\text{-nmod}$. By the inductive assumption and Theorem \ref{th:semi-cuspidal}, the number of these simples that have a semi-cuspidal decomposition with at least two parts accounts for all the terms where the $c(\beta_j)$ do not all have the same argument. Thus the remaining terms give the number of semi-cuspidal simples. 
\end{proof}

\begin{cor}\label{cor:finite-type}
  If $\fg$ is finite type and $c$ is a charge such that $\arg c(\alpha) \neq \arg c(\beta)$ for all $\alpha \neq \beta \in \Delta_+$, then there is a unique cuspidal representation $\mathscr{L}_\al$
  of $R(\al)$ for each positive root $\al$, and no others. 
\end{cor}

\begin{proof}
By Corollary \ref{cor:semi-cuspidal-count} the only $\nu$ for which there is a semi-cuspidal representation are $\nu= k \alpha$ for some $k \geq 1$ and $\alpha \in \Delta_+$, and in all these cases there is only one isomorphism class $L_{k \alpha}$ of semi-cuspidal representation. The semi-cuspidal representation $L_\alpha$ of dimension $\alpha$ must in fact be cuspidal, since there is no
element of the root lattice on the line from 0 to $\alpha$.

For $k \geq 2$, $L_\alpha^{\circ k}$ is semi-cuspidal, so every
composition factor must be the unique semi-cuspidal $L_{k \alpha}$ of
weight $k \alpha$. Since $L_\alpha^{\circ k}$ is clearly only
semi-cuspidal, the representation $L_{k \alpha}$ cannot be cuspidal. 
\end{proof}

\begin{rem}
For minimal roots (i.e. roots $\alpha$ such that $x \alpha$ is not a root for any $0 < x <1$; see section \ref{ss:convex}), the same arguments used in the proof of Corollary \ref{cor:finite-type} show
that the root multiplicity coincides with the number of cuspidal
representations. However, this is not true for other roots. In \S\ref{sec:an-example} we give an example where this is false for
$\mathfrak{\widehat{sl}}_2$ with $\nu=2\delta$.
\end{rem}

\subsection{Cuspidal decompositions for general convex orders}

We now develop a generalized notion of cuspidal representation and cuspidal decomposition, where we allow any convex order on $\Delta_+^{min}$, not just those coming from charges. 

\begin{defn} \label{def:bicon-comp}
Fix a convex pre-order $\succ$ on $\Delta_+^{min}$.
We say $L \in \KLR$ is {\bf $\succ$-(semi)-cuspidal} if $\wt(L)= \nu\in
  \text{span}_{\mathbb{R}_{\geq 0}}\mathscr{C}$ for some $\succ$-equivalence
  class $\mathscr{C}$ and $L$ is $c$-(semi)-cuspidal for some $(\mathscr{C}, \succ,\langle \nu, \rho^\vee \rangle)$-compatible charge $c$ (see Definition \ref{def:c-comp}).
\end{defn}

\begin{prop} A module 
$L \in \KLR$ with  $\wt(L)= \nu\in
  \text{span}_{\mathbb{R}_{\geq 0}}\mathscr{C}$ is {\bf $\succ$-(semi)-cuspidal} if and only if $L$ is
  $c$-(semi)-cuspidal for all $(\mathscr{C}, \succ,\langle \nu, \rho^\vee \rangle)$-compatible charges $c$.
\end{prop}

\begin{proof}
 Assume that $L$ is $\succ$-cuspidal, and let $c$ be the $(\mathscr{C},
 \succ,\langle \nu, \rho^\vee \rangle)$ compatible charge from Definition \ref{def:bicon-comp}. Let
 $c'$ be another $(\mathscr{C}, \succ,\langle \nu, \rho^\vee \rangle)$-compatible charge, and assume $L$
 is not cuspidal for $c'$. Thus, there exists a weight $\be$  with $\be
 >_{c'}\mathscr{C} $ such that $\Res{\nu}{\be;\nu-\be}L\neq 0$.  In
 fact, we can assume that $\be$ is a root, by refining $
 >_{c'}$ to a
 total convex order and letting $\be$ be the first root in the
 semi-cuspidal decomposition of $L$ for the refined order.  

 Since $c'$ is $(\mathscr{C}, \succ,\langle \nu, \rho^\vee \rangle)$-compatible, this implies that $\be \succeq
  \al$. Since $c$ is also \mbox{$(\mathscr{C}, \succ,\langle \nu, \rho^\vee \rangle)$} compatible, this implies
  $\be\geq_{c} \mathscr{C}$ as well. But $L$ is $c$-cuspidal, 
  so $\Res{\nu}{\be;\nu-\be}L\neq 0$ is a contradiction. Thus $L$ is in fact cuspidal for $c'$ as well. 
  
  The same
  argument carries through for semi-cuspidality.
\end{proof}
\begin{cor} \label{cor:numcusp}
 For any convex order $\succ$ on $\Delta_+^{\min}$, the number of $\succ$-semi-cuspidal representations of weight $\nu$ is
 \[\sum_{\substack{\nu=\be_1+\cdots+\be_n\\ \be_i\in \mathscr{C} }} \prod_{i=1}^n m_{\be_i},\]
 the sum  of the
  product of the root multiplicities over the distinct ways of writing $\nu$ as a sum of positive roots
  which lie in a single equivalence class $\mathscr{C}$ for the preorder. 
In particular, if $\fg$ is finite type then there is a unique $\succ$-cuspidal $L \in \KLR$ of weight $\alpha$ for each positive root $\al$, and no others. 
\end{cor}

\begin{proof}
This follows immediately from Corollary \ref{cor:semi-cuspidal-count} and Corollary \ref{cor:finite-type} using some

$(\mathscr{C},\succ,\langle \nu, \rho^\vee \rangle)$-com\-pat\-ible charge $c$. 
\end{proof}

\begin{lemma} \label{lem:is-u}
  Fix a convex pre-order $\succ$. Any $h$-tuple $L_1,\dots,L_h$ of $\succ$-semi-cuspidal representations with $\wt(L_1)\succ
  \cdots \succ
  \wt(L_h)$ is unmixing.
\end{lemma}

\begin{proof}
Fix $1 \leq r \leq h-1$. Let $\wt (L_r) \in \text{span} \mathscr{C} $ for some equivalence class $\mathscr{C}$, and choose a $(\mathscr{C}, \succ,\langle \nu, \rho^\vee \rangle)$ compatible charge $c$. Then, for any $i \leq r, j >r$, the $c$-cuspidal decomposition for any $L_i$ with $i \leq r$ only involves representations of weight $\geq_c \wt(L_r)$, and the $c$-cuspidal decomposition of $L_j$ for $j>r$ only involve representations of weight $>_c \mathscr{C}$. Hence there can be no suffix of $L_i$ with the same weight as a prefix of $L_j$, so $(L_i, L_j)$ is unmixing by Lemma \ref{lem:gum}. This holds for all $r$, so the lemma follows by Lemma \ref{lem:um2}.  
\end{proof}

\begin{thm} \label{assignment}
 Fix a convex pre-order $\succ$. If $L_1,\dots,L_h \in \KLR$ are $\succ$-semi-cuspidal and satisfy $\wt(L_1)\succ
  \cdots \succ
  \wt(L_h)$ then $L_1 \circ \cdots \circ L_h$
  has a unique simple quotient.   Furthermore, every gradable simple appears this way for a unique
  sequence of semi-cuspidal representations. \end{thm}

\begin{proof}
By Lemmata \ref{lem:unmixing-quotient} and \ref{lem:is-u}, $L_1 \circ \cdots \circ L_h$ has a unique simple quotient.
Now we must show that every simple $L$ is of this form for a unique
$h$-tuple $L_1, \ldots, L_h$.  We can assume that we are dealing with
a total order; otherwise we can refine to a total order, and
define $L_i$ as unique quotients of the inductions of semi-cuspidals
for this finer order of each equivalence class. We proceed by induction on weight. 
 
Consider $\al\in \Delta^{min}_+$ greatest in the order $\succ$
such that
$\Res{\nu}{m\al,\nu-m\al}L\neq 0$ for some $m \geq 1$. Fix an $(\alpha, \succ,\langle \nu, \rho^\vee \rangle)$ compatible charge $c$, and consider the $c$-cuspidal decomposition $L= A(L_1, \ldots L_h)$. If $h=1$ then $L$ is $\succ$-semi-cuspidal, and we are done. 

Otherwise $\wt (L_1)= r \alpha$ for some $r>0$. By induction 
$L'= A(L_2, \ldots, L_h)$ has a unique $\succ$-cuspidal decomposition $L'= A(L'_2, \ldots, L'_s)$ and, for all $j \geq 2$, $L_j'$ satisfies $\al \succ \wt(L_j')$. Hence $L= A(L_1, L_2', \ldots, L_s')$ is an expression of the desired form. 

By Theorem \ref{th:semi-cuspidal} we know the $c$-cuspidal decomposition of $L$ is unique, so for any other such
expression $L=A(L_1'',\dots, L_p'')$, we must have $L_1\cong L_1''$, and uniqueness follows using induction again.
\end{proof}

\begin{rem}
Theorem \ref{assignment} is a generalization of \cite[Theorem 7.2]{KlRa}, which
gives exactly the same sort of description of all simple modules, but 
only applies to the convex orders arising from Lyndon words.  
The finite-type case of Theorem \ref{assignment} (and thus
Corollary \ref{cor:finite-type}) has been shown independently by
McNamara \cite[3.1]{McN}, and this has been extended to affine type by Kleshchev in \cite{Kl}.
\end{rem}

We call the sequence $L_1, \ldots, L_h$ associated to $L \in \KLR$ by Theorem \ref{assignment} the {\bf semi-cuspidal decomposition} of $L$ with respect to $\succ$.

\begin{prop} \label{prop:tpi}
  For any convex order, and any real root $\al$, the iterated
  induction 
  $\scrL^n_\al=\scrL_\al\circ \cdots \circ \scrL_\al$ irreducible, and hence is the unique
  irreducible semi-cuspidal module $\scrL_{n\al}$ of weight $n\al$.
\end{prop}
\begin{proof}
  By  Corollary \ref{cor:semi-cuspidal-count}, any composition factor of $\scrL_\al\circ \cdots \circ \scrL_\al$ must be
  semi-cuspidal of weight $n\al$, and furthermore there is only one
  semi-cuspidal simple $\scrL_{n \al}$ of this weight. Thus we need only show that $\scrL_\al\circ \cdots \circ \scrL_\al$
  cannot be an iterated extension of many copies of $\scrL_{n \al}$. 

  Choose a list $i_1,i_2,\dots$ of simple roots in which each
  simple root occurs infinitely many times. 
  Consider the string data (see Definition \ref{def-string}) of the simple $\scrL_\al$,
  considered as an element of $B(-\infty)$. By the definition of the crystal operators, this is the lexicographically maximal list of integers $(a_1,a_2, \dots)$ such that $\cdots
  i_2^{a_2}i_1^{a_1}$ occurs in the character of $\scrL_\al$.  

  The word $\cdots i_2^{na_2}i_1^{na_1}$ occurs in the
  character of $\scrL^n_\al$, and thus in the character of $\scrL_{n \al}$. Furthermore, this is the maximal word in lexicographic order in $\scrL^n_\al$, so it must be the string data of $\scrL_{n\al}$.
  
  A simple inductive argument shows that the restriction of $\scrL_\al$ to $\cdots \otimes
  R_{a_2\al_{i_2}}\otimes R_{a_1\al_{i_1}}$ is a tensor product of
  irreducible modules over nilHecke algebras, and so the word
  $\cdots i_2^{a_2}i_1^{a_1}$ occurs with multiplicity $a_1!a_2!\cdots$ (see \cite[3.7(1)]{KLI} for details). Similarly, the
  multiplicity in $\scrL_{n\al}$ of $\cdots i_2^{na_2}i_1^{na_1}$ is
  $(na_1)!(na_2)!\cdots $.

  On the other hand, the multiplicity of $\cdots i_2^{na_2}i_1^{na_1}$ in $\scrL^n_\al$
  can be computed using shuffle product.  Any word in the character
  $\scrL^n_\al$ which ends with $na_1$ instances of $i_1$ must come from  shuffling $n$ words where the final number of $i_1$'s sum to at
  least 
  $n\al_i$. By the lex-maximality of the string
  data no word in $\scrL_{\al}$ can end with more than $a_1$
  instances of $i_1$, so we can only achieve this by shuffling $n$ words
  that end in $a_1$ instances of $i_1$. Proceeding by induction, we can
  only arrive at $\cdots i_2^{na_2}i_1^{na_1}$ by shuffling $n$ copies
  of $\cdots i_2^{a_2}i_1^{a_1}$.  In each $\scrL_\al$, the multiplicity
  of this word is $a_1!a_2!\cdots
  $ as argued above. For each $j$, there are
  $(na_j)!/(a_j!)^n$ ways of shuffling the letters $i_j$ from that index together. Thus the multiplicity of
  $\cdots i_2^{na_2}i_1^{na_1}$  in the character of $\scrL^n_\al$ is also
  $$(a_1!)^n(a_2!)^n\cdots  \frac{(na_1)!}{(a_1!)^n} \frac{(na_2)!}{(a_2!)^n} \cdots = (na_1)!(na_2)!\cdots .$$  
  Comparing characters shows that $\scrL^n_\al$ can only
  contain one copy of $\scrL_{n\al}$ as a composition factor, completing the proof. 
\end{proof}

\begin{rem}
  The argument in the proof of Proposition \ref{prop:tpi} also shows that, in general, the induction $M\circ
  N$ of two simples contains a unique composition factor whose string data is the sum of those for
  $M$ and $N$; interestingly, this gives a new proof that the set of
  string parametrizations is a semi-group (in finite type it is the integral
  points of a cone).  This same argument is given by Kleshchev \cite[2.31]{Kl}
\end{rem}

\subsection{Saito reflections on \texorpdfstring{$\QH$}{KLR}}
We now discuss how the Saito
reflection from \S\ref{ss:crystal} works when the underlying set
of $B(-\infty)$ is identified with $\QH$, and specifically how it
interacts with the operation of induction.

\begin{lemma}\label{lem:unmixing}
Assume that $(L_1,L_2)$ is an unmixing pair (see
Definition \ref{def:unmixing}) with unique simple quotient $L$, and that $\varphi_i^*(L_1)=\varphi_i^*(L_2)=0$. Then, for all $n \geq 0$, $(\te_i^*)^nL$ is the unique simple quotient of 
  \begin{equation*}
    L^{(n)}=
    \begin{cases}
       (\te_i^*)^nL_1\circ L_2 & n\leq \epsilon_i(L_1)\\
      (\te_i^*)^{\epsilon_i(L_1)} L_1\circ
      (\te_i^*)^{n-\epsilon_i(L_1)} L_2 & n> \epsilon_i(L_1).
    \end{cases}
  \end{equation*}
  Similarly, if $\varphi_i(L_1)=\varphi_i(L_2)=0$, then  $(\te_i)^nL$ is the unique simple quotient of 
  \begin{equation*}
    \begin{cases}
       L_1\circ (\te_i)^n L_2 & n\leq \epsilon_i^*(L_2)\\
      (\te_i)^{n-\epsilon_i^*(L_2)} L_1\circ
      (\te_i)^{\epsilon_i^*(L_2)} L_2 & n> \epsilon_i^*(L_2).
    \end{cases}
  \end{equation*}
\end{lemma}

\begin{proof}
Since there are no words in the character of $L_1$ or $L_2$ beginning
with $i$, the triple $(\scrL_i^n,L_1,L_2)$ is unmixing.  By Lemma
\ref{lem:unmixing-quotient}, the induction
$\scrL_i^n\circ L_1\circ L_2$ has a unique simple
quotient.
Thus, if we define a surjective map $\scrL_i\circ
  L^{(n-1)}\to L^{(n)}$, this will show by induction that $L^{(n)}$
  has a unique simple quotient, and that this is
  $(\te_i^*)^nL$.

If $n\leq \epsilon_i(L_1)$, then the map is the obvious one.  If
$n>\epsilon_i(L_1)$, then by the Lauda-Vazirani jump lemma (our Lemma \ref{lem:jump}),  
 \[\scrL_i\circ
(\te_i^*)^{\epsilon_i(L_1)} L_1\cong (\te_i^*)^{\epsilon_i(L_1)}
L_1\circ \scrL_i,\] 
so we have that \[\scrL_i\circ
  L^{(n-1)}\cong (\te_i^*)^{\epsilon_i(L_1)} L_1\circ \scrL_i\circ
  (\te_i^*)^{n-1-\epsilon_i(L_1)} L_2\] which has an obvious surjective
  map  to $L^{(n)}$.  
  The second statement follows by a symmetric argument.
\end{proof}

If $\varphi_i^*(L_1)=\varphi_i^*(L_2)=0$, then any
composition factor $L$ of $L_1\circ L_2$ also has $\varphi_i^*(L)=0$ by \cite[2.18]{KLI}, so Saito reflection of $L$ makes sense.

\begin{lemma}\label{lem:saito-unmixing}
  If $(L_1,L_2)$ is an unmixing pair in $\KLR^2$ such that $\varphi_i^*(L_1)=\varphi_i^*(L_2)=0$, and $(\Sai_i(L_1),\Sai_i(L_2))$ is also an  unmixing
  pair, then $\Sai_i(A(L_1, L_2))= A(\Sai_i(L_1), \Sai_i(L_2))$.
  
  More generally, if $(L_1, \ldots, L_h)$ is unmixing with $\varphi_i^*(L_i)=0$ for all $i$, and $(\Sai L_1, \ldots, \Sai L_h)$ is also unmixing, then 
  $\Sai_i(A(L_1, \ldots, L_h))= A(\Sai_i(L_1), \ldots, \Sai_i(L_h))$. 
\end{lemma}

\begin{proof}
Let $L=A(L_1, L_2)$ and $L'= A(\Sai_i(L_1), \Sai_i(L_2))$; note that
these are both simple. 
It follows from Proposition \ref{cor:comb-characterizaton2} that, for any $M \in \KLR$ with $\tf_i^*(M)=0$, and any $n \geq 0$, 
\begin{equation}
\varphi_i^*((\te^*_i)^{n}M)+
\varphi_i((\te^*_i)^{n}M)-\langle \wt((\te^*_i)^{n}M),\al_i^\vee\rangle=\max(0,\epsilon_i(M)-n), 
\end{equation}
\begin{equation}
\tf_i^n(\te^*_i)^{\epsilon_i(M)}M \cong
(\te^*_i)^{\epsilon_i(M)}\tf_i^n M \qquad \text{and} \qquad  \te_i^n(\te^*_i)^{\epsilon_i(M)}M \cong
(\te^*_i)^{\epsilon_i(M)+n} M\label{eq:2}.
\end{equation}

By Lemma \ref{lem:unmixing}, $(\te^*_i)^{\epsilon_i(L_1)+\epsilon_i(L_2)}L$  is the unique simple quotient
of $(\te^*_i)^{\epsilon_i(L_1)}L_1\circ
(\te_i^*)^{\epsilon_i(L_2)}L_2$,  and $\te_i^{\varphi_i(L_1)+\varphi_i(L_2)}L'$ is the unique simple quotient of 
\begin{align}
\te_i^{\varphi_i(L_1)+\varphi_i(L_2)-\varepsilon_i^*(\Sai_i L_2)}\Sai_i L_1\circ
\te_i^{\varepsilon_i^*(\Sai_i L_2)}\Sai_i L_2 = 
\te_i^{\varphi_i(L_1)}\Sai_i L_1\circ
\te_i^{\varphi_i(L_2)}\Sai_i L_2,
\end{align}  
where these two expression agree because, by Corollary \ref{cor:KS-diag}, we see $\varepsilon_i^*(\Sai_i L_j) = \varphi_i(L_j).$
By the definition of Saito
reflection (Definition \ref{def:ref}),  
\begin{equation*}
\te_i^{\varphi_i(L_j)}\Sai_i
L_j\cong \te_i^{\varphi_i(L_j)} (\te^*_i)^{\epsilon_i(L_j)}\tf_i^{\varphi_i(L_j)}  L_j\cong   \te_i^{\varphi_i(L_j)} \tf_i^{\varphi_i(L_j)}  (\te^*_i)^{\epsilon_i(L_j)} L_j\cong (\te^*_i)^{\epsilon_i(L_j)} L_j,
\end{equation*}
where the middle step uses \eqref{eq:2}. Thus
 \begin{equation}\label{eq:1}
  (\te^*_i)^{\epsilon_i(L_1)+\epsilon_i(L_2)}L=(\te_i)^{\varphi_i(L_1)+\varphi_i(L_2)}L'.
  \end{equation}
It follows that
\begin{align*}
  \Sai_iL& \cong  (\te^*_i)^{\epsilon_i(L)} \tf_i^{\varphi_i(L)} L &&\text{by Definition}\\
  & \cong \tf_i^{\varphi_i(L)} (\te^*_i)^{\epsilon_i(L)} L &&\text{by \eqref{eq:2}}\\
&\cong\tf_i^{\varphi_i(L_1)+\varphi_i(L_2)}\te_i^{\varphi_i(L_1)+\varphi_i(L_2)-\varphi_i(L)}(\te^*_i)^{\epsilon_i(L)}L \\
&\cong\tf_i^{\varphi_i(L_1)+\varphi_i(L_2)}\te_i^{\epsilon_i(L_1)+\epsilon_i(L_2)-\epsilon_i(L)}(\te^*_i)^{\epsilon_i(L)}L
&&\text{by additivity of weights} \\
&\cong
\tf_i^{\varphi_i(L_1)+\varphi_i(L_2)}(\te^*_i)^{\epsilon_i(L_1)+\epsilon_i(L_2)}L
&&\text{by \eqref{eq:2}}\\
&\cong
\tf_i^{\varphi_i(L_1)+\varphi_i(L_2)}\te_i^{\varphi_i(L_1)+\varphi_i(L_2)}L'
&&\text{by \eqref{eq:1}}\\
&=L'. 
\end{align*}
This completes the proof.

The iterated statement follows by a simple induction, since we have $A(L_1, \ldots, L_h)= A(L_1, A(L_2, \ldots, L_h))$. 
\end{proof}

\begin{prop} \label{prop:SaiA}
Fix a $L \in \KLR$ with $\tilde{f}_i^*L=0$, and let
$(L_1,\dots, L_h)$ be its semi-cuspidal decomposition with respect to a fixed convex pre-order
$\succ$ with $\al_i \succ \wt(L_1)$. Then the $\succ^{s_i}$-semi-cuspidal decomposition of $\Sai_i(L)$ is $(\Sai_iL_1,\dots, \Sai_iL_h)$.  In particular, $\Sai_i$ defines a bijection between $\succ$-semi-cuspidas with $\tilde{f}_i^*L=0$ and $\succ^{s_i}$-semi-cuspidals with $\tilde{f}_iL=0$. The inverse of this bijection is $\Sai_i^*$. 
\end{prop}

\begin{proof}
By Lemma \ref{refinement}, we can refine our chosen pre-order to a
total order, and the result for the pre-order is implied by the result
for this refined one.  Thus, we may assume we have a total order.

Choose a $(\mathscr{C}, \succ,\langle \wt(L), \rho^\vee \rangle)$ compatible charge $c$, where $\mathscr{C}$ is the equivalence class such that $\wt(L_1)$ is in its positive span. 
By applying Lemma \ref{lem:i-order} to $>_c$ and the hyperplane defined by $\arg(c(\nu))= \arg(c(\alpha))$, we may assume that $\alpha_i$ is greatest, since otherwise we can change the convex order without affecting the relative order of any pair of roots one of which is $\preceq \wt(L_1)$, and hence without affecting the cuspidal decomposition. 

By Lemma
\ref{lem:saito-unmixing}, it suffices to show that $(\Sai_iL_1,\dots, \Sai_iL_h)$ is unmixing. 
We proceed by induction, considering the two statements:
\begin{itemize}
\item [$(c_m)$] for any root $\al\neq \al_i$ of height at most $m$, any convex order
  $\succ$ such that $\al_i$ is greatest, and any simple $L$ which is
  $\succ$-semi-cuspidal of weight $\al$, the Saito reflection
  $\Sai_iL_\al$ is $\succ^{\alpha_i}$-semi-cuspidal of weight $s_i\al$.
\item [$(d_m)$] for any weight $\nu$ of height $m$, and any
  simple $L$ which is weight $\nu$, Proposition \ref{prop:SaiA} holds.
\end{itemize}

$(c_m)\Rightarrow (d_m)$: Fix $L$ of weight $\nu$ and let $L= A(L_1, \ldots L_h)$ be the
$\succ$-semi-cuspidal decomposition for $L$ with respect to
$\succ$. By $(c_m)$,  for all $1 \leq j \leq h$, the modules $\Sai_i L_j$ are semi-cuspidal with respect to $\succ^{s_i}$, and certainly $\wt(\Sai_i(L_1)) \succ^{s_i} \cdots \succ^{s_i} \wt( \Sai_i( L_h))$, so they are unmixing by Lemma \ref{lem:is-u}. Hence $(d_m)$ holds by Lemma \ref{lem:saito-unmixing}.

$(c_m) \text{ and } (d_m)  \Rightarrow (c_{m+1})$: Fix $\nu$ with height $m+1$. By Corollary \ref{cor:numcusp}, the sets of $\succ$-semi-cuspidals of weight $\nu$ and $\succ^{s_i}$-semi-cuspidals of weight $s_i\nu$
have the same number, and we know $\Sai_i$ is a bijection between the
set of simples $L'$ with $\tilde{f}_i^*L'=0$ and those with
$\tilde{f}_iL'=0$. Using $(d_m)$ and Lemma \ref{lem:saito-unmixing}, we see that if $L$
satisfies $\tilde{f}_i^*L=0$ and is not
semi-cuspidal, then $\Sai_iL$ will likewise not be semi-cuspidal.  The
pigeonhole principle thus implies that if $L$ {\it is} semi-cuspidal,
then $\Sai_iL$ must be as well.

The result follows by induction, using the trivial statement $c_0$ as the base case. 
\end{proof}

Putting Proposition \ref{prop:SaiA} another way:

\begin{cor}\label{saito-A} Fix a convex pre-order $\succ$ and assume that 
  $L_1,\dots,L_h$ are $\succ$-semi-cuspidal representations with
  $\al_i \succ\wt(L_1)\succ
  \cdots \succ
  \wt(L_h)$.  Then 
$$
  \Sai_iA(L_1,\dots, L_h)\cong A(\Sai_iL_1,\dots, \Sai_iL_h).  \qed
$$
\end{cor}

\begin{rem}
As was recently explained by Kato \cite{Kato??}, in symmetric finite type there are in fact equivalences of categories
\begin{equation}
\left\{ 
\begin{array}{l}
L \in R\text{-nmod}: \Res{\wt(L)}{\alpha_i,\wt(L)-\alpha_i}(L)=0
\end{array}
\right\}
\leftrightarrow
\left\{ 
\begin{array}{l}
L \in R\text{-nmod}: \Res{\wt(L)}{\wt(L)-\alpha_i, \alpha_i}(L)=0
\end{array}
\right\}
\end{equation}
which induce Saito reflections on the set of simples. Kato's proof uses the geometry of quiver varieties, which is why it is only valid in symmetric type, but it seems likely that
there is an algebraic version of Kato's functor as well, which should extend his result to all symmetrizable types. We feel this should give an alternative and perhaps more satisfying explanation for Proposition \ref{prop:SaiA} and Corollary \ref{saito-A}.
\end{rem}

Corollary \ref{saito-A} is a very important technical tool for us. In particular
it often allows us to reduce questions about cuspidal representations to the
case where the root is simple, using the following.

\begin{lemma} \label{le:saito-to-simple}
Fix a simple $L$ and a convex pre-order $\succ$, and assume the semi-cuspidal decomposition of $L$ is 
$L = A(L_1, \ldots, L_h).$

If $L_1=\mathscr{L}_{n\al}$ for some real root $\al$ which is accessible from below then there is a finite sequence $\sigma_{i_1}, \dots, \sigma_{i_k}$ of Saito reflections such that $s_{i_k} \cdots s_{i_1} \alpha$ is a simple root $\alpha_m$, for each $j$ we have $\varphi_j^*(\sigma_{i_{j-1}} \cdots \sigma_{i_1} L)=0$, and
 \[\sigma_{i_k} \cdots \sigma_{i_1} L=A(\mathscr{L}_{\al_m}^{n},\dots, \sigma_{i_k} \cdots \sigma_{i_1} L_h).\]
 In particular, this holds for all $\alpha$ in finite type, and all $\alpha \succ \delta$ in affine type.

If instead  $L_h=\mathscr{L}_{p\be}$ where $\beta$ is accessible from above, then there is a similar list of
dual Saito reflections $\sigma_{i_1}^*, \dots, \sigma_{i_h}^*$ with
\[\sigma_{i_h}^* \cdots \sigma_{i_1}^*
L=A(\sigma_{i_k}^* \cdots \sigma_{i_1}^*
  L_1,\dots, \mathscr{L}_{\al_\ell}^{p}).\]
   In particular, this holds for all $\beta$ in finite type, and all $\beta \prec \delta$ in affine type.
\end{lemma}
\begin{proof}
The two statements are swapped by the Kashiwara involution, so we need
only prove the first.  We proceed by induction on the number of 
positive roots $\eta \succ\al$ (which is 
  finite because $\alpha$ is accessible from below). The case when $\alpha$ is greatest with respect to $\succ$ and hence is a simple root is trivially true; so assume that for some $k \geq 1$ the statement is known for all pairs consisting of a root $\alpha$ and a preorder $\succ$ with at most $k-1$ positive
  roots $\eta \succ\al$. 

Fix $\succ$ and $\alpha$ with exactly $k$ roots $\succ \alpha$. Let $\al_{i_1}$ be the greatest root (which is
  necessarily simple).   Then $\vp^*_{i_1}(L)=0$, since
  $\mathscr{L}_{\al_{i_1}}$ does not appear in its cuspidal
  decomposition, and so we can apply Corollary \ref{saito-A}. This reduces to the same questions with $\succ^{s_i}$, and $s_i(\alpha)$ and $\Sai_{i_1}L$.
  Furthermore, the positive roots $\be\succ_{s_i} s_{i_1}\al $
  are exactly those of the form $\be=s_i\be'$ for $\be'\succ\al$ with 
  $\be'\neq \al_{i_1}$, so there are one fewer of these then for $\al$
  and $\succ$, and the induction proceeds until we have found the
  desired sequence.  

By Proposition \ref{prop:SaiA},
  the modules $(\sigma_{i_k} \cdots \sigma_{i_1} L_1,\dots,
  \sigma_{i_k} \cdots \sigma_{i_1} L_h)$ are the semi-cuspidal
  decomposition of $L$ with respect to the convex order $\succ^{s_{i_k} \cdots
    s_{i_1}}=(\cdots (\succ^{s_{i_k}})^{s_{i_2}}\cdots )^{s_{i_1}}.$
  Since $s_{i_k} \cdots
    s_{i_1}\al_{i_1}=\al_m$, we have that $\sigma_{i_k} \cdots
    \sigma_{i_1} L_1\cong \mathscr{L}_{\al_m}^n$.
\end{proof}

\section{KLR polytopes and MV polytopes}
\label{sec:KLR-MV}

\subsection{KLR polytopes} \label{ss:klrp}

\begin{defn}
  For each $L \in \KLR$, the {\bf character polytope} $P_L$ is the convex hull
 of the weights $\nu'$ such that
$\Res{\nu}{{\nu'}, {\nu-\nu'}}L\neq0$.
\end{defn}

\begin{rem} Recalling the definition of the character $\ch(L)$ of $L$ from Remark \ref{rem:char},  we can think of every word $\Bi$ appearing in $\ch(L)$
as a path in $\fh^*$; the polytope $P_L$ can also be described as the
convex hull of all these paths. This explains our terminology. 
\end{rem}

\begin{example}
  Let $\fg=\mathfrak{sl}_3$, and $\nu=2\al_1+\al_2$. Consider 
  the algebra $R(\nu)$, where we take
  $Q_{12}(u,v)= u+v$. Then $R(\nu)$ has 2 simple modules. In fact, it turns out that the
 module $\scrL_1\circ \scrL_2 \circ \scrL_1$ is semi-simple, and these are the two simple summands. Specifically, the subspace $L'$ spanned by the three diagrams
  
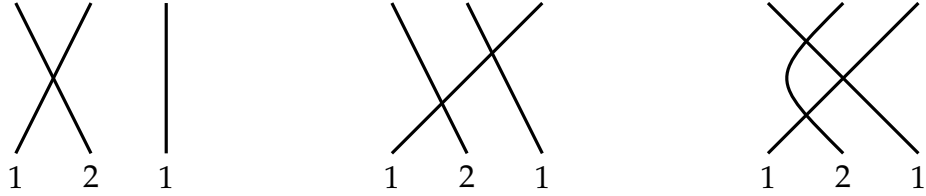
\begin{equation*}
    \begin{tikzpicture}[very thick,scale=1]

      \draw (0,0)--(1,2) node[below,at start]{$1$}; 
      \draw (1,0)--(0,2) node[below,at start]{$2$};
      \draw (2,0)--(2,2) node[below,at start]{$1$};

      \draw (5,0)--(7,2) node[below,at start]{$1$}; 
      \draw (6,0)--(5,2) node[below,at start]{$2$};
      \draw (7,0)--(6,2) node[below,at start]{$1$};
      
      \draw (10,0)--(12,2) node[below,at start]{$1$}; 
      \draw (11,0) .. controls (10,1) and (10,1) .. (11,2) node[below,at start]{$2$};
      \draw (12,0)--(10,2) node[below,at start]{$1$};
      
        \end{tikzpicture}
  \end{equation*}
 is one of the summands, and the other is spanned by the three diagrams obtained from these by flipping about a vertical axis. 

The characters of these modules are 
\[\ch(L)=2 w[112] + w[121]\qquad \ch(L')=2w[211] +
w[121].\]
The Kashiwara involution switches these simples.  From the characters,
we can read off their character polytopes:
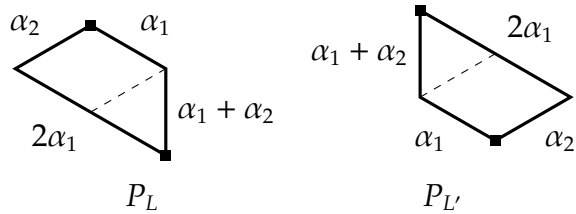
\nc{\opp}{.577350269}
\[
\begin{tikzpicture}
  \node(bba) [label=below:$P_L$]  at (-2,0){
\tikz[very thick,scale=1] {
\draw[dashed,thin] (-1,\opp) -- (0,2*\opp);
\draw (0,0)-- node[midway,below
  left]{$2\al_1$} (-2,2*\opp)-- node[midway,above left]{$\al_2$} (-1,3*\opp)
  node[inner sep=2pt, fill=black, at end] {}; \draw (-1,3*\opp)
  --node[midway,above right]{$\al_1$}(0,2*\opp) -- node[midway,right]{$\al_1+\al_2$}(0,0)
  node[inner sep=2pt,
  fill=black, at end] {};}
};
\node(abb) [label=below:$P_{L'}$] at (2,0){
\tikz[very thick,scale=1] {
\draw[dashed,thin] (-1,\opp) -- (0,2*\opp);
\draw (0,0)-- node[midway,below
  left]{$\al_1$} (-1,1*\opp)-- node[midway, left]{$\al_1+\al_2$} (-1,3*\opp)
  node[inner sep=2pt, fill=black, at end] {}; \draw (-1,3*\opp) --
  node[midway,above right]{$2\al_1$} (1,1*\opp) --node[midway,below right]{$\al_2$} (0,0)
  node[inner sep=2pt,
  fill=black, at end] {};}
};
\end{tikzpicture}
\]
\end{example}

The polytopes $P_L$ live in the $\mathfrak{h}^*$, which has
a natural height function given by pairing with $\rho^\vee$. This orients each edge of the polytope, and
gives every face $F$ a highest vertex $v_t$ and lowest vertex $v_b$.
We associate a KLR algebra $R_F$ to each face $F$ by $R_F:=R({v_t-v_b})$ and consider the subalgebra $R({v_b}) \otimes R_F\otimes R({\nu-v_t})$ of $R(\nu)$. Let $\Res{\nu}{F}$ be the functor restricting to this subalgebra.
\begin{prop} \label{prop:3irr}
  For any simple $L$ and face $F$ of $P_L$, the restriction $\Res\nu F L$ is simple and thus the outer tensor of three simples
  $L'\boxtimes L_F\boxtimes L''$.
  Furthermore, $L$ is the unique simple quotient of $L' \circ L_F \circ L''$. 
\end{prop}

\begin{proof}
Choose a linear function $\phi: \mathfrak{h}^* \rightarrow {\Bbb R}$ that obtains its minimum on $P_L$
exactly on $F$, and consider the charge $c=\phi + i \rho^\vee$. By Theorem \ref{th:semi-cuspidal}, 
$L$ is the unique simple quotient of  $L_1\circ\cdots \circ L_h$ for some semi-cuspidals with $\wt(L_1) >_c \cdots >_c \wt(L_h)$. 

If there is some index $k$ such that $\phi(\wt(L_k))=0$, let $L_F= L_k$, let $L'$
be the simple quotient of $L_1 \circ\cdots \circ L_{k-1}$, and let 
$L''$ be the simple quotient of $L_{k+1} \circ\cdots \circ L_{h}$.
Then $L'\circ L_F\circ L''$ is a quotient of $L_1\circ\cdots \circ
L_h$, and thus has a unique simple quotient, which is $L$.  On the
other hand $(L_1, \ldots, L_h)$ is unmixing, which implies $(L', L_F, L'')$ is also unmixing, so
$\Res\nu F (L'\circ L_F\circ L'')=L'\boxtimes L_F\boxtimes L''$, and so
$L$ must also restrict to this same module.

If there is no $k$ such that $\phi(\wt(L_k))=0$, then let $k$ be maximal such that $\phi(\wt(L_k))<0$. Then the same argument applies with 
$L_F= 0$, $L'$ the simple quotient of $L_1 \circ\cdots \circ L_{k}$, and
$L''$ the simple quotient of $L_{k+1} \circ\cdots \circ L_{h}$.
\end{proof}

\begin{defn}\label{def-KLR-poly}
Fix $L \in \QH$.
The {\bf KLR polytope} $\tilde P_L$ of $L$ is the polytope $P_L$ along with the data of the isomorphism class of the semi-cuspidal representation $L_E$ associated to each edge $E$ of $P_L$ in Proposition \ref{prop:3irr}. We denote by $\sfP^\QH$ the set of all KLR polytopes. 
\end{defn}

\begin{rem} \label{rem:not-arbitrary} 
The representations which can appear as the label of an edge $E$ in $\tilde P_L$ are not
arbitrary; they must be semi-cuspidal for any convex order $\succ$ such that $E$ is contained in the path $P_L^{\succ}$ from Lemma \ref{lem:ispath}. \end{rem}

\begin{prop} \label{prop:pseudo-weyl}
 Every edge of $P_L$ is parallel to a positive root of $\fg$. That is, $P_L$ is a pseudo-Weyl polytope. 
\end{prop}
\begin{proof}
  For any edge $E$, pick a functional $\phi$ which achieves
  its minimum on $P_L$ exactly on $E$ and consider the charge $c=\phi+ i \rho^\vee$. Since at most one element of $\Delta_+^{\min}$ is parallel to $E$, we can ensure that $\phi(\alpha)=0$ for at most one $\alpha \in \Delta_+^{min}$. But  $L_E$ is semi-cuspidal for $c$ so by Corollary \ref{cor:semi-cuspidal-count} $\wt(L_E)$ is a multiple of a positive root, and hence $E$ is parallel to that root.
\end{proof}

\begin{rem} \label{rem:dec-sup}
In finite type, Corollary \ref{cor:semi-cuspidal-count} and Remark \ref{rem:not-arbitrary} show that there is only ever one possible label for a given edge, so $\tilde P_L$ is completely determined by the character polytope $P_L$. Hence $\sfP^{\QH}$ can be thought of as simply a set of pseudo-Weyl polytopes. 
\end{rem}

As in Lemma \ref{lem:ispath}, each convex order $\succ$ defines a path $P_L^\succ$
through $P_L$. We obtain a list of simple
modules $L_1,\dots, L_h$ with $\wt(L_1)\succ\cdots \succ \wt(L_h)$
by taking the modules corresponding to the edges in $P^\succ$. 

\begin{prop} \label{prop:LA}
For any simple $L$ and any convex order $\succ$,
$$L=A(L_1,\dots,
  L_h),$$ where $L_1, \ldots, L_h$ are as described above.
\end{prop}
\begin{proof}
  We induct on $h$, the case when $L$ is $\succ$-semi-cuspidal being obvious. Let $E$ be the top edge in $P^\succ$, and consider $\Res\nu EL$; by Proposition \ref{prop:3irr} this is of the form
  $L'\boxtimes L_h$.  Obviously
  the edges in $P_{L'}^\succ$ and $P_L^\succ$ coincide up to but not including $E$.
  Thus,
  $L_1,\dots, L_{h-1}$ are the simples associated to this walk for $L'$ by the
  algorithm above, and by the inductive assumption, $L'=A(L_1,\dots, L_{h-1}).$  Thus, $A(L_1,\dots, L_h)$ is the
  unique simple quotient of $L'\circ L_h$, which, again by Proposition \ref{prop:3irr}, is equal to $L$.
  \end{proof}

Proposition \ref{prop:LA} has the following immediate consequences:

\begin{cor} \label{cor:bij1} For any $L \in \KLR$ and any convex order $\succ$, 
  the polytope $P_L$ with the labeling of just its edges along $P^\succ$ uniquely
  determines $L$.  In particular, the map $L\mapsto \tilde P_L$ is
  a bijection $\QH \to \sfP^{\QH}$. \qed
\end{cor}
\begin{cor} \label{cor:bij2} For any convex order $\succ$
  the function sending a labelled polytope to the list of semi-cuspidal
  representations attached to $P^\succ$ is a bijection from
  $\sfP^{\QH}$ to the set of 
  ordered lists of semi-cuspidal representations. \qed
\end{cor}

\begin{rem}
As mentioned in the introduction, Corollaries \ref{cor:bij1} and
\ref{cor:bij2} essentially mean that the semi-cuspidal decompositions
of $L \in \KLR$ with respect to convex orders can be thought of as
``general type" Lusztig data for $L$. So we have made precise and proven Theorem \ref{thmC}.
\end{rem}

Since the map which takes $L$ to $\tilde P_L$ is injective, the crystal structure on $\QH$ gives rise to a crystal structure on $\sfP^\QH$. Using Corollary \ref{cor:bij2}, we can now describe the resulting crystal operators. In the discussion below we repeatedly use the fact that, for any simple root $\alpha_i$, the unique semi-cuspidal $L_{n \alpha_i}$ of weight $n \alpha_i$ is exactly the induction $L_i^{\circ n}$; this is a special case of Proposition \ref{prop:tpi}, but is also a standard fact about the nil-Hecke algebra. 

\begin{prop}\label{polytope-crystal}
  To apply the operator $\tilde{f}_i$ to $\tilde P\in \sfP^\QH$,
  choose a convex order with $\al_i$ lowest, and read the path
  determined by that order to obtain a list of semi-cuspidal
  representations $L_1,\cdots, L_h$ corresponding to increasing roots
  in that order. If $L_h = \mathscr{L}_i^k$ for some $k \geq 1$, then \[
  \tilde{f}_i \tilde P= \tilde P_{A(L_1,\dots,L_{h-1}, \mathscr{L}_i^{k-1})}.\] If
  $L_h\ncong \mathscr{L}_i^k$, then $\tilde{f}_i \tilde
  P=0$.  \end{prop}

\begin{proof}
 If $L_h\ncong
\mathscr{L}_i^k$, then $L= A(L_1,\dots, L_h)$ is a quotient of $L_1\circ
\dots \circ L_h$, whose character is a quantum shuffle of words not
ending in $i$, and thus contains no words ending in $i$. Hence
$\tilde{f}_i A(L_1,\dots, L_h)=0$ by definition.  

If $L=A(L_1,\dots,L_{h-1}, \mathscr{L}_i^{k})$ for $k \geq 1$, then L is the unique simple quotient of
\[A(L_1,\dots,L_{h-1})\circ  \mathscr{L}_i^{k}\cong
A(L_1,\dots,L_{h-1})\circ  \mathscr{L}_i^{k-1} \circ \mathscr{L}_i.
\]
This surjects onto  
\[ A(L_1,\dots,L_{h-1},  \mathscr{L}_i^{k-1})\circ   \mathscr L_i,\] 
so 
by definition $\tilde{e}_i A(L_1,\dots,L_{h-1},
\mathscr{L}_i^{k-1})=L$.
Since we know the $\e_i, \f_i$ are a crystal structure, the result follows form the properties in Definition \ref{def:crystal}.
\end{proof}

We also have the following, which is simply a restatement of Corollary \ref{saito-A} in the language of polytopes.
\begin{cor}\label{cor:polytope-Saito}
  To apply a Saito reflection functor $\Sai_i$ to a polytope $\tilde P\in \sfP^\QH$ with
  $\tilde{f}_i^* \tilde P=0$, choose a convex order with $\al_i$ greatest and let $L_1,\cdots, L_h$ be
  as before.  Then \[ \Sai_i \tilde P=\tilde P_{A(\Sai_iL_1,\dots,\Sai_iL_h)}.   \] 
  Similarly, to apply $\Sai^*_i$ to a polytope $\tilde P\in \sfP^\QH$ with
  $\tilde{f}_i \tilde P=0$, choose a convex order with $\al_i$ least and let $L_1,\cdots, L_h$ be
  as before.  Then $\Sai^*_i \tilde P=\tilde P_{A(\Sai^*_iL_1,\dots,\Sai^*_iL_h)}.$
  \qed
\end{cor}

Comparing Propositions \ref{polytope-crystal} and Corollary \ref{cor:polytope-Saito} with the definition of crystal theoretic Lusztig data it is immediate that:
\begin{cor}\label{cor:crystal-polytope}
  Fix a simple $L$ and a convex order $\succ$. Let $b$ be the element in $B(-\infty)$ corresponding to $L$. The geometric Lusztig data $a^\succ_{\al}(P_L)$ from Definition \ref{def:geom-LD} agrees with the crystal-theoretic Lusztig
  data $\coa^\succ_{\al}(b)$ from Definition \ref{ctLusztig} for all real
  roots $\alpha$ which
  are accessible from above or below. 
  \qed
\end{cor}

\begin{proof}[Proof of Theorem~\ref{th:ft-iso}]
  Two pseudo-Weyl polytopes for a finite dimensional root system coincide if
  and only if their Lusztig data are identical for every convex
  order.  By Proposition \ref{prop:cld-mv} the geometric Lusztig data of the MV polytope corresponding to $b$ is given by the
  crystal-theoretic Lusztig data $\coa_{\bullet}(b)$, and that of the KLR
  polytope $P_L$ is given by $a_{\bullet}(P_L)$ by definition.  Thus Corollary
  \ref{cor:crystal-polytope} shows that these polytopes coincide.
\end{proof}

\subsection{Face crystals}
\label{sec:cryst-corr-face}
\mbox{}
Fix a charge $c$. By \cite[Theorem 1]{Bor91}, after possibly taking a central extension and including some extra derivations, the subalgebra of
$\mathfrak{g}$ spanned by root spaces of argument $\pi/2$
is a Borcherds algebra. This can have infinite rank, but it will have only finitely many positive entries on the diagonal of its Cartan
matrix.
Let $\fg_c$ be the subalgebra of this Borcherds algebra generated by the real
roots of argument $\pi/2$. This is the Kac-Moody algebra whose Cartan matrix consists of the rows and columns with positive diagonal entries. Let $\Delta_c$ be the root system of $\fg_c$. 

\begin{rem}
To understand the definition of $\Delta_c$, it is instructive to consider $\asl_4$ with a charge $c$
such that $c(\alpha_1),c(\alpha_3),$ and $c(\delta)$ all have argument
$\pi/2$, but $c(\alpha_2)$ does not. Then $\Delta_c$ is the product of
two copies of the $\asl_2$ root system. In particular, $\Delta_c$ has
two non-parallel imaginary roots, whereas $\Delta$ has no such pair.
\end{rem}

\begin{rem}
  As we discuss in \S\ref{ss:beyond-affine} below, in general it
  may be better to see a face as associated to the Borcherds algebra discussed
  at the beginning of this section. However, in affine type, defining
  $\fg_c$ as we do has some advantages.
\end{rem}

Let
$\be_{\underline 1},\dots,\be_{\underline s}$ be the simple roots of
$\fg_c$.  To avoid confusion between the roots of $\fg$ and those of $\fg_c$, we will index
the latter with underlined numbers. As in \S\ref{sec:c-good-words}, let
$\scrL_{\be_{\underline{i}}}$ denote the unique cuspidal module for $c$ with weight
$\be_{\underline i}$.  
We can now generalize Theorem \ref{th:ft-iso} a little bit.

\begin{prop}\label{prop:real-faces} If $c$ is such that there
  are either only finitely many roots $\al$ with
$\arg(c(\al))\leq \pi/2$ or only finitely many roots with
$\arg(c(\al))\geq\pi/2$, then $\fg_c$ is finite type and,
for each $L \in \QH$, the face $F$ of $P_L$ defined by $c$ is an MV polytope for $\fg_c$.
\end{prop}
\begin{proof}
Without loss of generality we can assume $L$ is semi-cuspidal with argument
$\pi/2$. We will handle   the case where there are only finitely many roots satisfying $\arg(c(\al))\geq\pi/2$. The other case follows by a symmetric argument. 

Clearly the root system for $\fg_c$ contains only finitely many roots, so it is finite type.
Choose a convex order refining $\succ_c$. 
By Lemma \ref{le:saito-to-simple} we can find a list of reflections
$s_{i_1},\dots, s_{i_k}$ such that Corollary \ref{cor:polytope-Saito} applies to show that
that $s_{i_1}\cdots
s_{i_k} F$ is a face of the polytope $P_{\sigma_{i_1}\cdots
  \sigma_{i_k} L}$, and the roots parallel to $s_{i_1}\cdots
s_{i_k} F$ all have argument with respect to $c^{s_{i_1}\cdots
  s_{i_k}}$ greater then all other roots. 

If $s_{i_1}\cdots
s_{i_k}\beta_{\underline{j}}$ is a sum of more than one simple root, either these
all have argument $\pi/2$ for $c^{s_{i_1}\cdots
s_{i_k}}$, which contradicts the simplicity of
$\beta_{\underline{j}}$ in $\mathfrak{g}_c$, or one of the simple
roots must have argument with respect to $c^{s_{i_1}\cdots
s_{i_k} }$ greater than $s_{i_1}\cdots
s_{i_k}\beta_{\underline{j}}$, which is impossible. Thus, the simple roots  $s_{i_1}\cdots s_{i_k}\beta_{\underline{j}}$ of the reflected face root system must all be simple for the full root system $\Delta$. 

It follows that $\sigma_{i_1}\cdots \sigma_{i_k} L$ is a representation of the KLR
algebra for the finite type algebra $\fg_{c^{s_{i_1}\cdots s_{i_k}}}$;
we simply don't use any strands labeled with other roots.  Thus, $P_L$
is an MV polytope for $\fg_{c^{s_{i_1}\cdots s_{i_k}}}$ by
Theorem \ref{th:ft-iso}. 
\end{proof}

\begin{defn} \label{def:afc}
Fix a charge $c$. The {\bf face crystal} $\QH[c]$ is the set of $c$-semi-cuspidal representations $L$ of argument $\pi/2$. 
\end{defn}

In fact, $\KLR[c]$ only depends on how the argument of $c(\beta)$ compares with $\pi/2$ for each root $\beta$. In particular, we see:

\begin{lemma} \label{lem:justclass}
If $c, c'$ are two charges and for every $\alpha \in \Delta$, 
$\arg c(\alpha) < \pi/2$ (resp. $\arg c(\alpha) > \pi/2$) if and only if $\arg c'(\alpha)< \pi/2$ (resp. $\arg c'(\alpha)> \pi/2$), then $\KLR[c]=\KLR[c']$. 
\end{lemma}

\begin{proof}
Fix $L \in \KLR$ of weight $\nu$ such that $\arg c(\nu)=\pi/2.$ Then
$L \in \KLR[c]$ if and only if $\Res{\nu}{\nu-\be,\be}L=0$ for all
$\beta$ with $\arg c(\beta) < \pi/2$. This is the same condition as  
$\Res{\nu}{\nu-\be,\be}L=0$ for all $\beta$ with $\arg c'(\beta) <
\pi/2$, since this is the same set of roots.  Thus, we have that $L \in \KLR[c']$.
\end{proof}

In fact,
$\QH[c]$ consists exactly of those representations which occur as the
representation $L_F$ associated in \S\ref{ss:klrp} to the face $F$ where the real part of $c$ takes on its minimal value, for some $L$. 
This justifies the term ``face" in Definition
\ref{def:afc}.  We now explain the term ``crystal."  

\begin{defn} \label{def:face-operator}For $L \in \QH[c]$, define
  \begin{equation*}
\begin{aligned}
& \tilde {\sse}_{{\underline{i}}}L=\cosoc(L\circ \mathscr{L}_{\be_{\underline{i}}}),
&& \tilde \sse_{{\underline{i}}}^*L=\cosoc(\mathscr{L}_{\be_{\underline{i}}} \circ L), \\
&\tilde \ssf_{{\underline{i}}}L=\soc(L \triangleright \mathscr{L}_{\be_{\underline{i}}}), 
&& \tilde \ssf_{{\underline{i}}}^*L=\soc(\mathscr{L}_{\be_{\underline{i}}}\triangleleft L) 
\end{aligned}
\end{equation*}
$$\varphi_{\underline{i}}(L)=\max\{n \mid \text{Res}^\nu_{\nu-n \beta_{\underline{i}}, n  \beta_{\underline{i}}  }  L \neq0\} \qquad
\varphi^*_{\underline{i}}(L)=\max\{n \mid \text{Res}^\nu_{ n
  \beta_{\underline{i}},\nu-n \beta_{\underline{i}} }  L \neq0\},$$
\[\varepsilon_{\underline i}= \varphi_{\underline i}- \langle \wt(L),
\beta_{\underline i}^\vee \rangle,\qquad   \varepsilon^*_{\underline i}=
\varphi^*_{\underline i}- \langle \wt(L), \beta_{\underline i}^\vee
\rangle,\] 
where $\triangleleft,  \triangleright, \circ$ and $\text{Res}$ are as in \S\ref{ssec:cs1} and $\beta_{\underline i}^\vee$
 is the co-root with respect to $\Delta_c$.
\end{defn}

If $\be_{\underline{i}}=\al_j$ is a simple root for $\Delta$, then these operations agree with
the crystal operators $\te_j,\tf_j$ from Proposition \ref{prop:cos}. This is precisely why we have modified the definition of $\tf_i$ from that used in \cite{LV}, as discussed in Remark
\ref{rem:def-change}.  It is also possible to
give a definition of $\tilde \ssf_{{\underline{i}}}$ generalizing that
of \cite{LV}, by replacing $\mathscr{L}_{\be_{\underline{i}}}$ with
its projective cover.

\begin{prop}\label{e-f-defined}
  For every $L\in \QH[c]$, the modules $\tilde
  {\sse}_{{\underline{i}}}L, \tilde {\sse}_{{\underline{i}}}^*L,\tilde
  {\ssf}_{{\underline{i}}}L, \tilde {\ssf}_{{\underline{i}}}^*L$ are
  all irreducible. Furthermore, 
  the operators $\tilde {\sse}_{{\underline{i}}}$ and $\tilde
{\ssf}_{{\underline{i}}}$ define a 
$\fg_c$ combinatorial bicrystal
structure with weight function given by the weight of $L$, and $\varphi_{\underline i}, \varphi_{\underline i}^*,\varepsilon_{\underline i},\varepsilon^*_{\underline i}$ as above. 
  \end{prop}

\begin{proof}
Fix $L \in \KLR$ and let $\nu=\wt(L)$. For each simple root $\beta_{\underline{i}}$ of $\Delta_c$, 
we can choose deformations $c_{\pm}$ of the
charge $c$ such that:
\begin{itemize}
\item For some small $\varepsilon>0$, elements $\mu$ of the weight
  lattice with $\nu-\mu\in \text{span}_{{\mathbb Z}_{\geq 0}} \{
  \alpha_i \}$ have
  \begin{itemize}
  \item $\arg(c_{\pm}(\mu))\in (\pi/2-\varepsilon,\pi/2+\varepsilon)$ if
    and only if $\arg(c(\mu))=\pi/2$,
   \item $\arg(c_{\pm}(\mu)) >\pi/2+\varepsilon$ if and only if
     $\arg(c(\mu))>\pi/2$,
   \item $\arg(c_{\pm}(\mu)) <\pi/2-\varepsilon$ if and only if $\arg(c(\mu))<\pi/2$.
\end{itemize}
\item The root $\beta_{\underline{i}}$ is greater for $>_{c_+}$ and
  lesser for $>_{c_-}$ than
  all other roots $\be\neq \beta_{\underline{i}}$ with $\arg(c(\be))=\pi/2$.
\end{itemize}   
The semi-cuspidal decompositions of $L$ with respect to $c_+$ and $c_-$ must be of the form 
\begin{equation} \label{eq:nk}
A(\scrL_{\beta_{\underline{i}}}^n,\dots) \qquad \text{and} \qquad A(\dots,
\scrL_{\beta_{\underline{i}}}^k),
\end{equation}
respectively, for some $n,k \geq 0$. 
The conditions on $c_\pm$ imply that every representation which
appears in these must be itself in $\KLR[c]$.

  Any quotient of $\scrL_{\beta_{\underline{i}}}\circ L$ is also
a quotient of $\scrL_{\beta_{\underline{i}}}^{n+1}\circ \cdots$. By Proposition \ref{prop:tpi}, $\scrL_{\beta_{\underline{i}}}^{n+1}$ is irreducible, so, by Theorem \ref{th:semi-cuspidal}, $\scrL_{\beta_{\underline{i}}}^{n+1}\circ \cdots$ has a unique simple quotient. Thus $\tilde
{\sse}_{{\underline{i}}}^*L = A(\scrL_{\beta_{\underline{i}}}^{n+1},\dots) $ and $\tilde
{\sse}_{{\underline{i}}}L = A(\dots, \scrL_{ \beta_{\underline{i}}}^{k+1}) $ are irreducible. The irreducibility of $\tilde {\ssf}_{{\underline{i}}} L$ and $\tilde
{\ssf}^*_{{\underline{i}}} L$ follow from a dual argument: using the
fact that any map from a simple into another module lands in its socle, and 
Frobenius reciprocity, 
\[\Hom_{R({\nu}-\be_{\underline{i}})}(L', \tilde {\ssf}_{{\underline{i}}}
L)\cong \Hom_{R({\nu}-\be_{\underline{i}})}(L',L \triangleright \scrL_{\beta_{\underline{i}}})\cong \Hom_{R({\nu})}(L'\circ
\scrL_{\beta_{\underline{i}}}, L).\]  The latter space of maps is 1
dimensional if $L=\tilde
{\sse}_{{\underline{i}}} L'$, and 0 otherwise, so $L'$ has multiplicity
one in $\tilde {\ssf}_{{\underline{i}}} L$ if $L=\tilde
{\sse}_{{\underline{i}}} L'$, and multiplicity 0 otherwise.  Thus, $\tilde {\ssf}_{{\underline{i}}} L=A(\dots,
\scrL_{\beta_{\underline{i}}}^{k-1})$, and in particular is
irreducible.  

Thus, we do have well defined operations of $\KLR[c]$. It remains to check that these
satisfy the conditions in the definition of combinatorial crystal (Definition \ref{def:crystal}), both for the unstarred and starred operators. 
Condition (i) is tautological from our definition of $\varepsilon_{\underline i}$ and $\varepsilon_{\underline i}^*$, and (iv) is
vacuous in this case.
For (ii) and (iii), it suffices to show that, for all $L \in \KLR[c]$, 
 \begin{enumerate}
  \item[(a)] $L\cong\tilde {\ssf}_{{\underline{i}}} \tilde  {\sse}_{{\underline{i}}}L\cong \tilde {\ssf}_{{\underline{i}}}^*\tilde  {\sse}_{{\underline{i}}}^*L$, and
 \item[(b)]  $\varphi_{{\underline i}}(L)= \max \{ n : \tssf^n_{{\underline i}} L \neq 0 \}$, and $\varphi_{{\underline i}}^*(L)= \max \{ n : (\tssf_{{\underline i}}^*)^n L \neq 0 \}$ .
  \end{enumerate}
Condition (a) has been established above, and these arguments also show that $\varphi_{{\underline i}}^*(L)=n$ and $\varphi_{{\underline i}}(L)=k$ for $n,k$ as in \eqref{eq:nk}, from which (b) follows. 
\end{proof}

\begin{lemma}\label{lem:saito-face-crystal} 
Fix a charge $c$. If $\alpha_i$ is a simple root such that $\arg c(\alpha_i) > \pi/2$ then 
Saito reflection $\sigma_i$ induces a bicrystal isomorphism between
$\KLR[c]$ and $\KLR[c^{s_i}]$. Similarly, if $\alpha_i$ is a simple root such that $\arg c(\alpha_i) <\pi/2$, then $\sigma_i^*$ induces a bi-crystal isomorphism between $\KLR[c]$ and $\QH[c^{s_i}]$.
\end{lemma}

\begin{proof}
By Proposition \ref{prop:SaiA} these maps are bijections. It remains to show that they respect the crystal structure. 
We consider $\sigma_i$, the statement about $\sigma_i^*$ following by a symmetric argument. 

Choose a simple root $\beta_{\underline j}$ in $\Delta_c$. As in the proof of Proposition \ref{e-f-defined},
refine $>_c$ into two convex orders $\succ_{\pm}$ such that, among roots with argument $\pi/2$,  $\beta_{\underline i}$ is minimal for $\succ_+$ and maximal for $\succ_-$. 

For $L\in \KLR[c]$, the semi-cuspidal decompositions with respect to $\succ_\pm$ have the form
\[L=A(\scrL_{\beta_{\underline{i}}}^n,\dots)=A(\dots,
\scrL^k_{\beta_{\underline{i}}})\]
for some $n,k \geq 0$, 
where we recall from Proposition \ref{prop:tpi} that $\scrL_{\beta_{\underline{i}}}^n$ is irreducible. 
 Then 
 \[\tsse_{\underline{i}}^*L=A(\scrL_{\beta_{\underline{i}}}^{n+1},\dots) \qquad \text{and} \qquad 
\tsse_{\underline{i}}L=A(\dots,
\scrL_{\beta_{\underline{i}}}^{k+1}).\]
By Proposition \ref{prop:SaiA} these
operations commute with Saito reflection as required.
\end{proof}

\begin{lemma} \label{lem:ssd}
Fix a charge $c$ and let $\beta_{\underline{j}}$ be a simple root for $\Delta_c$. There is a sequence of Saito reflections and dual Saito reflections, 
$\sigma_{i_k}^{x_k} \cdots \sigma_{i_1}^{x_1}$, where each $x_k$ is either $*$ or nothing (indicating dual Saito reflection or Saito reflection) such that
\begin{itemize}
\item $s_{i_k} \cdots s_{i_1} \beta_{\underline{j}}$ is a simple root
  $\alpha_i$, and 
\item at each stage $\sigma_{i_j}^{x_j}$ is a crystal
  isomorphism from $\KLR[c^{s_{i_1} \cdots s_{i_{j-1}}} ]$ to
  $\KLR[c^{s_{i_1} \cdots s_{i_{j-1}} s_{i_j}}]$.
\end{itemize}
\end{lemma}

\begin{proof}
We proceed by induction on the  height $\langle \beta_{\underline{j}}, \rho^\vee \rangle$ of $\beta_{\underline{j}}$ in $\Delta_+$, the height $1$ case being trivial. If $\langle \beta_{\underline{j}}, \rho^\vee \rangle>1$ then for some $i$ we must have $q= \langle \beta_{\underline{j}}, \alpha_i^\vee \rangle >0$. But then $\alpha_i$ and $\beta_{\underline{j}}-q \alpha_i$ are both positive roots, so, since $\beta_{\underline{j}} $ is simple for $\Delta_c$, they cannot both be in $\Delta_c$. It follows that $\arg c(\alpha_i) \neq \pi/2$, so by Lemma \ref{lem:saito-face-crystal} we can apply a Saito reflection $\sigma_i$ or $\sigma_i^*$ and get a face crystal isomorphism from $\KLR[c]$ to $\KLR[c^{s_i}]$. The new simple root corresponding to $\beta_{\underline{j}}$ under this reflection is $\beta_{\underline{j}}-q \alpha_i$, which has lower height. 
\end{proof}

\begin{lemma}\label{lem:e-commute}
  The operators $\tsse_{\underline i}$ and $\tsse_{\underline j}^*$ (and thus $\tssf_{\underline i}$ and
  $\tssf_{\underline j}^*$) for ${\underline i} \neq {\underline j}$ commute. That is, condition (\ref{ccc1})
  of  Proposition
\ref{cor:comb-characterizaton2} holds.
\end{lemma}
\begin{proof}
  Since $\be_{\underline j}$ and $\be_{\underline i}$
  are simple among the roots with $c$-argument $\pi/2$, there is a
  deformation $c'$ of $c$ such that $\be_{\underline i}$ is lowest
  among the roots with $arg(c(\be))= \pi/2$ and $\be_{\underline j}$ is
  greatest.
  Let $(\scrL^n_{\be_{\underline j}}, L_2, \dots, L_{h-1}, \scrL^k_{\be_{\underline i}})$ be the semi-cuspidal decomposition of $L$
  with respect to $c'$. Then $\tsse_{{\underline{i}}}\tsse_{\underline{j}}^*L=\tsse_{\underline{j}}^*\tsse_{{\underline{i}}}L=A(\scrL^{n+1}_{ \be_{\underline j}}, L_2, \dots, L_{h-1}, \scrL^{k+1}_{ \be_{\underline i}}).$
\end{proof}

\begin{lemma} \label{lem:vfc}
For each $\beta_{\underline i}$, the operators $\tsse_{\underline i}$ and $\tsse^*_{\underline i}$ satisfy the condition that, for all $L$,
$$\varphi_{\underline i}^*(\tsse_{\underline i} L) \geq \varphi_{\underline i}^*(L) \;\; \text{ and } \;\; \varphi_{\underline i}(\tsse^*_{\underline i} L) \geq \varphi_{\underline i}(L).$$
\end{lemma}

\begin{proof}
By Lemma \ref{lem:ssd} we can apply Saito reflections $\sigma_i$ and $\sigma_i^*$ a number of
times to reduce to the case when $\beta_{\underline{i}}$ is a simple
root. The condition is then immediate from Proposition \ref{cor:comb-characterizaton2} parts \eqref{ccc3} and \eqref{ccc4}
since $\KLR$ along with the full crystal operators $\tilde{e}_i, \tilde{f}_i$ is a copy of $B(-\infty)$. 
\end{proof}

\subsection{Affine face crystals}
\label{sec:affine-face-crystals}

Some aspects of face crystals are considerably simpler
in the affine case than the general; in other cases results may hold more generally, but we will stay in the affine setting to simplify notation and proofs. Thus in Sections
\ref{sec:affine-face-crystals}-\ref{sec:an-example}, unless otherwise stated, we assume
$\fg$ is affine with minimal imaginary root $\delta$.

Fix a charge $c$. 
If $\arg(c(\delta)) \neq \pi/2$, it is clear that $\mathfrak{g}_c$ is of finite type (although it may be reducible). If $\arg(c(\delta)) = \pi/2$, then $\mathfrak{g}_c$ is either affine or a product of affine algebras. To see this, note that for growth reasons $\mathfrak{g}_c$ cannot be worse than a product of affine algebras. Furthermore, for each root $\alpha$ of $\mathfrak{g}_c$, $n \delta-\alpha$ is also a root for some $n$, so $\alpha$ cannot be part of a finite type sub-root system. In the first case we say that $\mathfrak{g}_c$ is of finite type, and in the second case we say that $\mathfrak{g}_c$ is of affine type. 

\begin{lemma} \label{lem:face-comp}
Assume $L^h \in \KLR[c]$ is lowest weight for the $\fg_c$ bicrystal structure, and that 
$\wt(L^h)=n\delta$ for $n\in \Z_{\geq 0}$. Then the component of $\KLR[c]$ generated by $L^h$ under the crystal operators $\tsse_{{\underline j}}$ is the same as the component generated by $L^h$ under the $\tsse_{{\underline j}}^*$.
\end{lemma}

\begin{proof}
We proceed by induction on the sum
$d(L)$ of the coefficients of the expression for $\wt(L)-\wt(L^h)$ in
terms of the $\beta_{\underline k}$, which we call the depth of $L$. We will show that if 
$L= \tsse_{\underline{j_{d}}}\tsse_{\underline{j_{d-1}}} \cdots \tsse_{\underline{j_1}} L^h$
then $L$ is also in the starred component of $L^h$. 
The reversed statement follows via a symmetric proof. 

If $d=1$ then $L= \tsse_{\underline j} L$ for some $\underline j$. By Lemma \ref{lem:ssd} we can use a sequence of Saito reflections and dual Saito reflections to reduce to the case when $\beta_{\underline{j}}$ is a simple root $\alpha_i$. Then $\tilde e_i L^h = \tilde e_i^* L^h$ by Proposition
\ref{cor:comb-characterizaton2} and the fact that the whole crystal is $B^\fg(-\infty)$, so the claim holds.

Now assume that the component generated by $L^h$ under the ordinary crystal operators agrees with that generated by the $*$ operators at all depths $<d$, and fix $L$ with $d(L)=d$ in the unstarred component of $L^h$.
By the $d=1$ case 
$$L= \tsse_{\underline{j_{d}}}\tsse_{\underline{j_{d-1}}} \cdots \tsse_{\underline{j_2}} \tsse_{\underline{j_1}}^* L^h$$
for some $\underline{j_d}, \underline{j_{d-1}}, \ldots, \underline{j_2}, \underline{j_1}$.
By Lemma \ref{lem:vfc}, we see that
$\tssf_{\underline {j_1}}^* L \neq 0$. 

If $\underline {j_1} \neq \underline {j_{d}}$, then, by Lemma \ref{lem:e-commute},
$L =\tsse^*_{\underline {j_1}}\tsse_{\underline {j_{d}}}\tssf^*_{\underline{j_1}}\tssf_{\underline {j_{d}}}L$. The
module $  \tssf_{\underline {j_{d}}}L$ is manifestly in the component of the
unstarred component of $L^h$, and thus by induction in the starred
component as well. Using the inductive hypothesis again,  $\tsse_{\underline {j_{d}}} \tssf^*_{\underline {j_1}}\tssf_{\underline {j_{d}}}L$
is still in the starred
components of $L^h$, and so $L$ is as well. 

If $\underline {j_1} = \underline {j_{d}}$, then again using Lemma \ref{lem:ssd}, we can reduce to the case when $\be_{\underline
  i}$ to a simple root $\alpha_i$. It follows from Proposition
\ref{cor:comb-characterizaton2} (see also Corollary \ref{cor:KS-diag}) and the fact that the whole crystal is
$B^\fg(-\infty)$ that we have one of the following two situations:

(1) 
$L =\tsse^*_{\underline {j_1}}\tsse_{\underline {j_{1}}}\tssf^*_{\underline{j_1}}\tssf_{\underline {j_{1}}}L$. Then the same argument as in the case $\underline {j_1} \neq \underline {j_{d}}$ shows that $L$ is in the starred
component of $L^h$. 

(2) $\tssf_{\underline{j_1}}^*L=\tssf_{\underline{j_1}}L$.  In this case, by induction,
$\tssf_{\underline{j_1}}^*L=\tssf_{\underline{j_1}}L$ is in both the starred and unstarred
component of $L^h$.  So $L= \tsse_{\underline{j_{d}}}^*
\tssf_{\underline{j_{1}}}^*L$ is also in the starred component.
\end{proof}

\begin{prop} \label{prop:face-comp}
Assume $L^h \in \KLR[c]$ is lowest weight for the bicrystal
structure, and $\wt(L^h)=n\delta$  for $n\in \Z_{\geq 0}$. Then the component generated by $L^h$ under all
$\tsse_{\underline j}, \tsse_{\underline j}^*$ is  isomorphic (as a bicrystal) to the infinity crystal $B^{\fg_c}(-\infty)$. 
\end{prop}
\begin{proof}
By Proposition \ref{e-f-defined} and Lemma \ref{lem:face-comp} the component containing $L^h$ is a lowest weight combinatorial bicrystal. Hence it suffices to check the conditions
of Proposition
\ref{cor:comb-characterizaton2}. Condition \eqref{ccc0} is trivial and \eqref{ccc1} is checked in
Lemma \ref{lem:e-commute} above.
Each of  \eqref{ccc2}--\eqref{ccc5} only involves a
single $\be_{\underline i}$.  
By Lemma \ref{lem:ssd} we can find a sequence of Saito reflections which takes $\beta_{\underline{i}}$ to a simple root, and such that at each step we have an isomorphism of face crystals. This reduces to the case when $\beta_{\underline i}$ is simple for $\fg$, and then
the conditions follow from the isomorphism of $\QH$ with
$B(-\infty)$ for all of $\fg$.
\end{proof}

\begin{cor} \label{cor:aff-lowest}
If $\fg_c$ is of finite type, then $\KLR[c]\cong B^{\fg_c}(-\infty)$. 
If $\fg_c$ is of affine type, then $\KLR[c]$ 
is
isomorphic as a bicrystal to a direct sum of copies of $B^{\fg_c}(-\infty)$, all lowest weight elements $L^h$ have $\wt(L^h) = k \delta$ for some $k$, and the number of lowest weight elements of weight $k \delta$ is the number of $q$-multipartitions of $k$, where $q= r-s=\rk \fg-\rk \fg_c$. 
\end{cor}

\begin{proof}
By Proposition \ref{prop:face-comp}, the trivial representation generates a copy of $B^{\fg_c}(-\infty)$ as a bicrystal.   
  By Corollary \ref{cor:semi-cuspidal-count} the generating function for the number of
  $c$-cuspidal representations of argument $\pi/2$ in $\KLR[c]$ is 
  \[a(t)=\prod_{\al\in
    \Delta_c}\frac{1}{(1-t^{\al})^{\dim\fg_\al}}.\] 
    Comparing with
  the Kostant partition function \[b(t)=\prod_{\al\in
    \Delta_c}\frac{1}{(1-t^{\al})^{\dim(\fg_c)_\al}}\] for $\fg_c$, we
  see that, if $\fg_c$ is finite type, these functions agree. Hence the element of weight $0$ must generate everything and we are done.
  
  If $\fg_c$ is of affine type, then \[\frac {b(t)}{a(t)}= \prod_{k\geq
    1}\frac{1}{(1-t^{k\delta})^{q}}
    \] is the
  generating function of the number of $q$-multipartitions, where $t^\delta$ counts the total number of boxes.

We now proceed by induction.
Fix some $k \geq 0$, and make the assumption
\begin{enumerate}
\item[(A)] All lowest weight elements for $\KLR[c]$ for the unstarred crystal structure of weight at most $k \delta$ have weight $j \delta$ for some $j \leq k$. All of these are also lowest weight for the starred crystal structure as well, and hence by Proposition \ref{prop:face-comp} generate a copy of $B(-\infty)$, and the number of such highest weight elements for each $j \leq k$ is the number of $q$-multipartitions of $j$. 
\end{enumerate}
 Comparing generating functions, the copies of
  $B^{\fg_c}(-\infty)$ generated by lowest weight elements of weight at most $k \delta$ exhaust all elements on $\KLR[c]$  of depth less
  than $(k+1) \delta$, and miss exactly the number of
  $q$-multipartitions of $k+1$ in that depth. This holds true for both the unstarred and the starred crystal structures, and since each lowest weight element generates the same set under both crystal structures, the elements missed for both must coincide. Thus each of the lowest weight elements fund at weight $(k+1) \delta$ are in fact lowest weight for both crystal structures, and the induction proceeds.
\end{proof}

\begin{lemma} \label{lem:ecd}
In finite or affine type, for any charge $c$, the lowest weight elements of $\KLR[c]$ are exactly those which are $c'$-semi-cuspidal for all $c'$ in a neighborhood of $c$.
\end{lemma}

\begin{proof}
Fix $L \in \KLR[c]$.
Define $\Delta_+^{\text{res}} \subset \Delta_+^{\text{min}}$ to be those minimal roots of weight at most $\wt(L)$.

If $L$ is not lowest weight in $\KLR[c]$, then choose a simple root $ \beta_{\underline j}$ of $ \Delta_+^{\text{min}}$ such that $\tssf_{\underline j} L \neq 0$. For any deformation $c'$ of $c$ such that $\beta_{\underline{j}}$ is the minimal root in $\Delta_+^{\text{min}}$  and such that the order of any pair of root in $\Delta_+^{\text{res}}$ remains unchanged, it is clear that $L$ is no longer semi-cuspidal. Thus if $L$ remains semi-cuspidal for all $c'$ in a neighborhood of $c$ then $L$ is lowest weight in $\KLR[c]$.
If $\mathfrak{g}_c$ is finite type, then the only lowest weight element in $\KLR[c]$ is $\mathscr{L}_\emptyset$, so this is enough.  

If $\mathfrak{g}_c$ is of affine type then, by Corollary \ref{cor:aff-lowest}, $\wt(L)$ is a multiple of $\delta$. Fix a deformation $c'$ of $c$ which is small enough so as not to change the order of any pair of roots in $\Delta_+^{\text{res}}$. Assume for a contradiction that $L$ is not $c'$ semi-cuspidal, and let  
$L= A(L_1, \ldots L_h)$ be its semi-cuspidal decomposition. Then we must have $L_h <_{c'} \delta$ (since $\wt(L)=\delta$), so $\wt(L_h)$ is a multiple of a real root $\beta$. If $\beta$ is a simple root $\alpha_i$ for the whole root system $\Delta$ then $\alpha_i$ is a simple root in $\Delta_c$ as well, so clearly $L$ was not lowest weight in $\KLR[c]$. Otherwise, we can use Lemma \ref{lem:ssd} to reduce to this case. 
\end{proof}

\begin{prop} \label{prop:induce-with-lowest} 
Fix $M,N \in \KLR[c]$. Assume $M$ is lowest weight for the $\fg_c$
crystal structure, and $N$ is in the component generated by the
trivial representation. Then $M \circ N = N \circ M$, this module is irreducible, and  $N\mapsto M\circ N$ is a
  bicrystal isomorphism between the component of the trivial
  representation and that of $M$. 
\end{prop}

Before proving Proposition \ref{prop:induce-with-lowest}, we need the following weaker statement:

\begin{lemma} \label{lem:ccs}
  With the notation of Proposition \ref{prop:induce-with-lowest},
  $M\circ N$ has a unique simple quotient, and the map $N\mapsto
  A(M,N)$ commutes with the unstarred crystal operators.
\end{lemma}

\begin{proof}
For any list of weights $\nu_1,\dots, \nu_m$, let
$e_{\nu_1,\dots,\nu_m}$ be the idempotent that projects to all
sequences which consist of a chunk of strands summing to $\nu_1$, a
chunk summing to $\nu_2$, etc.

  Choose any infinite list of nodes $\underline{j_1},\underline{j_2},\dots$ in the Dynkin
  diagram of $\fg_c$ in which each node appears infinitely many times. Let $(a_1, a_2, \dots)$ be the string data of $N$, considered as an element of
  $B^{\fg_c}(-\infty)$, so in particular
  $N=\tsse_{\underline{j_{1}}}^{a_1}\tsse_{\underline{j_{2}}}^{a_{2}} \cdots
  \tsse_{\underline{j_\ell}}^{a_\ell} L_\emptyset.$
By Corollary \ref{cor:aff-lowest}, $\wt(M)=k\delta$ for some $k$. 

Set  $e_{\mathbf{a}}=e_{a_\ell \beta_{\underline {j_\ell}},\dots ,a_1\beta_{\underline {j_1}}},$ and $e_{k \delta, \mathbf{a}}=e_{k \delta, a_\ell \beta_{\underline {j_\ell}},\dots ,a_1\beta_{\underline {j_1}}}$. Let
\begin{equation} \label{eq:bie}
\scrL_{\bf a}= \scrL^{a_\ell}_{ \beta_{\underline{j_{\ell}}}} \circ
  \scrL^{a_{\ell-1}}_{\beta_{\underline{j_{\ell-1}}}} \circ \cdots \circ\scrL^{a_1}_{ \beta_{\underline{j_1}}}
  \end{equation}
and let
$L_{\mathbf{a}}$ be the quotient of
$\scrL_{\bf a}$
  by the subalgebra generated by $e_{{\bf a}'} \scrL_{\bf a}$ for all $\mathbf{a}'>\mathbf{a}$ in lexicographic order. By the definition of string data (Definition \ref{def-string}), $N$ is a quotient of $L_{\mathbf{a}}$.

Consider a word in the character of $e_{k\delta,\mathbf{a}} (M\circ L_{\mathbf{a}})$. This must be a shuffle of a word in each factor of $M\circ \scrL^{a_\ell}_{ \beta_{\underline{j_{\ell}}}} \circ
  \scrL^{a_{\ell-1}}_{\beta_{\underline{j_{\ell-1}}}} \circ \cdots \circ\scrL^{a_1}_{ \beta_{\underline{j_1}}}$. Each of the roots $\beta_{\underline{j}}$ is minimal, so
$\scrL_{\beta_{\underline{j}}}$ is necessarily cuspidal, not
just semi-cuspidal. Thus the letters from each factor that land in any
fixed chunk of
weight $a_k \beta_{\underline {j_k}}$ must have total weight
$a'\beta_{\underline{j_k}}$ for $a'\leq a_k$ (otherwise adding up the contributions from all factors gives something with argument less then $\pi/2$).  
Since  $\tilde
f_{\beta_{\underline j}}M =0$ for all $\underline{j}$, no such chunk can come from $M$. Furthermore, any diagram that permutes strands involving two different $\scrL_{\beta_{\underline{\ell_k}}}^{a_k}$ must
factor through the image of an idempotent $e_{k\delta,\mathbf{a}'}$ higher in
lexicographic order, which is then killed when we take the quotient by to get $L_{\bf a}$ (compare with the argument in \cite[3.7]{KLI}).  Thus
\begin{equation} 
e_{k\delta,\mathbf{a}} (M\circ L_{\mathbf{a}})\cong M\boxtimes \scrL^{a_\ell}_{ \beta_{\underline{j_{\ell}}}} \boxtimes
  \scrL^{a_{\ell-1}}_{ \beta_{\underline{j_{\ell-1}}}} \boxtimes \cdots
  \boxtimes \scrL^{a_1}_{ \beta_{\underline{j_1}}}, \qquad \text{and} \label{eq:3}
\end{equation}
\begin{equation}
e_{k\delta,\mathbf{a'}} (M\circ L_{\mathbf{a}})\cong 0, \quad \text{for all} \quad \mathbf{a'}>\mathbf{a} \label{eq:33}
\end{equation}

It now follows that 
  $M \circ L_{\mathbf{a}}$ has a unique simple quotient: any proper submodule
  is killed by $e_{k\delta,\mathbf{a}}$ so the sum of any two proper
  submodules is as well, and thus is still proper. But $M \circ N$ is clearly a quotient of $M \circ L_{\mathbf{a}}$, so it also has a unique simple quotient.

 Using the definition of the crystal operators, \eqref{eq:3} and \eqref{eq:33} imply that the string data of the unique simple quotient of $M \circ L_{\mathbf{a}}$ with respect to $\fg_c$ is ${\bf a}$, and so this module is actually $\tsse_{\underline{j_{1}}}^{a_1}\tsse_{\underline{j_{2}}}^{a_{2}} \cdots
  \tsse_{\underline{j_\ell}}^{a_\ell} M$. Hence the map $N \rightarrow A(M, N)$ commutes with the ordinary crystal operators.
\end{proof}

\begin{rem}
The reader may notice the
  resemblance of the above argument to that we used earlier based on the
  unmixing property; unfortunately, neither $(M, L_{\mathbf{a}})$ nor $(M, \scrL_{\underline{j_{\ell}}}^{a_\ell},
\dots, \scrL_{\underline{j_1}}^{a_1})$ is actually unmixing, so we must
use this more elaborate argument.
\end{rem}

\begin{lemma} \label{lem-sdc}
With the notation of Proposition \ref{prop:induce-with-lowest}, for any sequence $i_1,\dots,i_d$ where $\sum_j {\beta_{\underline{i_j}}}= \wt(N)$,
$$\dim e_{\beta_{\underline{i_d}}, \ldots, \beta_{\underline{i_1}}, k \delta} (M \circ N)= \dim e_{\beta_{\underline{i_d}}, \ldots, \beta_{\underline{i_1}}, k \delta} (N \boxtimes M).$$ 
\end{lemma}

\begin{proof}
If $\Bi$ is a non-trivial word in the character of
$M$, then the weight of any prefix $\Bi_p$ is either
$<_c\delta$ or is a multiple of $\delta$. In particular, given any word in the character of $M \circ N$ with a prefix of weight $\beta_{\underline i}$ for some $i$, all strands in that prefix must come from $N$. Proceeding inductively, any word in $M \circ N$ with a prefix beginning with blocks that step along $\be_{\underline{i_1}}, \dots,
\be_{\underline{i_d}}$ for an arbitrary sequence $i_1,\dots,i_d$ must have the property that all strands in that prefix must come from $N$. 
\end{proof}

\begin{proof}[Proof of Proposition \ref{prop:induce-with-lowest}]
By \cite[2.2]{LV}, the induction $M\circ N$ is
  isomorphic to the coinduction $\operatorname{coind}(N\boxtimes
  M)$, so there is an injection from $N \boxtimes M$ into the socle of $M \circ N$. By Lemma \ref{lem-sdc}, this implies that $e_{\beta_{\underline{i_d}}, \ldots, \beta_{\underline{i_1}}, k \delta} (M \circ N)$ is contained in the socle of $M \circ N$ for any sequence $i_1,\dots,i_d$ where $\sum_j {\beta_{\underline{i_j}}}= \wt(N)$.

Let $L$ be the cosocle of $M \circ N$. By Lemma \ref{lem:ccs}, $L$ is irreducible and in the unstarred crystal component of $M$, so by  Corollary \ref{cor:aff-lowest}, $L$ is also in the starred component of $M$. Equivalently, for some sequence $\underline{i_1},\dots,\underline{i_d}$ with $\sum_j {\beta_{\underline{i_j}}}= \wt(N)$, we have
$e_{\beta_{\underline{i_d}}, \ldots, \beta_{\underline{i_1}},k\delta}  L \neq 0$.  
In particular, the natural map from the socle of $M \circ N$ to the cosocle is non-zero. Since the cosocle $L$ is simple, this implies that $M \circ N$ itself is simple. 

Notice also that the natural map from $N\circ M$
to the socle of $M\circ N$ must be non-zero and thus an
isomorphism. Hence $N\circ M \simeq M \circ N$.

We have already established that $N \rightarrow A(M, N) = M \circ N$ is a crystal
isomorphism for the unstarred operators; the symmetric argument for $N\circ M$ establishes
that it is for the starred operators as well. 
\end{proof}

\subsection{Affine KLR polytopes}
\label{sec:affine-type}

Outside of finite type, the conventional definition of MV polytope
fails, although, as shown in \cite{BKT}, an alternate geometric
definition can be extended to symmetric affine type. We propose to use
the decorated polytopes $\tilde P_L$ as the ``general type MV
polytopes.'' This construction is not completely combinatorial, as the
decoration consists of various representations of KLR algebras. However, in
affine type we can extract purely combinatorial objects.

For the rest of this section fix $\fg$ of affine type with rank $r+1$.  
As usual, label the simple roots of $\fg$ by
$\al_0,\dots, \al_r$ with $\al_0$ being the distinguished vertex as in \cite{Kac90}. 
Let $\fg_{\mathrm{fin}}$ be the finite type Lie algebra for the diagram with the $0$ node removed. Let $\Delta_{\text{fin}}$ be the root system of $\fg_{\mathrm{fin}}$ and $\bar \alpha_i$ be its simple roots. 

Consider the projection $p:\Delta \rightarrow \Delta_{\text{fin}}$ defined by $p(\alpha_i)=\bar \alpha_i$
for $i \neq 0$, $p(\delta)= 0$. In all cases other than $A_{2n}^{(2)}$ the image of this map is exactly the set of finite type roots along with $0$ (this can be seen by checking that $p$ sends the simple affine roots to a set of finite type roots including all the simples, and using the affine Weyl group). For $A_{2n}^{(2)}$, the image also contains $\alpha/2$ for each of the long roots $\alpha$ in the finite type root system.

For each chamber coweight $\gamma = \theta \omega_i^\vee$ in the finite
type root system (i.e. each element in the Weyl group orbit of a fundamental coweight), define a charge $c_\gamma$ by 
$$c_\gamma(\alpha) = \langle \gamma, p(\alpha) \rangle+i\rho^\vee(\alpha).$$ 
Then $\arg(c_\gamma(\delta))=0$ so $c_\gamma$ defines a vertical face of $P_L$.
As in Section \ref{sec:affine-face-crystals}, the face crystal $\KLR[c_\gamma]$ is a crystal for a product of affine algebras. 

Our next goal is to attach a partition to this face, giving a precise definition of the partitions $\pi^\gamma$ from the introduction.
Let  $\Delta_{\mathrm{fin};\gamma}$ be the sub-root-system of $\Delta_{\text{fin}}$ 
on which $\gamma$ vanishes. Fix a set $\Pi = \{ \eta_1,\dots,
\eta_{r-1} \}$ of simple roots for $\Delta_{\mathrm{fin};\gamma}$. There is
 a unique $\eta_r \in \Delta_{\text{fin}}$ such that 
 \begin{itemize}
\item  
 $\{\eta_1,\dots,
\eta_{r-1} , \eta_r \}$ is a set of simple roots for
  $\Delta_{\mathrm{fin}}$, and 
  \item $\langle
\gamma,\eta_r\rangle=1$. 
\end{itemize}
Explicitly, $\eta_r$ is the unique root with $\langle
\gamma,\eta_r\rangle=1$ such that $\eta_r-\eta_i$ is never a root.

Let $c_\Pi$ be a charge such that the roots sent to $\pi/2$ are exactly the linear
combinations of $p^{-1}(\eta_r)$ and $\delta$, and such that, for all $1 \leq i \leq r-1$, the positive roots in  $p^{-1}(\eta_i)$  are $>_{c_\Pi}\delta$. In particular, for any root $\alpha$,
\begin{equation} \label{eq:oimp}
\alpha <_{c_\Pi} \delta\;\;  \text{ implies } \;\; \alpha \leq_{c_\gamma} \delta.
\end{equation}

The root system $\fg_{c_\Pi}$ is rank 2 affine, and thus is of type
$A_1^{(1)}$ or $A_2^{(2)}$.
 The positive cone for $\fg$ defines simple roots for $\fg_{c_\Pi}$, which we denote by $\be_{\underline 1}$ and
$\be_{\underline 0}$, choosing the labeling so that
$\langle \gamma,p(\be_{\underline 1})\rangle<0$ and thus
$\be_{\underline 1}>_{c_\gamma} \be_{\underline 0}$. For $i=0,1$,
define $\ell_i = \frac{|\beta_{\underline{i}}|}{\sqrt{2}}$ (which is
always $1$ or $2$).  Certainly 
$\ell_0\be_{\underline 0}+\ell_1\be_{\underline 1}$ must be an integer multiple of $\delta$. 

\begin{defn}\label{def:d}
  Let $d_{\gamma}$ be the integer such that 
  $\ell_0\be_{\underline 0}+\ell_1\be_{\underline 1}=d_\gamma \delta$.   
\end{defn}

\begin{rem}
These $d_\gamma$ appear, with a slightly different definition, in \cite[(2.2)]{BeckN}. 
\end{rem}

\begin{example}
Let $\mathfrak{g}$ be of type $A_5^{(3)}$. Then the Dynkin diagram is
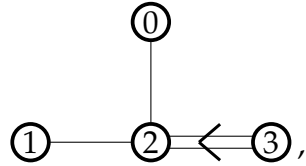
\begin{center}
\begin{tikzpicture}[scale=0.8]

\draw (0,0)--(0,2);
\draw (-2,0)--(0,0);
\draw (0,0.1)--(2,0.1);
\draw (0,-0.1)--(2,-0.1);

\draw[line width = 0.04cm] (0.8,0)--(1.15,0.3);
\draw[line width = 0.04cm] (0.8,0)--(1.15,-0.3);

\draw node [shape=circle, fill=white,draw,line width=0.05cm] at (0,2) {$\;$};
\draw node  at (0,2) {$0$};

\draw node [shape=circle, fill=white,draw,line width=0.05cm] at (-2,0) {$\;$};
\draw node  at (-2,0) {$1$};
\draw node [shape=circle, fill=white,draw,line width=0.05cm] at (0,0) {$\;$};
\draw node  at (0,0) {$2$};
\draw node [shape=circle, fill=white,draw,line width=0.05cm] at (2,0) {$\;$};
\draw node at (2,0) {$3$};

\draw node at (2.5,-0.2) {,};
\end{tikzpicture}
\end{center}
$\delta = \alpha_0+\alpha_1+2\alpha_2+\alpha_3$, and
the underlying finite-type root system is of type $C_3$. 
Consider $\gamma=\omega_3$. Then certainly $\beta_{\underline 0}=\alpha_3$. One might hope that $\beta_{\underline 1}$ was equal to 
$\delta-\alpha_3$, but it turns out that this is not an affine root. Instead, $\beta_{\underline 1}= 2\alpha_0+ 2 \alpha_1 + 4 \alpha_2 +\alpha_3$. 
One can calculate $\ell_0=\ell_1=1$ and $\beta_{\underline 0}+\beta_{\underline 1}= 2 \delta.$ Hence $d_{\omega_3} =2$. 

In this case it is fundamental weights corresponding to long roots that have $d_\gamma \neq 1$, but this is not the general pattern since, as discussed in \cite{BeckN}, $d_\gamma=1$ for all chamber weights in all non-twisted cases. 
\end{example}

\begin{defn} \label{def:isrface}
For each partition $\lambda$, 
  let $\mathscr{L}_{\la;\gamma}$ be the element of the lowest weight $\fg_{c_\Pi}$-crystal generated by the trivial module $\mathscr{L}_{\emptyset}$ which has
  purely imaginary Lusztig
  datum $\la$ for the ordering $\be_{\underline 1}>\be_{\underline 0}$, as defined in
  \cite{BDKT}.   Explicitly, one can easily show using the combinatorics in \cite{BDKT} that
\[\mathscr{L}_{\la;\gamma}=\tsse_{\underline
  1}^{\ell_1\la_1}(\tsse_{\underline
  0}^*)^{\ell_0\la_1}(\tsse_{\underline
  1}^*)^{\ell_1\la_2}\tsse_{\underline 0}^{\ell_0\la_2}\tsse_{\underline
  1}^{\ell_1\la_3}(\tsse_{\underline 0}^*)^{\ell_0\la_3}\cdots 
\mathscr{L}_{\emptyset}.
\]
using the operators $\tsse_{\underline
  j}$ defined in Definition \ref{def:face-operator}.
\end{defn}
Note that the weight of the module $\mathscr{L}_{\la;\gamma}$ is 
$d_\gamma|\lambda|\delta$.

\begin{lemma}\label{friendly-Saito}
  The Saito reflection $\Sai_i$ induces a bicrystal isomorphism from  $\KLR[c_\gamma]$ to $\KLR[c_{s_i\gamma}]$ if $\langle \gamma, p(\al_i)
\rangle\leq 0$ and $\Sai_i^*$  induces a bicrystal isomorphism from  $\KLR[c_\gamma]$ to $\KLR[c_{s_i\gamma}]$ if $\langle \gamma, p(\al_i)
\rangle\geq 0$.
\end{lemma}
\begin{proof}
We consider the case where  $\langle \gamma, p(\al_i)
\rangle\leq 0$, the other case being similar.
By Lemma \ref{lem:saito-face-crystal}, $\sigma_i$ induces a bicrystal isomorphism between $\KLR[c_\gamma^{s_i}]$ and $\KLR[c_\gamma]$. Thus it suffices to show that $\KLR[c_\gamma^{s_i}]$ and $\KLR[c_{s_i \gamma}]$ are the same set. But this is clear since for any $\beta$
$$c_{\gamma}^{s_i}(\beta) < \pi/2 \;\Leftrightarrow \; c_\gamma(s_i(\beta))< \pi/2 \;\Leftrightarrow\; c_{s_i \gamma}(\beta) < \pi/2,$$
so the conditions of being cuspidal of argument $\pi/2$ for these two charges is identical (note however that the charges themselves are not identical).
\end{proof}

\begin{lemma}\label{not-cuspidal}
  Fix $M \in \KLR[c_\gamma] \cap \KLR[c_\Pi]$ of weight $n\delta$, and assume $M$ is in the $\fg_{c_\Pi}$-crystal component of $\mathscr{L}_{\emptyset}$. Then $M\cong \mathscr{L}_{\la;\gamma}$
  for some $\la$.
\end{lemma}

\begin{proof} 
By construction there is a unique minimal root for $>_{c_\gamma}$, and this is a simple root $\alpha_i$ for $\Delta$. 
Since $\beta_{\underline{0}} <_{c_\gamma} \delta$, 
there can only be finitely many $\alpha \in \Delta_+^{min}$ with
$\alpha \leq_{c_\gamma} \beta_{\underline{0}}$. If
$\beta_{\underline{0}}\neq \alpha_i$, then, by Lemma
\ref{friendly-Saito}, the Saito reflection $\sigma_i^*$ is a crystal isomorphism from $\KLR[c_\gamma]$ to $\KLR[c_{s_i\gamma}]$ 
and from $\KLR[c_\Pi]$ to $\KLR[c_{s_i\Pi}]$. This reduces the claim to a case where there are fewer simple roots $\leq \beta_{\underline{0}}$. In this way, we reduce to the case when $\beta_{\underline{0}}$ is a simple root $\alpha_i$ for $\Delta_+$. 

Consider a representation $M$ which is $c_\Pi$-cuspidal of weight $n
\delta$. By Theorem \ref{th:unique-aff} (see also Remark \ref{rem-lda}), $M$ is of the form $\mathscr{L}_{\lambda; \gamma}$ if and only if its crystal-theoretic Lusztig data $\coa_{(m+1) \beta_{\underline{0}}+
  m\beta_{\underline{1}}}(M)$ (see Definition
\ref{ctLusztig}) with respect to the order $\beta_{\underline{1}} >
\beta_{\underline{0}}$ is always trivial.
Thus it suffices to prove that
if $M$ is semi-cuspidal and in the component of $\mathscr{L}_\emptyset$ for $\fg_{c_\Pi}$ , and $M$ has non-trivial Lusztig data of the form  $\coa_{(m+1) \beta_{\underline{0}}+ m\beta_{\underline{1}}}(M)$  for some $m \geq 0$, then $M$ is not semi-cuspidal for $c_\gamma$.

We proceed by induction on the smallest integer $m$ such
  that $\coa_{(m+1)\be_{\underline 0}+m\be_{\underline 1}}(M)\neq 0$, proving the statement for all $\gamma$ simultaneously. If $m=0$ the statement is clear, giving the base case of the induction.  

So assume $m>0$, and recall that we have already reduced to the case when $\beta_{\underline 0}= \alpha_i$. 
  Consider $\Sai^*_iM$. By
  Corollary \ref{saito-A} this must be semi-cuspidal for the charge
  $c_\Pi^{s_{\underline 0}}$.  The face-crystal  $\QH[c_\Pi^{s_i}]$ is still rank-2 affine, with simple roots $\be_{\underline 0}$ and $\be_{\underline 1}$, 
  and the Lusztig data of $\Sai^*_i M$ for
  the order $\be_{\underline 1}<\be_{\underline 0}$ are given by
  $\bar{a}_{\al}(\Sai^*_{\underline 0}M)=a_{s_{\underline 0}\al}(M)$
  for $\al\neq \be_{\underline 0}$.  But \[s_{i}((m+1)\be_{\underline 0}+m\be_{\underline
    1})= s_{\underline
    0}((m+1)\be_{\underline 0}+m\be_{\underline
    1})=(m-1)\be_{\underline 0}+m\be_{\underline 1},\] so, since our
  inductive assumption covered all chamber weights, we are assuming that
  $\Sai^*_{\underline 0}M$ is not semi-cuspidal for $c_{s_{\underline{0}} \gamma}$.
  But then applying Corollary \ref{saito-A} again it is clear that
  $M$ is not semi-cuspidal for $c_\gamma$.  This completes the proof.
  \end{proof}

\begin{prop}\label{prop-pi-all}
The modules $\mathscr{L}_{\pi;\gamma}$ are a complete, irredundant
list of lowest-weight semi-cuspidal modules of argument $\pi/2$ for
$c_\gamma$, and this labeling is independent of the choice of base
in $\fg_{c_\gamma}$.
\end{prop}

Before proving Proposition \ref{prop-pi-all} we first prove a weaker fact: 

\begin{lemma}\label{lem-pi-all-omega}
  Proposition \ref{prop-pi-all} holds when $\gamma=\om_i^\vee$ is a
  fundamental coweight, and the base $\Pi= \{\eta_j\}$ is given by the simple
  roots excluding $\al_i$. 
\end{lemma}

\begin{proof}

Fix a lowest-weight
$L \in \KLR[c_{\omega^\vee_i}]$. Assume for a contradiction that $L$ is not $c_\Pi$-semi-cuspidal. Then there must be a $c_\Pi$-cuspidal $Q$
whose weight is a real root $\al <_{c_\Pi} \delta$ such that $L$ is a quotient of $Q'\circ
Q$ for some simple $Q'$. By \eqref{eq:oimp}, we have $\al\leq_{c_{\omega^\vee_i}} \delta$ and, since $L$ is $c_{\omega^\vee_i}$-semi-cuspidal $\al\geq_{c_{\omega^\vee_i}} \delta$, we see $\al=_{c_{\omega^\vee_i}} \delta$, or equivalently $\alpha$ has argument $\pi/2$ for $c_{{\omega^\vee_i}}$. Since $L$ is $c_{\omega^\vee_i}$-semi-cuspidal and lowest weight for 
$\fg_{c_{{\omega^\vee_i}}}$, this implies $Q$ has these properties as well. But by
Corollary \ref{cor:aff-lowest} all such lowest weight semi-cuspidals have weight a multiple of $\delta$, so this is impossible, and so $L$ is in fact 
$c_\Pi$-semi-cuspidal.  

As in Proposition \ref{prop:induce-with-lowest}, there exist canonical
$M,N \in \KLR[c_\Pi]$, with $M$ lowest-weight and $N$ in the component of the
identity for $\fg_{\Pi}$, such that $L=M\circ N=N\circ M$.  Thus both $M$ and $N$ must be semi-cuspidal and lowest-weight for
$\fg_{c_{\omega^\vee_i}}$. In particular, $M$ is killed
\begin{itemize}
\item by $\tf_i$ since it is lowest-weight in $\KLR[c_\Pi]$,
\item by $\tf_0$
  since it is semi-cuspidal for 
$c_{\omega_i^\vee}$ and $\al_0$ is the lowest root
  for this order, and 
\item by all other $\tf_j$'s since it is lowest-weight
  for $\KLR[c_{\omega^\vee_i}]$.  
\end{itemize}
Thus $M$ is lowest weight for the full $\mathfrak{g}$ crystal structure, so $M=\mathscr{L}_\emptyset$, and hence
  $L=N$.

Thus 
$L$ is semi-cuspidal for $c_{\omega^\vee_i}$ and for $c_\Pi$, and is in the component of the trivial module for $\fg_{c_\Pi}$, so it follows by Lemma
\ref{not-cuspidal} that $L= \scrL_{\pi;\gamma}$ for some $\pi$.  By Corollary \ref{cor:aff-lowest}  the number of lowest weight cuspidals of weight $n \delta$ for $\fg_{c_{\omega_i^\vee}}$ is exactly the number of partitions of $n$, so all $\scrL_{\pi;\gamma}$ must occur.  \end{proof}
\begin{proof}[Proof of Proposition \ref{prop-pi-all}]
We reduce all other cases to that covered in Lemma
\ref{lem-pi-all-omega}.  

If 
$\gamma=\omega^\vee_i$ but we have chosen a
different base 
$\Pi'= \{\eta_i \}'$ of $\fg_{\mathrm{fin};\gamma}$, then we can find an element $w=s_{i_1}\cdots s_{i_k}$ of the
Weyl group $W_{\mathrm{fin};\gamma}$ such that $w\eta_i=\eta_i'$.
Applying a sequence of the dual or primal Saito reflections $\sigma_{i_1}^x\cdots
\sigma_{i_k}^x$ where $x$ is taken to be nothing or $*$, depending on
the sign  of $\langle \gamma,p(\al_i)\rangle$, gives a crystal
isomorphism from $\KLR[c_{\Pi}]$ to $\KLR[c_{\Pi'}]$, and thus sends $\mathscr{L}_{\pi;\gamma}$ as defined using $\{\eta_i\}$
to $\mathscr{L}_{\pi;\gamma}$ as defined using the $\{\eta_i'\}$. On
the other hand, these operators
leave $\mathscr{L}_{\pi;\gamma}$ unchanged (since
it is killed by $\tf_{i_m}$ and
$\tf^*_{i_m}$ and has weight a multiple of $\delta$).  Thus
$\mathscr{L}_{\pi;\gamma}$ is independent of this choice.

Now consider a general chamber coweight $\gamma$. If $\gamma$ is not a
fundamental coweight then, for some $1 \leq i \leq r$, we must have $\langle \gamma, p(\al_i)\rangle<0$, and so $\alpha_i >_c \delta$. Notice that
$\al_i\neq \be_{\underline 0}$, since $\langle \gamma,
p(\beta_{\underline{0}})\rangle >0$. Thus 
$\varphi_i^*(\mathscr{L}_{\pi;\gamma})=0$, so we can apply $\sigma_i$.
If $\al_i\neq \be_{\underline 1}$ then by Lemmata
\ref{lem:saito-face-crystal}  and \ref{friendly-Saito} applying $\sigma_i$ to all cuspidal
modules for $c_\gamma$ defines an isomorphism of crystals to the same
set-up for $c_{s_i\gamma}$, which is negative on one fewer positive root
in the finite type system than $\gamma$; in particular it sends
$\scrL_{\pi;\gamma}$ to $\scrL_{\pi;s_i\gamma}$.  If $\al_i=
\be_{\underline 1}$, the same fact follows from the known action of
Saito reflections on $B(-\infty)$ for affine rank 2 Lie algebras
by \cite[3.9]{MT??}.  By
induction, we may reduce to the case where $\gamma$ is a fundamental
coweight, so the result follows by Lemma \ref{lem-pi-all-omega}.
\end{proof}

Fix a generic charge $c$ such that $\delta$ has argument $\pi/2$. 
This defines a positive system in the finite type root system, where
we say $\bar \alpha$ is positive if $p^{-1}(\alpha) >_c
\delta$. Let $\bar \chi_1, \ldots, \bar \chi_r$ be the corresponding
set of simple roots and
$\gamma_1,\dots, \gamma_r$ the dual set of coweights.  
For each 
$r$-tuple of partitions $\bpi=(\pi^{\gamma_1},\dots, \pi^{\gamma_r})$, define
  \begin{equation}\label{eq:fac}\displaystyle
    \mathscr{L}(\bpi)=\mathscr{L}_{\pi^{\gamma_1};\gamma_1}\circ 
    \mathscr{L}_{\pi^{\gamma_2};\gamma_2}\circ \cdots \circ 
    \mathscr{L}_{\pi^{\gamma_r};\gamma_r}.
  \end{equation}

\begin{rem}
The modules $ \mathscr{L}(\bpi)$ agree with Kleshchev's imaginary modules \cite[\S 4.3]{Kl}.  Note that
  in contrast to Kleshchev, we have a {\it canonical} labeling of these by
  multipartitions. 
After the appearance of this paper on the arXiv, Kleshchev-Muth
\cite{KMuth} and McNamara \cite{McNIII} reproduced this indexing using
other methods.  
The match with Kleshchev-Muth's indexing reduces to
the rank 2 case by \cite[5.10]{Kl}, and is clear
in that case since Definition \ref{def:isrface} above shows that
$\mathscr{L}_{\pi;\gamma}$ in this case contains the ``Gelfand-Graev
word'' $\mathbf{g}^\pi$ in its character. By \cite[Th. 9]{KMuth},
Kleshchev-Muth's bijection is uniquely characterized by this
property.  Similarly \cite[14.6]{McNIII} shows that McNamara's
bijection must be the same as ours. 
\end{rem}

\begin{lemma} \label{lem:face-lw}
Fix a charge $c$ such that $\delta$ has argument $\pi/2$, and let $s$ be the rank of $\Delta_c$. 
Then there are chamber weights $\gamma_1, \ldots, \gamma_{r-s}$ for $\Delta_{fin}$ such that the lowest weight elements in the face crystal $\KLR[c]$ are exactly  
$$\mathscr{L}_{\pi^{1};\gamma_1}\circ 
    \mathscr{L}_{\pi^{2};\gamma_2}\circ \cdots \circ 
    \mathscr{L}_{\pi^{{r-s}};\gamma_{r-s}}$$
for all choices of partitions $\pi^1, \ldots, \pi^{r-s}$. 
 This module is irreducible and independent of the
 ordering of $\gamma_1, \ldots, \gamma_{r-s}$.   
\end{lemma}

\begin{proof}
Proposition \ref{prop-pi-all} shows that  this statement holds for
$r-s=1$. We proceed by induction on $r-s$, assuming the holds for $r-s=j-1$ for some $2 \leq j \leq r$.
Thus, we have assumed that $$\mathscr{L}_{\pi^{1};\gamma_1}\circ 
    \mathscr{L}_{\pi^{2};\gamma_2}\circ \cdots \circ 
    \mathscr{L}_{\pi^{{j-1}};\gamma_{j-1}}$$ is irreducible.  
    Choose $c_\Pi=\sum_{k\neq j}c_{\gamma_k}$ in the definition of
    $\mathscr{L}_{\pi^{\gamma_j};\gamma_j}$. Then, since
    $\eta_j$ is parallel to the face $F_{j-1}$ defined by the vanishing of the
    weights $\gamma_1,\dots, \gamma_{j-1}$, the module 
    $\mathscr{L}_{\pi^{\gamma_j};\gamma_j}$ is in the
    component of the identity of the face crystal for $F_{j-1}$, so, by Proposition \ref{prop:induce-with-lowest}, $\mathscr{L}_{\pi^{\gamma_1};\gamma_1}\circ 
    \mathscr{L}_{\pi^{\gamma_2};\gamma_2}\circ \cdots \circ 
    \mathscr{L}_{\pi^{\gamma_{j}};\gamma_{j}}$ is irreducible.
    It is lowest weight in $\KLR[c]$ since it is an irreducible induction of
    lowest weight representations. 
    
    The partition $\pi^{\gamma_i}$ is
uniquely determined by the isomorphism type of the induced module, so by induction the modules $\mathscr{L}_{\pi^{1};\gamma_1}\circ 
    \mathscr{L}_{\pi^{2};\gamma_2}\circ \cdots \circ 
    \mathscr{L}_{\pi^{{r-s}};\gamma_{r-s}}$ are all distinct. By Corollary
    \ref{cor:aff-lowest} this is the right number, so there can be no others. This establishes the result
    for $r-s=j$.  
\end{proof}
If $s=0$ (which holds for generic $c$), then the face crystal is
trivial, and every semi-cuspidal is lowest weight.  Thus:
\begin{cor}\label{cor:commute}
If $s=0$, the modules $\mathscr{L}(\bpi)$ give a
complete and irredundant list of semi-cuspidal modules with argument $\pi/2$.
\qed
\end{cor}

Corollary \ref{cor:commute} tells us the possible decorations for an edge of $P_L$ parallel to $\delta$. The following explains how to read off the decoration for a given $L$. 
So, fix $L \in \KLR$ and a finite type chamber coweight $\gamma$.
Consider the $c_\gamma$-semi-cuspidal decomposition $(\dots, L_2,
L_1,L_0,L^1,L^2,\dots,)$ of $L$, where $\wt(L_0)$ has argument
$\pi/2$.  

\begin{defn} \label{def:pig}
Let
$\pi^\gamma(L)$ be the partition such that $L_0$ lies in the $\mathfrak{g}_{c_\gamma}$-crystal
component of $\mathscr{L}_{\pi^\gamma(L);\gamma}$.  
\end{defn}

\begin{prop}
  The representation decorating any imaginary edge $E$ in $P_L$ as in Definition \ref{def-KLR-poly} is exactly 
  $\mathscr{L}(\pi^{\gamma_1}(L),\dots, \pi^{\gamma_r}(L))$, where the $\gamma_i$
  are the chamber coweights which achieve their lowest value on $E$, and $\pi^{\gamma_i}(L)$ is the
  partition from Definition \ref{def:pig}.
\end{prop}

\begin{proof} 
Let $c$ be a generic charge such that $E$ is part of the path $P_L^c$.
Let $L_0$ be the representation in the
$>_c$-semi-cuspidal decomposition of $L$ whose weight is a multiple
of $\delta$.  Then, by Corollary \ref{cor:commute}, $L_0=  \mathscr{L}(\xi^{\gamma_1},\dots,
\xi^{\gamma_r})$ for some partitions $\xi^{\gamma_i}$. We need to show that, for each $i$, the partition $\pi^{\gamma_i}$ attached to $L$
by Definition \ref{def:pig} is $\xi^{\gamma_i}$.  

The module
$\mathscr{L}(\xi^{\gamma_1},\dots,\xi^{\gamma_{i-1}}, \emptyset,
\xi^{\gamma_{i+1}},\dots, 
\xi^{\gamma_r})$ is in the crystal component of the identity for
$c_{\gamma_i}$, since $\gamma_i$ vanishes on $\eta_j$ for $j\neq i$.
On the other hand, we already know that
$\mathscr{L}_{\xi^{\gamma_i};\gamma_i}$ is lowest weight for the face
crystal of $c_{\gamma_i}$.  Thus, by Proposition \ref{prop:induce-with-lowest}, the induction of these two modules
is irreducible and in the component of $\mathscr{L}_{\xi^{\gamma_i};\gamma_i}$ for
the face crystal of the facet defined by $\gamma_i$.  But then by definition,
thus $\pi^{\gamma_i}=\xi^{\gamma_i}$.
\end{proof}

\begin{defn}\label{def:aMV}
The {\bf affine MV polytope} $P_L$ associated to $L \in \KLR$ is the character polytope along with the data of $\pi^\gamma(L)$ for each chamber co-weight $\gamma$. 
\end{defn}

This is a decorated affine pseudo-Weyl polytope as defined in
the introduction. It encodes the same information as the KLR polytope
$\tilde P_L$ in the sense of Definition \ref{def-KLR-poly}.  To obtain
the KLR polytope from the affine MV polytope, we
decorate each edge parallel to a real root with only possible
semi-cuspidal representation, and each edge $E$ parallel with $\delta$ with the representation
$\scrL(\bpi)$ associated to $\pi^\gamma$ for the chamber co-weights which achieve their minimum on $P_L$ along the edge $E$. 

\begin{defn} \label{def:adw-lus}
The {\bf Lusztig data} of a decorated affine pseudo-Weyl with respect to a convex order $\succ$ is the geometric data of the underlying polytope, along with the information of the partitions $\pi^\gamma$ for fundamental weight $\gamma$ of the positive system in $\Delta_{\text{fin}}$ defined by $\succ$. 
\end{defn}

\subsection{Proof of Theorems \ref{thmB} and \ref{corB}}
\label{sec:proof}

\begin{proof}[Proof of Theorem \ref{thmB}]
By Lemma \ref{lem:skel}, any vertex of a pseudo-Weyl polytope $P$ is in the path $P^{>_c}$
for some generic charge $c$.  Thus, any pseudo-Weyl polytope is the
convex hull of the paths $P^{>_c}$ when $c$ ranges over generic charges.

Fix a generic charge $c$ and consider the convex order $>_c$. We
first claim that there can be at most one decorated polytope
satisfying the conditions of Theorem \ref{thmB} with a given Lusztig
datum with respect to $>_c$. To see this, fix such a decorated
polytope $P$, and consider another generic charge $c'$. Lemma
\ref{lem:gen-braid} shows that the path $P^{>_c}$ can be changed to the path
$P^{>_{c'}}$ by moving across finitely many 2-faces in such a way
that, at each step, the path passes through both the top and bottom
vertex of that 2-face. The conditions of Theorem \ref{thmB} then allow
us to determine $P^{>_{c'}}$ from $P^{>_{c}}$.

  By Theorem \ref{th:semi-cuspidal} and Corollary
  \ref{cor:semi-cuspidal-count} (see also Corollary \ref{cor:commute} for the
  imaginary part), we can find a simple $L$ such that $P_L$ has any
  specified Lusztig datum with respect to $>_c$. To prove
  Theorem \ref{thmB}, it thus suffices to show that each $P_L$ satisfies all
  the specified conditions on $2$-faces.

  Every 2-face is either real or parallel to $\delta$.  The real
  2-faces are themselves MV polytopes by Proposition
  \ref{prop:real-faces}. Thus it remains to check that 2-faces
  parallel to $\delta$ yield affine MV polytopes (after
  shortening the imaginary edge as in the statement).  Fix a charge $c$ such that the
  roots sent to the imaginary line form a rank 2 affine sub-root system, and
  let $\fg_c$ be the associated rank 2 affine algebra. This defines a
  (possibly degenerate) 2-face $F_c(P_L)$ of any $P_L$, and all imaginary
  2-faces occur this way for some $c$.

  Let $\gamma_1, \ldots, \gamma_{r-1}$ be the $r-1$ finite type
  chamber weights which define facets of $P_L$ containing $F_c$ for
  all $L$, and $\gamma_+, \gamma_-$ the two chamber weights that
  define faces that intersect $F_c$ in vertical lines.
   If you deform
  $c$ a small amount, then it gives a complete order on roots, and
  picks out one of the two vertical edges of $F_c$. We can choose
  deformations $c_\pm$ such that the set of chamber weights associated
  with these charges are $\{ \gamma_1, \ldots, \gamma_{r-1},
  \gamma_\pm \}$. 
  
  By Lemma \ref{lem:face-lw}, the $c_\pm$ semi-cuspidal modules are exactly those of the form 
  $$L= \mathscr{L}_{\pi^{1};\gamma_1}\circ 
    \mathscr{L}_{\pi^{2};\gamma_2}\circ \cdots \circ 
    \mathscr{L}_{\pi^{{r-1}};\gamma_{r-1}} \circ \mathscr{L}_{\pi^{{\pm}};\gamma_\pm}$$
    for partitions $\pi^1, \ldots, \pi^{r-1}, \pi^\pm$, and
    furthermore the first $r-1$ factors give the lowest weight element
    in the component of the $c$-face crystal containing $L$. Thus it
    suffices to show that, for any $M$ in $\KLR[c]$ in the component
    $\KLR[c, \mathscr{L}_\emptyset]$ of the face crystal generated by $\mathscr{L}_\emptyset$, the face $F_c(P_M)$ is an MV polytope for $\fg_c$. For this it suffices to show that the map 
 $M \mapsto  F_c(P_M)$  from $\KLR[c, \mathscr{L}_\emptyset]$ to the set of
 $\fg_c$-pseudo-Weyl polytopes
 satisfies the
  conditions of Theorem \ref{th:unique-aff}. 
  \begin{enumerate}
  \item[(i)] This is clear for the trivial element (in which
    case the weight is 0 on both sides), and it is also clear that
    this property is preserved by the $\fg_c$ crystal operators.
  \item[(ii.1-4)] Using Lemmata \ref{cor:polytope-Saito} and \ref{lem:ssd}, we can find
    Saito reflections in $B(-\infty)$ which reduce us to the case
    where $\be_{\underline 0}$ or $\be_{\underline 1}$ is simple for
    $\fg$. Hence 1 and 2 follow from Proposition
    \ref{polytope-crystal}. Parts 3 and 4 are then clear from the form of $*$ involution.
        
  \item[(iii.1-4)] Using  Lemmata \ref{cor:polytope-Saito} and \ref{lem:ssd}, we can again
    reduce to a case where $\be_{\underline 0}$ is simple in $\Delta$. Then Saito reflection in this root for the face crystal $B^{\fg_c}(-\infty)$
    is the restrictions of the corresponding reflection in the full
    crystal $B(-\infty)$. Hence the statements for $\beta_{\underline{0}}$ are a consequence of Corollary \ref{saito-A}. To get
    the statements for the reflections in $\beta_{\underline 1}$ we instead use
    Saito reflections in $B(-\infty)$ to reduce this to a simple root.
  \item[(iv)] By definition (see Definition \ref{def:isrface}),
    $\mathscr{L}_{\la;\gamma}=\tsse_{\bar{1}}^{\ell_1\la_1}(\tsse_{\bar{0}}^*)^{\ell_0\la_1}\bar{\mathscr{L}}_{\la\setminus
      \la_1;\gamma}$, and this is semi-cuspidal for the other convex order on $\Delta_c$, from which (iv) is immediate. \qedhere
  \end{enumerate}
\end{proof}

\begin{proof}[Proof of Theorem \ref{corB}]
  By \cite[Theorems 5.9 and 5.12]{MT??}, if one takes the transpose of
  each partition $\lambda_\gamma$ decorating the Harder-Narasimhan
  polytopes $HN_b$ from \cite[Sections 1.5 and 7.6]{BKT}, then these
  satisfy the conditions in Theorem \ref{thmB} (i.e the same
  conditions satisfied by the KLR polytopes $P_L$). By Lemmas
  \ref{lem:skel} and \ref{lem:gen-braid} a decorated affine
  pseudo-Weyl polytope satisfying the conditions of Theorem \ref{thmB}
  is uniquely determined by its Lusztig data with respect to any one
  charge. The number of elements of $B(-\infty)$ of each weight is
  given by the Kostant partition function, which also counts the
  number of possible Lusztig data. Thus the set of KLR polytopes and
  the set of HN polytopes (with decoration transposed) coincide, and
  this set is indexed by the possible Lusztig data for any fixed
  charge. Since both index $B(-\infty)$, we get a bijection
  $B(-\infty)\to B(-\infty)$. This bijection commutes with the crystal
  operators $\tilde f_i$, since in both cases $\tilde f_i$ acts in a
  simple way on the Lusztig datum for any convex order $\succ$ with
  $\alpha_i$ minimal. Since $B(-\infty)$ is connected, this map is the
  identity.
\end{proof}

\subsection{An example}
\label{sec:an-example}

Fix a generic charge $c$.
If one were trying to naively generalize the notion of Lusztig data in $\KLR$ from the finite type situation, one might hope to
find a totally ordered set of cuspidal simples 
such that the modules $A(L_1^{n_1}, \ldots, L_k^{n_k})$ for $L_1 \geq_c
\cdots \geq_c L_k$ are a
complete list of the simples. We now illustrate how, even in affine type, this will
fail. We note that this example is also
treated in \cite[Example 3.3]{Kashnote} 
for different purposes.

Consider the case of $\asl_2$.
Choose the polynomial $Q_{01}(u,v)$ to be $u^2+quv+v^2$ for some
$q\in \K$ (this is not a completely general choice of $Q$, but any choice of
$Q$ gives an algebra isomorphic this one after passing to a finite
field extension).

Choose a charge where $\alpha_0 <_c \alpha_1$.
There are exactly two semi-cuspidal representations of weight
$2\delta$.  These can be
described as $\mathscr{L}_{(2);\om}=\te_1^2\te_0^2\mathscr{L}_\emptyset$ and  $\mathscr{L}_{(1,1);\om}=\te_1\te_0
\te_1\te_0\mathscr{L}_\emptyset$. Consider the induction
$\mathscr{L}_{(1);\om}\circ \mathscr{L}_{(1);\om}$.  This is
6-dimensional, spanned by the elements 
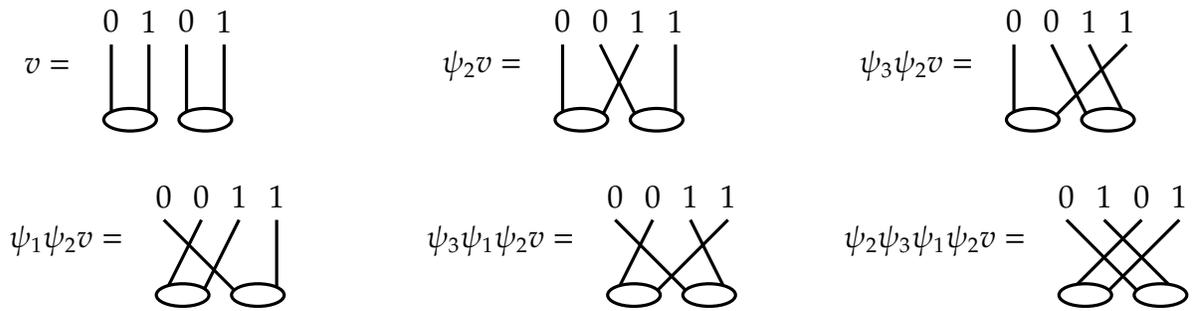
\[ \tikz{  
\node[label=left:{$v=$}] at (-6,0){
\tikz[very thick]{\draw (0,0) -- node [above, at end]{$0$} (0,1); \draw (.5,0) -- node [above, at end]{$1$}  (.5,1); \draw (1,0)
  -- node [above, at end]{$0$}  (1,1); \draw (1.5,0) -- node [above, at end]{$1$}  (1.5,1);
  \draw[fill=white]  (.25,0) ellipse (3.5mm and 1.5mm);    
 \draw[fill=white]  (1.25,0) ellipse (3.5mm and 1.5mm);    
}
};
\node[label=left:{$\psi_2v=$}] at (0,0){
\tikz[very thick]{\draw (0,0) -- node [above, at end]{$0$} (0,1); \draw (.5,0) -- node [above, at end]{$1$}  (1,1); \draw (1,0)
  -- node [above, at end]{$0$}  (.5,1); \draw (1.5,0) -- node [above, at end]{$1$}  (1.5,1);
  \draw[fill=white]  (.25,0) ellipse (3.5mm and 1.5mm);    
 \draw[fill=white]  (1.25,0) ellipse (3.5mm and 1.5mm);    
}
};
\node[label=left:{$\psi_3\psi_2v=$}] at (6,0){
\tikz[very thick]{\draw (0,0) -- node [above, at end]{$0$} (0,1); \draw (.5,0) -- node [above, at end]{$1$}  (1.5,1); \draw (1,0)
  -- node [above, at end]{$0$}  (.5,1); \draw (1.5,0) -- node [above, at end]{$1$}  (1,1);
  \draw[fill=white]  (.25,0) ellipse (3.5mm and 1.5mm);    
 \draw[fill=white]  (1.25,0) ellipse (3.5mm and 1.5mm);    
}
};
}
\]
\[ \tikz{  
\node[label=left:{$\psi_1\psi_2v=$}] at (-6,0){
\tikz[very thick]{\draw (0,0) -- node [above, at end]{$0$} (.5,1); \draw (.5,0) -- node [above, at end]{$1$}  (1,1); \draw (1,0)
  -- node [above, at end]{$0$}  (0,1); \draw (1.5,0) -- node [above, at end]{$1$}  (1.5,1);
  \draw[fill=white]  (.25,0) ellipse (3.5mm and 1.5mm);    
 \draw[fill=white]  (1.25,0) ellipse (3.5mm and 1.5mm);    
}
};
\node[label=left:{$\psi_3\psi_1\psi_2v=$}] at (0,0){
\tikz[very thick]{\draw (0,0) -- node [above, at end]{$0$} (.5,1); \draw (.5,0) -- node [above, at end]{$1$}  (1.5,1); \draw (1,0)
  -- node [above, at end]{$0$}  (0,1); \draw (1.5,0) -- node [above, at end]{$1$}  (1,1);
  \draw[fill=white]  (.25,0) ellipse (3.5mm and 1.5mm);    
 \draw[fill=white]  (1.25,0) ellipse (3.5mm and 1.5mm);    
}
};
\node[label=left:{$\psi_2\psi_3\psi_1\psi_2v=$}] at (6,0){
\tikz[very thick]{\draw (0,0) -- node [above, at end]{$0$} (1,1); \draw (.5,0) -- node [above, at end]{$1$}  (1.5,1); \draw (1,0)
  -- node [above, at end]{$0$}  (0,1); \draw (1.5,0) -- node [above, at end]{$1$}  (.5,1);
  \draw[fill=white]  (.25,0) ellipse (3.5mm and 1.5mm);    
 \draw[fill=white]  (1.25,0) ellipse (3.5mm and 1.5mm);    
}
};
}
\]
where $v$ is any non-zero element of $\mathscr{L}_{(1);\om}\boxtimes
\mathscr{L}_{(1);\om}$, which is 1-dimensional.

The span $H$ of the basis vectors other than $v$ is a submodule (it is the
kernel of a map to $\mathscr{L}_{(1,1);\om}$). 
The image of the idempotent $e_{0011}$ is irreducible over
$R(2\al_0)\otimes R(2\al_1)$, and generates $H$.  Thus, either
\begin{itemize}
\item $H$ is irreducible or 
\item $\psi_2\psi_3\psi_1\psi_2v$ spans a submodule.
\end{itemize}
But, 
\[\tikz[xscale=1.1]{ \node[label=left:{$\psi_2^2\psi_3\psi_1\psi_2v=$}] at (-6,0){ \tikz[very
    thick]{\draw (0,0) to[out=45,in=-90] (1,1) to[out=90,in=-45]  node [above, at end]{$0$} (.5,1.5); \draw (.5,0) to[out=45,in=-90] (1.5,1) to[out=90,in=-90] node [above, at end]{$1$}  (1.5,1.5); \draw (1,0)
  to[out=135,in=-90] (0,1) to[out=90,in=-90] node [above, at end]{$0$}  (0,1.5); \draw (1.5,0) to[out=135,in=-90] (.5,1) to[out=90,in=-135] node [above, at end]{$1$} (1,1.5);
  \draw[fill=white]  (.25,0) ellipse (3.5mm and 1.5mm);    
 \draw[fill=white]  (1.25,0) ellipse (3.5mm and 1.5mm);    
}
};
\node[label=left:{$=q$}] at (-3,0){
\tikz[very thick]{\draw (0,0)  to[out=30,in=-90] node [above, at end]{$0$}
  node[pos=.8,fill=black,circle,inner sep=2pt]{} (.5,1); \draw (.5,0) -- node [above, at end]{$1$}  (1.5,1); \draw (1,0)
  -- node [above, at end]{$0$}  (0,1); \draw (1.5,0)  to[out=150,in=-90] node[pos=.8,fill=black,circle,inner sep=2pt]{}  node [above, at end]{$1$}  (1,1);
  \draw[fill=white]  (.25,0) ellipse (3.5mm and 1.5mm);    
 \draw[fill=white]  (1.25,0) ellipse (3.5mm and 1.5mm);    
}
};
\node[label=left:{$=-q$}] at (-0,0){
\tikz[very thick]{\draw (0,0) -- node [above, at end]{$0$} (0,1); \draw (.5,0) -- node [above, at end]{$1$}  (1,1); \draw (1,0)
  -- node [above, at end]{$0$}  (.5,1); \draw (1.5,0) -- node [above, at end]{$1$}  (1.5,1);
  \draw[fill=white]  (.25,0) ellipse (3.5mm and 1.5mm);    
 \draw[fill=white]  (1.25,0) ellipse (3.5mm and 1.5mm);    
}
};
}\]
 Thus, if
$q\neq 0$, $H$ is irreducible and thus 
$H\cong\mathscr{L}_{(2);\om}$.  Its inclusion is
split, with complement spanned by $qv+ \psi_2\psi_3\psi_1\psi_2v$.  In particular,
$\mathscr{L}_{(1);\om}\circ \mathscr{L}_{(1);\om}$ is semi-simple with
both $\mathscr{L}_{(2);\om}$ and $\mathscr{L}_{(1,1);\om}$ occurring as
summands.  We see that neither of these modules can thus be cuspidal, since
 \[\ch(\mathscr{L}_{(2);\om})=4\cdot w[0011]+w[0101].\]
If $q=0$, then the behavior is quite different; in this case $
\psi_2\psi_3\psi_1\psi_2v$ spans the socle of $\mathscr{L}_{(1);\om}\circ
\mathscr{L}_{(1);\om}$, and $H$ is its radical.  In particular, $\mathscr{L}_{(1);\om}\circ
\mathscr{L}_{(1);\om}$ is indecomposable,
and a 3-step extension where a copy of   $\mathscr{L}_{(2);\om}$ is
sandwiched between the socle and cosocle, both isomorphic to
$\mathscr{L}_{(1,1);\om}$. So in particular, when $q=0$, the representation
$\mathscr{L}_{(2);\om}$ {\it is} cuspidal,
since \[\ch(\mathscr{L}_{(2);\om})=4\cdot w[0011].\]

The KLR polytopes of these representations are independent of $q$ and
are given by 
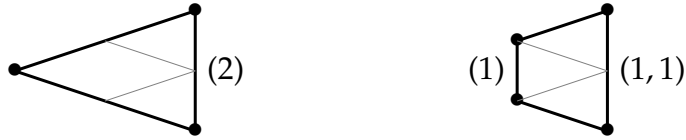
\[
\begin{tikzpicture}
  \node at (-3,0){
\tikz[yscale=0.4, xscale=1.2] {
\node at (0,0) {$\bullet$}; 
\node at (0,4) {$\bullet$};
\node at (-2,2) {$\bullet$};
\draw [line width = 0.04cm] (0,0) -- (-2,2);
\draw [line width = 0.04cm] (0,4) -- (-2,2);
\draw [line width = 0.04cm] (0,0) -- node[right,midway]{$(2)$}(0,4);
\draw [line width = 0.01cm, color=gray] (-1,1) -- (0, 2);
\draw [line width = 0.01cm, color=gray] (-1,3) -- (0, 2);
}
};
 \node at (3,0){
\tikz[yscale=0.4, xscale=1.2] {
\node at (0,0) {$\bullet$}; 
\node at (0,4) {$\bullet$};
\node at (-1,1) {$\bullet$};
\node at (-1,3) {$\bullet$};
\draw [line width = 0.04cm] (0,0) -- (-1,1);
\draw [line width = 0.04cm] (-1,3) -- node[midway,left] {$(1)$}(-1,1);
\draw [line width = 0.04cm] (0,4) -- (-1,3);
\draw [line width = 0.04cm] (0,0) -- node[right,midway]{$(1,1)$}(0,4);
\draw [line width = 0.01cm, color=gray] (-1,1) -- (0, 2);
\draw [line width = 0.01cm, color=gray] (-1,3) -- (0, 2);
}
};
\end{tikzpicture}
\]

If one takes the choice of parameters as in \cite{VV} corresponding to
an Ext-algebra of perverse sheaves on the moduli of representations of
a Kronecker quiver (which is also that fixed by \cite{BKKL} in order
to find a relationship to affine Hecke algebras with $\nu=-1$ or in
characteristic $2$), then we take $q=-2$.  Thus, if the field $\K$ has
characteristic $\neq 2$, we have $q\neq 0$ and $\dim
\mathscr{L}_{(2);\om}=5$ whereas if  $\K$ does have characteristic 2, then
$q=0$ and $\dim \mathscr{L}_{(2);\om}=4$.  Under Brundan and
Kleshchev's isomorphism \cite{BKKL} between quotients of KLR algebras and
cyclotomic Hecke algebras, this corresponds to the change in
characters as we pass from the Hecke algebra at a root of unity to the
symmetric group, or the difference between the canonical basis and
$2$-canonical basis.

In the $q=0$ case, the number of cuspidals in this example is in fact the root multiplicity of $2 \delta$. One might naively hope that at $q=0$ this holds more generally, but explicit calculations in more complicated examples show that it does not.

\subsection{Beyond affine type} \label{ss:beyond-affine}

In affine type, while we can have many different semi-cuspidal
representations corresponding to an imaginary root, we still have
considerable control over the structure of these representations. In particular, all the required labels for these MV polytopes can be encoded with the data of a partition associated to each facet parallel to $\delta$.

In general, we expect that the structure of a 2-face should be
controlled by the set of roots obtained by intersecting a
2-dimensional plane with $\Delta$. If $\fg$ is of finite type then
this set is also a finite type root system and the 2-faces are finite
type MV polytopes. In affine type, this intersection can also be a rank-2
affine root system, and 2-faces are essentially rank 2 affine MV
polytopes. But because of the multiplicities, the sum of these root
spaces is actually not quite an affine root system---rather, it is the root system of an
infinite-rank Borcherds algebra whose Cartan matrix is obtained by
adding infinitely many rows and columns of zeroes to the rank 2 affine
matrix.  The structure we have observed in the 2-faces (many copies of the same crystal $B(-\infty)$ in the case when the intersection is affine) seems to be a manifestation of this larger
algebra.

Beyond affine type, when one intersects $\Delta$ with a 2-plane, the
resulting set of real roots
will generate a root system of rank at most 2. However, if there is to be a
generalization of Theorem \ref{thmB}, considering this small rank root system is probably not enough. Rather, one should consider the
entire sum of the root spaces; by \cite[Theorem 1]{Bor91},
up to a central extension which may split up imaginary root spaces, this will be the root system of a Borcherds algebra (of possibly infinite rank). 
Nonetheless, one could hope to define MV polytopes for this algebra, and that
the 2-faces could be matched to these. Unfortunately, even if this
were possible, ``reduction to rank two" would mean
reduction to a Borcherds algebra of possibly infinite rank, leaving it
debatable whether this actually improves matters; it still may shed light on the structure of KLR algebras and their representations.  

In any case, there will certainly be new difficulties beyond affine type. To illustrate some of these, consider the Cartan matrix 
\begin{equation}
\left(
\begin{array}{ccc}
2&-2&-2\\
-2&2&-2\\
-2&-2&2
\end{array}
\right).
\end{equation}
This is of hyperbolic type, and the imaginary root $\beta= \alpha_1+\alpha_2+\alpha_3$ has multiplicity 2. 
Fix a charge $c$ with $c(\alpha_0)=1+i, c(\al_1)=-1+i, c(\al_2)=i$.
The only real root with $c(\al)\in i\R$ is $\al_2$ itself, so real roots only generates a copy of $\mathfrak{sl}_2$, but the intersection is rank 2, since it contains $\beta$. This is already a new phenomenon as in finite and affine type the real roots corresponding to a 2-face always generated a rank 2 root system. 

Nonetheless,
Proposition \ref{e-f-defined} 
shows that
the semi-cuspidals of
argument $\pi/2$ are a combinatorial bicrystal for $\mathfrak{sl}_2$.
If the naive analogue of Corollary \ref{cor:aff-lowest} held, then we would
have that $\te_2$ and $\te_2^*$ act identically on every
semi-cuspidal  of argument $\pi/2$, since this is the case in
$B^{\mathfrak{sl}_2}(-\infty)$.  However, both $\te_2 \te_1 \te_0 \scrL_\emptyset$ and $
 \te_2^*\te_1 \te_0\scrL_\emptyset$ are 1-dimensional; the former
 has character $w[012]$ and the latter $w[201]$.  Thus, they are
 necessarily distinct. 

Attacking this case will require stronger techniques than we possess
at the moment. For instance, the sharp-eyed reader will note that we give no direct
connection between the KLR algebra attached to a face and the lower
rank KLR algebra for the root system spanned by that face.  While this
seems like an obvious suggestion, we see no such connection (say, a
functor) at the moment.  Perhaps more progress can be made  
if such a functor can be found.

\bibliography{./gen}
\bibliographystyle{amsalpha}
\end{document}